\documentclass[12pt]{iopart}

\usepackage{subcaption}
\usepackage{capt-of}
\usepackage{booktabs}
\usepackage{varwidth}
\usepackage[export]{adjustbox}
\usepackage{tikz}
\DeclareRobustCommand\full  {\tikz[baseline=-0.6ex]\draw[thick] (0,0)--(0.5,0);}
\DeclareRobustCommand\dotted{\tikz[baseline=-0.6ex]\draw[thick,dotted] (0,0)--(0.54,0);}
\DeclareRobustCommand\dashed{\tikz[baseline=-0.6ex]\draw[thick,dashed] (0,0)--(0.54,0);}

\DeclareRobustCommand\chainn{\tikz[baseline=-0.6ex]\draw[thick,dash dot] (0,0)--(0.5,0);}
\DeclareRobustCommand\ch{\tikz[baseline=-0.6ex]\draw[thick,dashed,line cap=square] (0,0)--(0.5,0);}
\usepackage{tikz}
\newlength\replength

\usepackage{amssymb}
\usepackage{bbm}
\usepackage[utf8]{inputenc}  
\newenvironment{proof}{\noindent\textbf{Proof.} }{\hfill$\square$\par}
\makeatletter
\def\@begintheorem#1#2{\trivlist
   \item[\hskip \labelsep{\bf #1\ #2.}]} 
\def\@opargbegintheorem#1#2#3{\trivlist
   \item[\hskip \labelsep{\bf #1\ #2\ (#3).}]} 
\makeatother

\newtheorem{remark}{Remark}
\makeatletter
\def\@begintheorem#1#2{\trivlist
   \item[\hskip \labelsep{\bf #1\ #2.}]} 
\def\@opargbegintheorem#1#2#3{\trivlist
   \item[\hskip \labelsep{\bf #1\ #2\ (#3).}]} 
\makeatother
\newtheorem{lemma}{Lemma}
\makeatletter
\def\@begintheorem#1#2{\trivlist
   \item[\hskip \labelsep{\bf #1\ #2.}]} 
\def\@opargbegintheorem#1#2#3{\trivlist
   \item[\hskip \labelsep{\bf #1\ #2\ (#3).}]} 
\makeatother
\newtheorem{definition}{Definition}
\makeatletter
\def\@begintheorem#1#2{\trivlist
   \item[\hskip \labelsep{\bf #1\ #2.}]} 
\def\@opargbegintheorem#1#2#3{\trivlist
   \item[\hskip \labelsep{\bf #1\ #2\ (#3).}]} 
\makeatother
\newtheorem{theorem}{Theorem}
\makeatletter
\def\@begintheorem#1#2{\trivlist
   \item[\hskip \labelsep{\bf #1\ #2.}]} 
\def\@opargbegintheorem#1#2#3{\trivlist
   \item[\hskip \labelsep{\bf #1\ #2\ (#3).}]} 
\makeatother
\newtheorem{corollary}{Corollary}
\usepackage{enumerate}
\usepackage[shortlabels]{enumitem}

\usepackage{color}  
\definecolor{darkred}{rgb}{0.6,0,0}
\definecolor{darkgreen}{rgb}{0,0.5,0}
\definecolor{darkmagenta}{rgb}{0.5,0,0.5}
\usepackage[pdftex, colorlinks=true,
  linkcolor=blue,
  citecolor=darkgreen,
  urlcolor=darkmagenta]{hyperref}

\newcommand{\R}{\mathbb R}

\newcommand{\dott}{\, \cdot\,}

\expandafter\let\csname equation*\endcsname\relax
\expandafter\let\csname endequation*\endcsname\relax
\usepackage{amsmath}
\numberwithin{equation}{section}  
\usepackage[sort]{cite}
\usepackage{bigints}
\newcommand{\Z}{\mathbb{Z}}
\DeclareMathOperator{\lip}{Lip}
\DeclareMathOperator{\BV}{BV}
\newcommand{\sgn}{\mathop\mathrm{sgn}}
\newcommand{\norma}[1]{{\left\|#1\right\|}}

\renewcommand{\d}[1]{\mathinner{\mathrm{d}{#1}}}
\DeclareMathOperator{\TV}{TV}
\newcommand{\D}{\Delta}
\newcommand{\abs}[1]{{\left|#1\right|}}

\begin{document}

\title[Nonlocal balance laws]{Error Estimates for Systems of Nonlocal Balance Laws Modeling Dense Multilane Vehicular Traffic}
\author{Aekta Aggarwal$^1$, Helge Holden$^2$, 
Ganesh Vaidya$^3$}
\address{$^1$ Operations Management and Quantitative Techniques, Indian Institute of Management, Prabandh Shikhar, Rau--Pithampur Road, Indore, Madhya Pradesh, 453556, India}
\address{$^2$ Department of Mathematical Sciences,
    NTNU  Norwegian University of Science and Technology,
     NO--7491 Trondheim, Norway}
     \address{$^3$ Department of Mathematics,    Eberly College of Science, Penn State University, University Park, PA 16802, USA.}
\ead{aektaaggarwal@iimidr.ac.in,helge.holden@ntnu.no,gmv5228@psu.edu }
\begin{abstract} 
We discuss a class of coupled systems of nonlocal nonlinear balance laws modeling multilane traffic, with the nonlocality present in both convective and source terms.  The uniqueness and existence of the entropy solution are proven via doubling of the variables arguments and convergent finite volume approximations, respectively. The primary goal is to establish that the finite volume numerical approximations of the system converge to the unique entropy solution at a rate of 
$\sqrt{\Delta t}$, even when using relatively less regular one-sided kernels, compared to the globally smooth kernels analyzed in [Num. Math., 156(1):237–271, 2024] and {\color{black}[IMA J. Numer. Anal., 44(6):3354-3392, 2024].}
The applicability of the proven theory to a general class of systems of nonlocal balance laws coupled strongly through the convective part and weakly through the source part, is indicated. 
{\color{black}As the support of the kernel tends to zero, the convergence of the entropy solutions of the proposed model to its local counterparts [SIAM J. Math. Anal., 51: 3694--3713, 2019] is also discussed. Numerical simulations illustrating the behavior of the entropy solutions of the coupled nonlocal systems are also shown.}
\end{abstract}
\section{Introduction}\label{sec:intro}
The literature on ``local'' conservation laws is very rich and describes various multidisciplinary phenomena such as sedimentation, blood flow, traffic dynamics, etc. 
However, recent advancements have extended these equations to ``nonlocal" conservation laws of the type 
\begin{align}\label{NL_k}
u_t+\partial_x(f(x,u)\nu(u*\omega_{\eta}))=0,
\end{align}
capturing more complex interactions and physically relevant behavior across the  spatial domain. Here, the solution $u$ (density in the case of traffic, for example)  at any point depends on its average in its neighborhood through the term $u*\omega_{\eta}$, defined as the usual convolution of $u$ and $\omega_{\eta}$. Some recent applications with smooth $f$ include crowds, traffic, opinion formation, sedimentation, multi-commodity networks, supply chain, granular material dynamics, and conveyor belt dynamics. 
The theoretical and numerical analysis of this class has been well studied, a non-exhaustive list being \cite{CG2019, GHS+2014, KLS2018, CMR2015,BBKT2011,ANT2007,AS2012,CHM2011,BHL2023, ACT2015, ACG2015,AG2016,BG2016,FGKP2022,CGL2012,CL2011,AHV2023}, and various references therein. The case of $f$ with spatial discontinuities is useful for modeling situations with rough roads, for example, and has been handled in \cite{AV2023,FCV2023,KP2021,AHV2023}.
{\color{black}The lane changing phenomenon of multilane vehicular traffic was modeled using macroscopic models} using a weakly coupled system of ``local" hyperbolic conservation laws in {\color{black}\cite{HR2019,chiarello2019stability}, by including a source term governed by the flux difference between adjacent lanes}. In this article, we consider a \textit{nonlocal} counterpart, which is a coupled system of equations $N$ to model traffic in $N$ lanes, {\color{black}coupled through both local and nonlocal terms in the source and only nonlocal terms in the convective part}. More precisely, we consider a system of $N$ partial differential equations (PDEs), where the $k^{\rm th}$ equation  is given by
\begin{align}\label{PDE_NL_k}
\partial_t(u^k)+\partial_x(f^k(u^k)\nu^k({x, } u^k*\omega_{\eta}))=R^k ({x, }\boldsymbol{u}, \boldsymbol{u} \circledast \omega_{\eta}),\, \, k\in \mathcal{N}:=\{1, \ldots, N\},
\end{align}
where, {\color{black}for each $k\in \mathcal{N}$}, the following assumptions hold:
\begin{enumerate}[label=(\textbf{A\arabic*})]
	\item \label{A1} {\color{black}$f^k(u)=ug^k(u)$}, $g^k\in \lip(\R), g^k(1)=0$, and $g^k$ is non-increasing;
 \item \label{A2}  $\nu^k \in L^{\infty}(\R\times[0,1])$, $\nu^k(\dott,u) \in C^2(\R)$ for $u\in \R$ and $\nu^k(x,\dott) \in C^2(\R) \cap W^{2,\infty}(\R)$ for $x\in \R$;
 \item \label{A3}  {\color{black}
  There exist 
 a constant $\mathcal{C}_0$,
  and a $\BV$ function $a_0:\R\rightarrow\R$ such that 
\begin{align}
   \label{a0_BV}  \abs{\nu^k(x,u)-\nu^k(y,u) } & \leq \mathcal{C}_0 \abs {a_0(x)-a_0(y)},\,\, \text{for all }u\in\R;
\end{align}
   \item \label{A4}For every {\color{black}$\eta\in(0,\infty]$}, let  $\omega_{\eta} \in (C^2 \cap   \BV  \cap \, W^{2,\infty}) ([0,\eta])$ be non-negative and non-increasing, with
}
 \end{enumerate}
\begin{align}\label{con}
    (u^k *\; \omega_{\eta}) (t,x)=
  \displaystyle\displaystyle \int_{x}^{x+\eta}\omega_{\eta} (y-x) \; u^k(t,y) \, \d y,  \quad {\color{black}(t,x)\in Q_T:=(0,T)\times \R},\end{align} for a final time $T>0.$ {\color{black} Assuming that $\omega_\eta(\eta)=0$, the monotonicity of $\omega_{\eta}$ implies that $\norma{\omega_{\eta}}_{L^{\infty}(\R)}=|\omega_{\eta}|_{\BV(\R)}$. }
  Further, the symbol $\circledast$ denotes the component-wise convolution defined by $\boldsymbol{u}\circledast \omega_{\eta}=(u^k*\omega_{\eta})_{k\in\mathcal{N}}.$ {\color{black}While the assumption \ref{A1} is crucial for establishing the invariant region principle of the solution,  
  \ref{A2}--\ref{A3} are essential for establishing various stability bounds on solutions and integrals in the manuscript. The assumption \ref{A4} ensures the choice of kernel being physically relevant with respect to drivers choosing their speed based on the downstream
car density only (i.e. they look forward, not backward)}. 
Also, the source term $R^k$, a locally Lipschitz function in each of the variables,
denotes the net flow into  lane $k$ from its neighboring lanes.
  For each $k\in\mathcal{N}$, the source term $R^k$ is defined by
\begin{align}\label{eq:Rk}
\begin{split}
R^k({x,}\boldsymbol{u},\boldsymbol{u} \circledast \omega_{\eta})&:=S^{k-1}({x,}u^{k-1},u^k,u^{k-1}*\omega_{\eta},u^k*\omega_{\eta})\\&\quad-S^k({x,}u^k,u^{k+1},u^k*\omega_{\eta},u^{k+1}*\omega_{\eta}),
\end{split}
\end{align}
with $S^0({x,}a,b,A,B)=S^{N}({x,}a,b,A,B)=0,$ and  for $k\in{\color{black}\mathcal{N}\setminus\{N\}}$,
\begin{align}\label{eq:Sk}
 \begin{split}
S^k({x,}a,b,A,B)&:= \big(g^{k+1}(b)\nu^{k+1}({x, }B)-g^k(a)\nu^k({x, }A)\big)^+a\\ & \quad -\big(g^{k+1}(b)\nu^{k+1}({x, }B)-g^k(a)\nu^k({x, }A)\big)^-b,
\end{split}
\end{align}
where $a^+=\max(a,0)$ and $ a^{-}=-\min(a,0)$. 
 We denote $\boldsymbol{S}:=(S^1,\ldots,S^{N}), \boldsymbol{f}:=(f^1,\ldots,f^{N})$, and $\boldsymbol{\nu}:=(\nu^1,\ldots,\nu^{N})$.
It can be seen that the function $S^k$ is locally Lipschitz in each of its state variables and $\BV$ in space variable for each $k\in\mathcal{N}$, i.e., for any compact set  $\Omega \subset \R^4$,we have: 
\begin{align}\label{est:Lip_S}\begin{split}
    &\abs{S^k({x,}a,b,A,B)-S^k({x,}\tilde{a},\tilde{b},\tilde{A},\tilde{B})} \\&\qquad\leq \abs{S^k}_{L^{\infty}(\R;\lip(\Omega))} (|a-\tilde{a}|+|b-\tilde{b}|+|A-\tilde{A}|+|B-\tilde{B}|).
 \end{split} 
 \end{align}
 In the context {\color{black}of} multilane traffic flow, $S^k$ denotes the vehicles that \textit{leave} lane {\color{black}$k$} to enter lane {\color{black}$k+1$}.
  Furthermore, $S^0=S^{N}=0$ indicates that vehicles are not allowed to cross the boundary lanes. In addition, the velocity $\nu^k$ in the convolution term and the source term are spatially dependent to incorporate changing road conditions. Note that in contrast to the local counterpart of this model considered in \cite{HR2019}, here each density $u^k$ convects and transfers according to its average density defined by \eqref{con}. {\color{black}In} contrast to \cite{HR2019}, the decision of lane change in $k$ defined by $R^k$ is affected not only by neighboring densities $u^{k\pm 1}, u^k$ but also by their respective averages. {\color{black} Similar source terms (but independent of space variables) have also been considered in \cite{FGR2021,BFK2022}. The article \cite{CK24} also deals with spatially dependent source terms and multilane models, but  with $g^k(u)=1$ and $\omega_{\eta}$ as the decreasing exponential kernel.  Further, \cite{CNAP2022,CGL2024} deal with {\color{black}$g^k(u)=1=N$ and more general $\BV$ kernels. While the former deals with one dimension and uses fixed point theory for well-posedness, the latter establishes local existence results and corresponding finite blow-up in multiple dimensions.}
  
One of the objectives of the paper is the well-posedness of \eqref{PDE_NL_k} on $Q_T$ with the initial data
\begin{align}
\label{eq:u11A}
  u^{k}(0,x)&=u_0^{k}(x), \quad x \in \R.
\end{align} {\color{black} with $u_0\in ((L^1 \cap \BV) (\R;[0,1]))^{N}$.
In this paper we will for simplicity assume that the maximum density equals unity (cf.~the assumption that $ g^k(1)=0$), and thus we assume that the initial data take values in the interval $[0,1]$, and we will show that the solution remains in $[0,1]$, see Lemma \ref{lem:stability}.}

In Section \ref{uni}, we establish the uniqueness of the \textit{entropy} solution (cf.~\eqref{kruz2}), and the approximate solutions obtained from the finite volume schemes proposed in Section~\ref{exis}, are shown to converge to this entropy solution. The main objective of this article is the error analysis of our finite volume approximations, proposed in Section \ref{exis}.
This {\color{black}was} achieved by following the techniques of~\cite{KUZ1976}, where the so-called doubling of variables argument is used to express the error between the entropy solution $u$ (that is, $N = 1$) and its numerical approximation $u_{\Delta}$ in terms of their relative entropy. 

It should be noted that, {\color{black}the analysis presented in this paper differs from the existing literature \cite{AHV2023,AHV2023_1,ACG2015,ACT2015,BBKT2011}. These articles studied \eqref{PDE_NL_k} with $R^k\equiv 0$ with smooth global $\omega_{\eta}$. In particular, \cite{AHV2023_1,ACG2015} dealt with $N\ge 1$, where \cite{AHV2023_1} handled the well-posedness and error estimates for spatially discontinuous $f^k,$ while \cite{ACG2015} established the existence of entropy solutions in multi-dimensions with smooth $f^k$. Further,  \cite{ACT2015,BBKT2011,AHV2023} dealt with $N=1$ with \cite{AHV2023} dealing with well-posedness and error estimates with spatially discontinuous $f^k$ and multiple dimensions, and \cite{BBKT2011,ACT2015} studying the well-posedness with smooth $f^k$. The present article handles} less regular, though more {\color{black}realistic}, one sided  kernels and in the presence of the source term. Additionally, the results presented here are more general as compared to those in previous studies of nonlocal multilane models\cite{CG2019, BFK2022,CK24,CCV2022,FGR2021}. Three key differences set this work apart: first, the non-linearity of $f^k$,
second, the spatial dependence of 
$\nu^k$
  and 
$R^k$, which depicts changing road conditions; and third, the rigorous error analysis of the proposed finite volume schemes. Further, the article presents the analysis for the same choice of kernel for both the convective and source terms.

{\color{black}Another interesting question the article explores is the nonlocal-to-local limit of \eqref{PDE_NL_k} to its local counterpart proposed in \cite{HR2019} as the support of kernel $\omega_{\eta}$ goes to zero. We consider the particular case of $g^k(u)=1$ and with spatially independent $\nu^k$ and $R^k$, as the current literature on this topic, even with $N=1,$ with or without source term, is available only for this case. The local counterpart reads as:
\begin{align}\label{PDE_NL_k_l_1}
\partial_t(u^k)+\partial_x(u^k\nu^k( u^k))=R^k (\boldsymbol{u}, \boldsymbol{u}),\, \, k\in \mathcal{N},
\end{align} {\color{black} with  $R^k$ as in \eqref{eq:Rk}--\eqref{eq:Sk}.} The analysis of nonlocal to local convergence, has attracted significant attention recently, with various results, see, for example, \cite{Coc23a,Col23a,CCCNKMKL2023,CCMS2024,CCS2019,CCMS2023,CGES2021} with $R^k=0$ and $N=1.$ The article \cite{CK24} also works with multilane models, like this article and handles $\eqref{PDE_NL_k_l_1},$ with nonzero $R^k,$ but with a decreasing exponential kernel as a special choice for $\omega_{\eta}$. It will be shown in Section \ref{NLL} that the entropy solutions of \eqref{PDE_NL_k} converge to those of the local model \eqref{PDE_NL_k_l_1}, in the limit as the support $\eta$ of $\omega_\eta$ tends to zero using a source-splitting technique for a wider class of $\omega_{\eta}$.}

{In general, practical problems like the multi-agent systems, as well as laser technology, dynamics of multiple crowds/vehicles, lane formation, sedimentation and conveyor belts, see, for example, \cite{CMR2016, GKLW2016,CM2015} are modeled by a very general coupled system of balance laws for any $T>0$, where the $k^{\text{th}}$ equation is given by 
\begin{align}\label{PDE_NL_kd} 
\partial_t(u^k)+ \partial_{x}(f^{k}(u^k)\nu^{k}({x,}(\boldsymbol{\omega}_{\eta}\circledast \boldsymbol{u})^k))&=R^k ({x,}\boldsymbol{u},(\boldsymbol{\tilde{\omega}}_{\tilde{\eta}} \circledast {\boldsymbol{u}})^{k}), \qquad k\in \mathcal{N}.
\end{align}
For every $\eta,\tilde{\eta}>0$, the initial-value problem (IVP) \eqref{PDE_NL_kd}, \eqref{eq:u11A} is strongly coupled due
to the nonlocal
terms $\boldsymbol{\omega}_{\eta} \circledast  \boldsymbol{u}:\R \rightarrow \R^{N^2}$
in the convective part and also weakly coupled through $\boldsymbol{\tilde{\omega}}_{\tilde{\eta}} \circledast \boldsymbol{u}:\R \rightarrow \R^{N^2}$ in the source term where $\boldsymbol{\omega}_{\eta},\boldsymbol{\tilde{\omega}}_{\tilde{\eta}}$ are $N\times N$ matrices. For $\boldsymbol{z}=\boldsymbol{\omega}_{\eta},\boldsymbol{\tilde{\omega}}_{\tilde{\eta}}$, and for each $k \in\mathcal{N}$, 
$
(\boldsymbol{z}\circledast\boldsymbol{u})^k:= \left(z^{k,1}*u^1,\ldots,z^{k,N}*u^{N}\right).$
The analysis presented in this paper can be generalized to IVP \eqref{PDE_NL_kd}, \eqref{eq:u11A} with the following generalization of the hypothesis \ref{A1}--\ref{A4}, 
{\color{black}\begin{enumerate}[label=(\textbf{B\arabic*})]
	\item \label{B1} $f^k\in \lip(\R), f^k(1)=0=f^k(0)$;
\item \label{B2}  $\nu^k \in L^{\infty}(\R\times[0,1]^N)$, $\nu^k(\dott,\boldsymbol{u}) \in C^2(\R)$ for $\boldsymbol{u}\in \R^N$ and $\nu^k(x,\dott) \in (C^2(\R) \cap W^{2,\infty}(\R))^N$ for $x\in \R$;
 \item \label{B3}  
  There exist a constant $\mathcal{C}_0$ and a $\BV$ function $a_0:\R\rightarrow\R$ such that 
  \begin{align*}
\abs{\nu^k(x,\boldsymbol{u})-\nu^k(y,\boldsymbol{u}) } & \leq \mathcal{C}_0 \abs {a_0(x)-a_0(y)},\,\,\,\text{for all }\boldsymbol{u}\in\R^N;
\end{align*}
   \item \label{B4}For every {\color{black}$\eta\in(0,\infty]$} and $\boldsymbol{\omega}_{\eta}\in {(C^2 \cap   \BV  \cap \, W^{2,\infty}) ([0,\eta],\R^{N^2})}$, and for every {\color{black}$\tilde{\eta}\in(0,\infty]$}, $\boldsymbol{\tilde{\omega}}_{{\tilde{\eta}}} \in {(C^2 \cap   \BV  \cap \, W^{2,\infty}) ([0,\tilde{\eta}],\R^{N^2})}$; with $\omega_{\eta}^{k,j},\tilde{\omega}_{\tilde{\eta}}^{k,j}$ being non-negative and non-increasing for all $k,j\in\mathcal{N};$
   \item \label{B5} For each $x\in\R$, $R^k(x,\dott) \in \lip(\R^{2N})$, with $R^k(\dott,\boldsymbol{0},\dott)=R^k(\dott,\boldsymbol{1},\dott)=0$,  {\color{black} and for all $x \in \mathbb{R}$, $\boldsymbol{u}, \boldsymbol{v} \in \mathbb{R}^N$, and for each $j \in \mathcal{N} \setminus \{k\}$, the function $R^k$ is non-decreasing in the $j$th component of the state variable $\boldsymbol{u}$. Thus, for any $a, b \in \mathbb{R}$ with $a\le b$, the following inequality holds:
\begin{align*}
R^k \left( x, (u^1, \dots, u^{j-1}, a, u^{j+1}, \dots, u^N), \boldsymbol{v} \right) \leq R^k \left( x, (u^1, \dots, u^{j-1}, b, u^{j+1}, \dots, u^N), \boldsymbol{v} \right).
\end{align*}
   }
 \end{enumerate}}
In fact, it should be noted that $R^k$ has a very general form with dependence on all $N$ components of $\boldsymbol{u}$ and $(\boldsymbol{\tilde{\omega}}_{\eta} \circledast {\boldsymbol{u}})^{k}$. 
    The assumptions \ref{B1} and \ref{B5} are crucial, and they restrict the solutions to the interval [0,1], circumventing the $L^{\infty}$ blow up in finite time reported in \cite[eq.~(2.3)]{CM2015} and consequently lead to the well-posedness in the time interval $[0,T]$ for any $T>0$, unlike the local well-posedness in time established in \cite[Thm.~2.2]{CM2015}. {\color{black} The remaining assumptions \ref{B2}--\ref{B4} are technical assumptions required for the stability of solutions and are generalizations of \ref{A2}--\ref{A4}.} Though the analysis in this paper is presented only for a special form of $R^k$ suited to nonlocal extension of multilane traffic model of \cite{HR2015}, we give the directions on the extension of our proofs to the general class \eqref{eq:u11A}, \eqref{PDE_NL_kd} in Section \ref{exten}.}

The paper is organized as follows. In Section \ref{def}, we introduce some definitions and notation to be used in the article. In Section \ref{uni}, we prove the uniqueness result for \eqref{PDE_NL_k},\eqref{eq:u11A}. In Section \ref{exis}, we present the convergence of the {\color{black} nonlocal adaptions of the monotone schemes such as Godunov and the Lax--Friedrichs schemes, for} the IVP \eqref{PDE_NL_k},\eqref{eq:u11A}. In Section \ref{sec:error}, {\color{black}we show that  the scheme converges to the unique entropy solution at an optimal rate of 1/2 as seen in the local conservation laws \cite{Sab1997}.} In Section~\ref{exten}, we also briefly comment on extensions to general systems of the type \eqref{PDE_NL_kd}.  Finally, in Section~\ref{num}, we present some numerical experiments that illustrate the theory and {\color{black}the behavior of the entropy solution as the support of the kernel goes to
zero (nonlocal to local limit).}

\section{Definitions and notation}\label{def}
Throughout this paper, for simplicity, we scale $\omega_{\eta}$ by its $L^1$ norm and therefore assume that $\norma{\omega_{\eta}}_{L^1(\R)}=1$ {\color{black}  (extending the function to the whole line by setting it zero outside $[0,\eta]$)}. 
With $u: \overline{Q}_T\rightarrow \R$, $\boldsymbol{u}=(u^k)_{k\in\mathcal{N}}:\overline{Q}_T \rightarrow \R^{N}$ 
  and $\tau>0$, we define:
\begin{align*}
&|u|_{\lip_tL^1_x}:=\sup_{0\leq t_1<t_2\leq T}\frac{{\color{black}\norma{u^k(t_1)-u^k(t_2)}}_{L^1(\R)}}{|t_1-t_2|},& 
&|\boldsymbol{u}|_{(\lip_tL^1_x)^{N}}:=\max_{{k\in\mathcal{N}}}|u^k|_{\lip_tL^1_x}, \\ 
&|\boldsymbol{u}|_{(L^\infty_t\BV_x)^{N}}:=\max_{{k\in\mathcal{N}}}\sup_{t\in[0,T]} \TV(u^k(t,\dott)),& &
\norma{\boldsymbol{u}}_{(L^{\infty}(\overline{Q}_T))^{N}}:=\max_{{k\in\mathcal{N}}} \norma{u^k}_{L^{\infty}(\overline{Q}_T)},\\
&\gamma(u,\tau):=\sup_{\substack{
    \abs{t_1-t_2} \leq \tau\\  0\leq t_1\leq t_2\leq T }} \norma{u^k(t_1)-u^k(t_2)}_{L^1(\R)}, & &\norma{\boldsymbol{u}}_{(L^1(\overline{Q}_T))^{N}}:=\sum_{k\in\mathcal{N}}\norma{u^k}_{L^1(\overline{Q}_T) }
,\\&{\color{black}|u|_{L^{\infty}_t \lip_x}:=\sup_{t\in [0,T]}\sup_{\substack{x,y\in\R\\x\neq y}}\frac{\abs{{\color{black}u^k(t,x)-u^k(t,y)}}}{{\color{black}\abs{x-y}}}}, &&
    \boldsymbol{\gamma}(\boldsymbol{u},\tau):= \max\limits_{k\in\mathcal{N}}\gamma(u^k,\tau).
\end{align*}
{\color{black}We will use the following definition of the entropy condition, based on the Kru\v{z}kov
formulation, see, e.g., \cite{Kruzkov, HR2015}.}\\
{\color{black}\begin{definition}
  [Entropy Condition]\label{def:sol}
Let $u_0\in ((L^1 \cap \BV) (\R;[0,1]))^{N}$.  A function $\boldsymbol{u} \in (\lip([0,T];L^1(\R;[0,1]))\cap L^{\infty}([0,T]; \BV(\R)))^{N}$  is an entropy solution of \eqref{PDE_NL_k} with initial data $u_0$ if, for $k=1,\ldots,N$, 
for each $(k,\alpha)\in \mathcal{N}\times\R$, and for all\,\, 
$0\le \phi\in C_c^{\infty}([0,T)\times \R)$,
 \setlength{\abovedisplayskip}{0.1pt}
\setlength{\belowdisplayskip}{0pt}
\setlength{\abovedisplayshortskip}{0pt}
\setlength{\belowdisplayshortskip}{0pt}
\begin{align}\label{kruz2}
 \begin{split}
&\int_{Q_T}\abs{u^k(t,x)-\alpha} \phi_t(t,x)\d t\d x+ \int_{Q_T}  \mathcal{G}^k(u^k(t,x),\alpha)\mathcal{U}^k(t,x)  \phi_x(t,x) \d t \d x\\
 &-\int_{Q_T} \sgn(u^k(t,x)-\alpha) f^k(\alpha) {\mathcal{U}}_x^k(t,x) \phi(t,x )\d t \d x\\
 &+ \int_{\R} \abs{u^k(0,x)-\alpha}\phi(0,x) \d x  \geq -\int_{Q_T} \sgn(u^k(t,x)-\alpha) {\color{black}\mathcal{R}}^k_{\boldsymbol{u}}(t,x)\phi(t,x) \d t \d x,
 \end{split}
 \end{align}
 where for $a,b\in\R,$
 $ \mathcal{G}^k(a,b):=\sgn (a-b) (f^k(a)-f^k(b)), {\color{black}\mathcal{R}}^k_{\boldsymbol{u}}(t,x):= R^k(x,\boldsymbol{u}(t,x),(\boldsymbol{u} \circledast \omega_{\eta})(t,x)), \mathcal{U}^k(t,x):=\nu^k({x,}({u}^k*\omega_{\eta})(t,x))$, for $(t,x)\in Q_T$.
 \end{definition}}
\section{Kuznetsov-type estimate for nonlocal balance laws and Uniqueness}\label{uni}
 {\color{black} One of the aims of this article is to quantify the error between the approximate and the entropy solution of \eqref{PDE_NL_k},\eqref{eq:u11A} for which we need a Kuznetsov-type lemma(cf. \cite{KUZ1976} for the local case and \cite{AHV2023} for the scalar nonlocal case with $N=1$), which will be established in this section.}
To this end, we now introduce some notation. For $\epsilon,\epsilon_0>0$, define $\Phi: \overline{Q}_T^2 \rightarrow \R$ by, \begin{align*}
\Phi(t,x,s,y):=\Phi^{\epsilon,\epsilon_0}(t,x,s,y)=\Theta_{\epsilon}(x-y)\Theta_{{\epsilon}_0}(t-s), \quad {\color{black}(t,x),(s,y)\in \overline{Q}_T},
\end{align*}
where $\Theta_a(x)=\frac{1}{a}\Theta\left(\frac{x}{a}\right)$, $a>0$ and $\Theta$ is a standard symmetric mollifier with $\operatorname{supp} (\Theta) \in [-1,1]$. Furthermore, {\color{black}it is assumed that} the mollifiers satisfy \begin{align}\label{mol}\int_{\R} \Theta_a(x) \d x =1 \text{ and } \int_{\R} \abs{\Theta'_a(x)} \d x =1/a.
\end{align}It is straightforward to see that $\Phi$ is symmetric,
$\Phi_x=-\Phi_y$ and $\Phi_t=-\Phi_s$. 
For $\boldsymbol{u},\boldsymbol{v} \in ((L^1 \cap L^{\infty})(\R))^{N}, 0\le\phi\in C_c^{\infty}({\overline{Q}}_T)$, $(t,x),(s,y)\in Q_T$, $a,b\in\R$, and
$k\in\mathcal{N}$, we define the following:
\begin{align*}
&
 {\color{black}\mathcal{R}}^k_{\boldsymbol{v}}(t,x):= R^k({x,}\boldsymbol{v}(t,x),\boldsymbol{v} \circledast \omega_{\eta})(t,x),\\
&\mathcal{V}^k(t,x):=\nu^k({x,}(v^k*\omega_{\eta})(t,x)), &
 &\mathcal{G}^k(a,b):=\sgn (a-b) (f^k(a)-f^k(b)),
 \end{align*}
 {\allowdisplaybreaks
 \begin{align*}
\Lambda^k_T\left(u^k,\phi, \alpha\right)&:= \int_{Q_T}\left(\left|u^k(t,x)- \alpha\right|\phi_{t}+\mathcal{G}^k\left(u^k(t,x), \alpha\right)\mathcal{U}^k(t,x)\phi_{x}\right)\d t \d x\\
    &\quad -   \int_{Q_T}\sgn \left(u^k(t,x)- \alpha \right) f^k(\alpha)\mathcal{U}^k_x(t,x)\phi \d t \d x\\
    &\quad+\int_{Q_T} \sgn(u^k(t,x)-\alpha) {\color{black}\mathcal{R}}_{\boldsymbol{u}}^k(t,x) \phi \d t \d x\\
     & \quad -\int_{\R}\left|u^k(T,x)- \alpha\right|\phi(T,x)\d x+\int_{\R}\left|u_0^k(x)- \alpha\right|\phi(0,x)\d x,\\ 
\Lambda^k_{\epsilon,\epsilon_0}(u^k,v^k)&:=\int_{Q_T}\Lambda^k_T\left(u^k(\dott,\dott),\Phi(\dott,\dott,s,y),v^k(s,y)\right)\d s \d y, \\      
   K&:=\Big\{\boldsymbol{u}:\overline{Q}_T \rightarrow [0,1]^{N} \mid \norma{\boldsymbol{u}}_{(L^{\infty}(\overline{Q}_T))^{N}}+|\boldsymbol{u}|_{(L^{\infty}([0,T];\BV(\R)))^{N}} < \infty 
   \\&\qquad \text{ and }\norma{\boldsymbol{u}(t)}_{(L^1(\R))^{N}}=\norma{\boldsymbol{u}(0)}_{(L^1(\R))^{N}} \text{ for }t\ge 0 \Big\}.
   \end{align*}}
   In the sequel, we call $\Lambda^k_{\epsilon,\epsilon_0} (u^{k}_{\Delta},u^k)$ as the \textit{relative entropy functional}.}
   {\color{black}The following lemma compares an entropy solution of the IVP \eqref{PDE_NL_k},\eqref{eq:u11A} with an arbitrary function from the set $K$.\\
\begin{lemma}[A Kuznetsov-type lemma for  systems of nonlocal balance laws]\label{lemma:kuz}
Let $\boldsymbol{u}$ be an entropy solution of the IVP \eqref{PDE_NL_k},\eqref{eq:u11A}, $\boldsymbol{v} \in K$ and {\color{black}$\boldsymbol{v}_0(x):=\boldsymbol{v}(x,0)$ for $x\in\R$} and {\color{black}$T>0$}. Then, 
\begin{align}
 \label{est:kuz}
\begin{split}
\norma{\boldsymbol{u}(T,\dott)-\boldsymbol{v}(T,\dott)}_{(L^1(\R))^{N}}  &\leq {\mathcal{C}_T} \bigg( \sum_{k\in\mathcal{N}}\big( \gamma(v^k,\epsilon_0)-\Lambda^k_{\epsilon,\epsilon_0}(v^k,u^k) \big)  \\&\quad
+\left( \norma{\boldsymbol{u}_0-\boldsymbol{v}_0}_{(L^1(\R))^{N}}+N\big(\epsilon+\epsilon_0 \big) \right) \bigg),
\end{split}
\end{align}
where $N$ is the number of lanes and the constant $\mathcal{C}_T$ depends on {\color{black}$\norma{\boldsymbol{v}(t,\dott)}_{(L^1(\R))^N}$, $|\boldsymbol{v}|_{(L^{\infty}_t\BV_x)^{N}}$, 
$\norma{\boldsymbol{v}}_{(L^1(\overline{Q}_T))^N},\norma{\boldsymbol{v}}_{(L^{\infty}({Q}_T))^N},\norma{\boldsymbol{u}_0}_{(L^1({\R}^N)},|\boldsymbol{u}|_{
(L^{\infty}_t\BV_x)^{N}}, |\boldsymbol{u}|_{(\lip_tL^1_x)^{N}},\norma{\omega_{\eta}}_{L^{\infty}(\R)}$, $\norma{\omega_{\eta}}_{L^1(\R)}
,\abs{\omega_{\eta}}_{\lip([0,\eta])}$, $\norma{D_1\nu^k}_{L^\infty(\R\times [0,1])}, \norma{D_2\nu^k}_{L^\infty(\R\times [0,1])},\norma{D_{11}\nu^k}_{L^\infty(\R\times [0,1])}$,  \\ $\norma{D_{12}\nu^k}_{L^\infty(\R\times [0,1])}
,\norma{D_{21}\nu^k}_{L^\infty(\R\times [0,1])},$ 
$\norma{D_{22}\nu^k}_{L^\infty(\R\times [0,1])},\abs{f^k}_{\lip(\R)},\abs{\boldsymbol{S}}_{(L^{\infty}(\R;\lip({[0,1]^4})))^N}$, $\abs{\boldsymbol{\nu}}_{(\lip(\R))^{N}},N, T, \mathcal{C}_0$ and $\abs{a_{0}}_{\BV(\R)}$}, but is independent of $\epsilon,\epsilon_0$.
\end{lemma}
 \begin{proof} {\color{black}Let $(t,x),(s,y)\in \overline{Q}_T$.}  For any $k\in\mathcal{N}$, 
 add the {\color{black}{relative}} entropy functionals $\Lambda^k_{\epsilon,\epsilon_0}(u^k,v^k)$ and $\Lambda^k_{\epsilon,\epsilon_0}(v^k,u^k)$, and invoke the symmetry of $\Phi$ to get:
\begin{align*}
\Lambda^k_{\epsilon,\epsilon_0}(u^k,v^k)+\Lambda^k_{\epsilon,\epsilon_0}(v^k,u^k)&=
I^k_{\Phi'}+I^k_{\Phi}+I^k_S+I^k_0-I^k_T,
\end{align*}  
with {\allowdisplaybreaks
\begin{align*}
I^k_{\Phi'}&=\int_{Q_T^2}\mathcal{G}^k(u^k(t,x),v^k(s,y))(\mathcal{U}^k(t,x)-\mathcal{V}^k(s,y)) \Phi_{x}(t,x,s,y)\d t \d x \d s \d y,\\[2mm]
I^k_{\Phi}&=-\int_{Q^2_T} \sgn (u^k(t,x)-v^k(s,y))f^k(v^k(s,y))  \mathcal{U}^k_x(t,x)\Phi(t,x,s,y) \d x \d t  \d y \d s\\
&\quad +\int_{Q^2_T}\sgn (u^k(t,x)-v^k(s,y))f^k(u^k(t,x))  \mathcal{V}^k_y(s,y)\Phi(t,x,s,y) \d x \d t  \d y \d s,\\[2mm]
I^k_S&=-\int_{Q^2_T} \sgn (u^k(t,x)-v^k(s,y)) \left({\color{black}\mathcal{R}}_{\boldsymbol{v}}^k(s,y)-{\color{black}\mathcal{R}}_{\boldsymbol{u}}^k(t,x)\right) \Phi(t,x,s,y) \d x \d t  \d y \d s ,\\[2mm]
I^k_T&=\int_{Q_T}\int_{\R}\left|u^k(T,x)-v^k(t,y)\right|\Phi(t,x,T,y)\d y  \d x \d t \\&\quad+\int_{Q_T}\int_{\R}\left|v^k(T,y)-u^k(t,x)\right|\Phi(t,x,T,y)\d y  \d x \d t,\\[2mm]
I^k_0&=\int_{Q_T}\int_{\R}\left(\left|u^k_0(x)-v^k(t,y)\right|+\left|v^k_0(y)-u^k(t,x)\right|\right)\Phi(t,x,0,y)\d x \d y \d t.
 \end{align*}}
 Since $\boldsymbol{u}$ is the entropy solution of the IVP \eqref{PDE_NL_k},\eqref{eq:u11A}, we have that $\Lambda^k_{\epsilon,\epsilon_0}(u^k,v^k)\ge 0, $ and hence
 \begin{align}
 \label{est:I_T^k}
I^k_T &\le-  \Lambda^k_{\epsilon,\epsilon_0}(v^k,u^k) + I^k_{\Phi'}+I^k_{\Phi}+I^k_S+I^k_0.
\end{align}The terms $I^k_0$ and $I^k_T$  can be estimated as in \cite[Thm.~3.14]{HR2015} (see also \cite[Lemma~2]{GTV2022}) to get:
{\begin{align}\label{est:I_0,T}\begin{split}
   I^k_T&\ge \norma{u^k(T,\dott)-v^k(T,\dott)}_{L^1(\R)}-\mathcal{C}_1(\epsilon+\epsilon_0+\gamma(v^k,\epsilon_0)),\\[2mm]
   I^k_0&\le \norma{v^k_0-u^k_0}_{L^1(\R)}+\mathcal{C}_1(\epsilon+\epsilon_0+\gamma(v^k,\epsilon_0)),
   \end{split}
\end{align}}
where $\mathcal{C}_1=3 \max \{|\boldsymbol{u}|_{
(L^{\infty}_t\BV_x)^{N}}, |\boldsymbol{v}|_{(L^{\infty}_t\BV_x)^{N}},|\boldsymbol{u}|_{(\lip_tL^1_x)^{N}}\}$.

{\color{black}The proof of the lemma follows by showing the following two estimates:
\begin{align}
I_{\Phi}^k+I_{\Phi'}^k&\label{est:PP'}\le\mathcal{C}_{\Phi\Phi'}(\epsilon +\epsilon_0 + \int_0^T \norma{{\boldsymbol{u}(s,\dott)-\boldsymbol{v}(s,\dott)}}_{(L^1(\R))^{N}} \d s),   \\
I_S^k&\label{est:I_S}\le\mathcal{C}_{S}(\epsilon +\epsilon_0 + \int_0^T \norma{{\boldsymbol{u}(s,\dott)-\boldsymbol{v}(s,\dott)}}_{(L^1(\R))^{N}} \d s),  
\end{align}
where the constants $\mathcal{C}_{\Phi\Phi'}$ and $\mathcal{C}_{S}$ are appropriate constants defined later in the proof and  may depend on $\boldsymbol{u},\boldsymbol{v},\boldsymbol{f},\omega_{\eta}, \boldsymbol{\nu},T$ but would be independent of $\epsilon$ and $\epsilon_0$. 
Now, we prove \eqref{est:PP'}--\eqref{est:I_S} one by one. \\
\textbf{Proof of \eqref{est:PP'}}: We follow closely the proof of \cite[Lemma~3.2]{AHV2023}.
 Using integration by parts,  $I^k_{\Phi'}$ can be written as,
 \begin{align*}
I^k_{\Phi'}&=-\int_{Q^2_T}\Phi\big[\mathcal{G}^k_x(u^k(t,x),v^k(s,y))(\mathcal{U}^k(t,x)-\mathcal{V}^k(s,y))\\&\qquad \qquad  \quad +\mathcal{G}^k(u^k(t,x),v^k(s,y))\mathcal{U}^{k}_x(t,x) \big]\d t \d x \d y \d s.
\end{align*}
Consequently,
\begin{align}\nonumber
&I^k_{\Phi'}+I^k_{\Phi}\\ \nonumber &=-\int_{Q^2_T}\Phi \big(\mathcal{G}^k_x(u^k(t,x),v^k(s,y))(\mathcal{U}^k(t,x)-\mathcal{V}^k(s,y))\\& \nonumber \qquad \qquad +\sgn(u^k(t,x)-v^k(s,y))(f^k(u^k(t,x))-f(v^k(s,y)))\mathcal{U}^{k}_x(t,x) \big)\d t \d x \d y \d s  \\ \nonumber
&\quad-\int_{Q^2_T}\sgn (u^k(t,x)(t,x)-v^k(s,y)) \Big(f^k(v^k(s,y))  \mathcal{U}^{k}_x(t,x)-f^k(u^k(t,x))\mathcal{V}^k_y(s,y)\Big)\\ \nonumber
&\qquad \qquad \times \Phi\d x \d t \d y \d s \nonumber \\ \nonumber
  &=-\int_{Q^2_T}\Phi(\mathcal{G}^k_x(u^k(t,x),v^k(s,y))\big[(\mathcal{U}^k(t,x)-\mathcal{V}^k(s,y))\\ \nonumber 
  & \qquad \qquad +\sgn(u^k(t,x)-v^k(s,y))f^k(u^k(t,x))(\mathcal{U}^{k}_x(t,x)-\mathcal{V}^k_y(s,y))\big]\d t \d x \d y \d s \nonumber\\
  &:= I_{\mathcal{U}^k}+I_{\mathcal{U}^{k}_x} \label{est:IUU_x}. 
\end{align}We proceed by estimating each of the terms $I_{\mathcal{U}^k}$ and $I_{\mathcal{U}^{k}_x}$ one by one.
Note that using the properties of the convolution we have:
\begin{align}
\begin{split}\label{V_t:Lip1}
\abs{\tilde{\mathcal{V}}^k(s,x)-\tilde{\mathcal{U}}^k(s,x)}
&=\abs{\int_{\R}(v^k(s,z)-u^k(s,z))\omega_{\eta}(x-z)\d z}\\
&\le\norma{\omega_{\eta}}_{L^{\infty}(\R)}
\norma{v^k(s,\dott)-u^k(s,\dott)}_{L^1(\R)},\\
\abs{\tilde{\mathcal{U}}^k(t,x)-\tilde{\mathcal{U}}^k(s,x)}
&=\abs{\int_{\R}(u^k(t,z)-u^k(s,z))\omega_{\eta}(x-z)\d z}\\
&\le\norma{\omega_{\eta}}_{L^{\infty}(\R)}
\norma{u^k(s,\dott)-u^k(t,\dott)}_{L^1(\R)},\\
\abs{\tilde{\mathcal{V}}^k(s,y)-\tilde{\mathcal{V}}^k(s,x)}
&=\abs{\int_{\R}v^k(s,z)(\omega_{\eta}(x-z)-\omega_{\eta}(y-z))\d z}\\
&\le\abs{\omega_{\eta}}_{\lip([0,\eta])}
\norma{v^k(s,\dott)}_{L^1(\R)}|y-x|.
\end{split}
\end{align}
First we consider the term $I_{\mathcal{U}^k}$.
\begin{align}
\begin{split}\label{Ik}
I_{\mathcal{U}^k}&=\int_{Q^2_T}\Phi\big[\mathcal{G}^k_x(u^k(t,x),v^k(s,y))(\mathcal{V}^k(s,y)-\mathcal{U}^k(t,x))\big]\d{x}\d{t}  \d{y}\d{s}\\&\le\int_{Q^2_T}\Phi\abs{f^k}_{\lip(\R)}\abs{D_2u^k(t,x)}(\abs{\mathcal{V}^k(s,y)-\mathcal{V}^k(s,x)}+\abs{\mathcal{V}^k(s,x)-\mathcal{U}^k(s,x)})\d{x}\d{t}  \d{y}\d{s}\\
&\qquad+\int_{Q^2_T}\Phi\abs{f^k}_{\lip(\R)}\abs{D_2u^k(t,x)}\abs{\mathcal{U}^k(s,x)-\mathcal{U}^k(t,x)}\d{x}\d{t}  \d{y}\d{s},\end{split}\end{align}
since $|\mathcal{G}^k_x(u^k(t,x),v^k(s,y))|\le \abs{f^k}_{\lip(\R)}\abs{D_2u^k(t,x)}$ (in the sense of measures, see~\cite[Lemma A2.1]{BP1998} for details). 
   Note that 
\begin{align}
\begin{split}\label{V_t}
&\abs{{\mathcal{V}}^k(s,y)-\mathcal{V}^k(s,x)}=\abs{\nu^k({y,}(\tilde{\mathcal{V}}^k(s,y))-\nu^k({x,}(\tilde{\mathcal{V}}^k(s,x)}\\
&\leq \abs{\nu^k({y,}(\tilde{\mathcal{V}}^k(s,y))-\nu^k({x,}(\tilde{\mathcal{V}}^k(s,y))}+\abs{\nu^k({x,}(\tilde{\mathcal{V}}^k(s,y))-\nu^k({x,}(\tilde{\mathcal{V}}^k(s,x))}\\
&\le \norma{D_2\nu^k}_{L^\infty(\R\times [0,1])} \abs{\tilde{\mathcal{V}}^k(s,y)-\tilde{\mathcal{V}}^k(s,x)}+\norma{D_1\nu^k}_{L^\infty(\R\times [0,1])} \abs{y-x},\\
&\le \left(\norma{D_2\nu^k}_{L^\infty(\R\times [0,1])} \abs{\omega_{\eta}}_{\lip([0,\eta])}
\norma{v^k(s,\dott)}_{L^1(\R)}+\norma{D_1\nu^k}_{L^\infty(\R\times [0,1])}\right) \abs{y-x},
\end{split}
\end{align}owing to \eqref{V_t:Lip1}. Further,
\begin{align} \begin{split}
\abs{\mathcal{V}^k(s,x)-\mathcal{U}^k(s,x)}&=\abs{\nu^k({x,}\tilde{\mathcal{V}}^k(s,x))-\nu^k({x,}\tilde{\mathcal{U}}^k(t,x))}\\&\le \norma{D_2\nu^k}_{L^\infty(\R\times[0,1])}\abs{\tilde{\mathcal{V}}^k(s,x))-\tilde{\mathcal{U}}^k(t,x))}\\
&\le\norma{D_2\nu^k}_{L^\infty(\R\times [0,1])}\norma{\omega_{\eta}}_{L^{\infty}(\R)}\norma{v^k(s,\dott)-u^k(s,\dott)}_{L^1(\R)}\label{V_t:Lip2}, 
\end{split}
\end{align}\begin{align}
\begin{split}
\abs{\mathcal{U}^k(s,x)-\mathcal{U}^k(t,x)}&=\abs{\nu^k({x,}\tilde{\mathcal{U}}^k(s,x))-\nu^k({x,}\tilde{\mathcal{U}}^k(t,x))}\\ 
&\le \norma{D_2\nu^k}_{L^\infty(\R\times[0,1])}\abs{\tilde{\mathcal{U}}^k(s,x))-\tilde{\mathcal{U}}^k(t,x))}\\
&\le\norma{D_2\nu^k}_{L^\infty(\R\times [0,1])}\norma{\omega_{\eta}}_{L^{\infty}(\R)}
\norma{u^k(s,\dott)-u^k(t,\dott)}_{L^1(\R)}\label{V_t:Lip5}, 
\end{split}
\end{align}
owing again to \eqref{V_t:Lip1}. 
Consequently, using \eqref{V_t}--\eqref{V_t:Lip5} in \eqref{Ik} (see also \cite[Lemma~3.2, pp.~243--247]{AHV2023}), we get:
\begin{align*}
I_{\mathcal{U}^k}&\le \abs{f^k}_{\lip(\R)}\int_{Q^2_T}\Phi\abs{D_2u^k(t,x)}\abs{\mathcal{V}^k(s,y)-\mathcal{U}^k(t,x)}\d{x} \d{t}\d{y} \d{s}\\
&\leq I^1_{\mathcal{U}^k}+I^2_{\mathcal{U}^k}+I^3_{\mathcal{U}^k},
\end{align*}
where  $I^1_{\mathcal{U}^k},I^2_{\mathcal{U}^k}$, and $I^3_{\mathcal{U}^k}$ satisfy the following estimates, using \eqref{V_t:Lip1}, \eqref{V_t:Lip2}--\eqref{V_t:Lip5}:
\begin{align}
\begin{split}\label{I1u}
I^1_{\mathcal{U}^k}&=\int_{Q^2_T}\Phi\abs{f^k}_{\lip(\R)}\abs{D_2u^k(t,x)}\abs{\mathcal{V}^k(s,y)-\mathcal{V}^k(s,x)}\d{x}
\d{y}\d{t}\d{s}\\
&\le\abs{f^k}_{\lip(\R)}\norma{D_2\nu^k}_{L^\infty(\R\times [0,1])} \abs{\omega_{\eta}}_{\lip([0,\eta])}\\
&\qquad\times\int_{Q^2_T}\Theta_{\epsilon}(x-y)\Theta_{{\epsilon}_0}(t-s)\abs{D_2u^k(t,x)}
\norma{v^k(s,\dott)}_{L^1(\R)}|y-x|\d{x}
\d{y}\d{t}\d{s}\\
&\qquad+\abs{f^k}_{\lip(\R)}\norma{D_1\nu^k}_{L^\infty(\R\times [0,1])} \int_{Q^2_T}\Theta_{\epsilon}(x-y)\Theta_{{\epsilon}_0}(t-s)\abs{D_2u^k(t,x)}  \abs{y-x}\d{x}
\d{y}\d{t}\d{s}\\
&\le \abs{f^k}_{\lip(\R)}\norma{D_2\nu^k}_{L^\infty(\R\times [0,1])} \abs{\omega_{\eta}}_{\lip([0,\eta])}\epsilon\abs{u^k}_{L^{\infty}_t\operatorname{BV}_x}
\norma{{v}^k}_{L^1(\overline{Q}_T)}\\
&\qquad+\abs{f^k}_{\lip(\R)}\norma{D_1\nu^k}_{L^\infty(\R\times [0,1])} T\epsilon\abs{u^k}_{L^{\infty}_t\operatorname{BV}_x},
\end{split}
\end{align}
\begin{align*}
I^2_{\mathcal{U}^k}&=\int_{Q^2_T}\Phi\abs{f^k}_{\lip(\R)}\abs{D_2u^k(t,x)}\abs{\mathcal{V}^k(s,x)-\mathcal{U}^k(s,x)}\d{x}
\d{y}\d{t}\d{s}\\&\le\abs{f^k}_{\lip(\R)}\norma{D_2\nu^k}_{L^\infty(\R\times [0,1])}\norma{\omega_{\eta}}_{L^{\infty}(\R)}\\
&\qquad\times\int_{Q_T^2}\Theta_{{\epsilon}}(x-y)\Theta_{{\epsilon}_0}(t-s)
\abs{D_2u^k(t,x)}\norma{v^k(s,\dott)-u^k(s,\dott)}_{L^1(\R)}\d{t} \d{s}\d{x}\d{y}
\\
&\le \abs{f^k}_{\lip(\R)}\norma{D_2\nu^k}_{L^\infty(\R\times [0,1])}\norma{\omega_{\eta}}_{L^{\infty}(\R)}\abs{u^k}_{L^{\infty}_t \operatorname{BV}_x}\int_{0}^T
\norma{v^k(s,\dott)-u^k(s,\dott)}_{L^1(\R)}\d{s},\end{align*} and
\begin{align*}
I^3_{\mathcal{U}^k}
&=\int_{Q^2_T}\Phi\abs{f^k}_{\lip(\R)}\abs{D_2u^k(t,x)}\abs{\mathcal{U}^k(s,x)-\mathcal{U}^k(t,x)}\d{x}
\d{y}\d{t}\d{s}\\&\le\abs{f^k}_{\lip(\R)}\norma{D_2\nu^k}_{L^\infty(\R\times [0,1])}\norma{\omega_{\eta}}_{L^{\infty}(\R)}\\&\qquad \times \int_{Q^2_T}\Theta_{\epsilon}(x-y)\abs{D_2u^k(t,x)}\Theta_{{\epsilon}_0}(t-s)
\norma{u^k(s,\dott)-u^k(t,\dott)}_{L^1(\R)}\d{x} \d{y}\d{t}\d{s}\\
&\le \abs{f^k}_{\lip(\R)}\norma{D_2\nu^k}_{L^\infty(\R\times [0,1])}\norma{\omega_{\eta}}_{L^{\infty}(\R)}|u^k|_{\lip_tL^1_x}\\&\qquad\times\int_{Q^2_T}\Theta_{\epsilon}(x-y)\Theta_{{\epsilon}_0}(t-s)\abs{D_2u^k(t,x)}
|t-s|\d{x} \d{y}\d{t}\d{s}\\
&\leq \abs{f^k}_{\lip(\R)}\norma{D_2\nu^k}_{L^\infty(\R\times [0,1])}\norma{\omega_{\eta}}_{L^{\infty}(\R)}|u^k|_{\lip_tL^1_x}\abs{u^k}_{L^\infty_t \operatorname{BV}_x} \int_0^T\int_0^T\Theta_{{\epsilon}_0}(t-s)
{\epsilon_0}\d{t}\d{s}\\
&\leq \abs{f^k}_{\lip(\R)}\norma{D_2\nu^k}_{L^\infty(\R\times [0,1])}\norma{\omega_{\eta}}_{L^{\infty}(\R)}|u^k|_{\lip_tL^1_x}\abs{u^k}_{L^\infty_t \operatorname{BV}_x} T
 {\epsilon_0}.\end{align*}
Before, we proceed further, note that, for $s\in [0,T]$ and for a.e. $x\in \R$,  we have
\begin{align}\label{U_x}
{\tilde{\mathcal{V}}}^{k}_x(s,x)
&= \int_{x}^{x+\eta}v^k(s,z) \omega_{\eta}^{'}(z-x)\d z + v^k(s,x+\eta)\omega_{\eta}(\eta)-v^k(s,x)\omega_{\eta}(0).
\end{align}
Consequently, for $s\in [0,T],$
\begin{align}\label{VxLinf}
\norma{\tilde{\mathcal{V}}^{k}_x(s,\dott)}_{L^{\infty}(\R)}&\leq 3 \norma{v^k}_{L^{\infty}({Q}_T)}
\norma{\omega_{\eta}}_{L^{\infty}(\R)}.
\end{align}This implies that  $\tilde{\mathcal{V}}^k(s,\dott)$ is indeed a Lipschitz continuous function for $s\in [0,T].$
Note that due to $\omega_{\eta}$ being less regular compared to the kernel in \cite{AHV2023},  the function $\tilde{\mathcal{V}}^{k}_x$ is  not Lipschitz continuous in the space variable, restricting us to follow  arguments used in treating the terms $I_{\mathcal{U}^k}$ to treat the term $I_{\mathcal{U}_x^k}$ (unlike in \cite[Lemma~3.2, pp.~247--249]{AHV2023} where $\tilde{\mathcal{V}}^{k}_x$ turns out to be Lipschitz in space variable).
Instead, we take an alternative approach to estimate the term $I_{\mathcal{U}_x^k}$.
Before we begin estimating $I_{\mathcal{U}_x^k}$, note that from  \eqref{U_x}, we get: 
\begin{align}\nonumber
\int_{\R}\abs{{\tilde{\mathcal{V}}}^{k}_x(s,x)} \d x 
& \leq \int_{\R} \int_{x}^{x+\eta} \abs{v^k(s,z)} \abs{\omega_{\eta}^{'}(z-x)}\d z \d x + 2 \norma{v^k(s,\dott)}_{L^1(\R)} \norma{\omega_{\eta}}_{L^{\infty}(\R)}\\ \label{VxBV}
& \leq 3\norma{v^k(s,\dott)}_{L^1(\R)} \norma{\omega_{\eta}}_{L^{\infty}(\R)}< \infty,
\end{align}
which implies that $\tilde{\mathcal{V}}^k \in L^{\infty}(0,T;\BV(\R))$.   Also, from \eqref{U_x}, it follows that \begin{align}
\begin{split}\label{V_xBV}|\tilde{\mathcal{V}}^k_x (s,\dott)|_{\BV(\R)} &\leq 3 \norma{\omega_{\eta}}_{L^{\infty}(\R)}|v^k(s,\dott)|_{\BV(\R)}.
\end{split}
\end{align}
Similar bounds can be obtained on $\tilde{\mathcal{U}}^k$ and $\tilde{\mathcal{U}}_x^k$. Additionally, again using \eqref{U_x}, we get
\begin{align}
\begin{split}\label{U_xT}
\norma{\tilde{\mathcal{U}}^k_x (s,\dott)-\tilde{\mathcal{U}}^k_x (t,\dott)}_{L^1(\R)}&\leq 3 \norma{\omega_{\eta}}_{L^{\infty}(\R)}\norma{u^k(s,\dott)-u^k(t,\dott)}_{L^1(\R)}. \end{split}
\end{align} 
Let us now consider,
\begin{align}\nonumber
I_{\mathcal{U}^{k}_x}&=\int_{Q^2_T}\Phi\sgn(u^k(t,x)-v^k(s,y))f^k(u^k(t,x))(\mathcal{U}^{k}_x(t,x)-\mathcal{V}^k_y(s,y))\d t \d x \d y \d s\\
 &\le \abs{f^k}_{\lip(\R)}\int_{Q^2_T}\Theta_{\epsilon}(x-y)\Theta_{{\epsilon}_0}(t-s)\abs{u^k(t,x)}|\mathcal{V}^{k}_{y}(s,y)-\mathcal{U}^k_{x}(t,x)|\d t \d x \d y \d s. \label{IUx}
 \end{align}
Further, for any $s, t\in [0,T]$, and for a.e. $x,y \in \R$, we have
 \begin{align}
|\mathcal{V}^{k}_{y}(s,y)-\mathcal{U}^k_{x}(t,x)| \le\abs{\mathcal{V}^k_{y}(s,y)-\mathcal{V}^k_{x}(s,x)}+\abs{\mathcal{V}^k_{x}(s,x)-\mathcal{U}^k_{x}(s,x)}+\abs{\mathcal{U}^k_{x}(s,x)-\mathcal{U}^k_{x}(t,x)}.\label{u_b}\end{align}
We start by estimating each of the terms on the RHS of \eqref{u_b} one by one.
 Firstly, we consider
\begin{align*}
\abs{\mathcal{V}^{k}_x(s,x)-\mathcal{V}^{k}_y(s,y)} &\leq \abs{D_1\nu^k(x,\tilde{\mathcal{V}}^k(s,x))-D_1\nu^k(y,\tilde{\mathcal{V}}^k(s,y))}\\&\quad +\abs{D_2\nu^k(x,\tilde{\mathcal{V}}^k(s,x))\tilde{\mathcal{V}}_x^k(s,x)-D_2\nu^k(y,\tilde{\mathcal{V}}^k(s,y))\tilde{\mathcal{V}}_y^k(s,y)}.
\end{align*}
Further, using \eqref{V_t:Lip1}, we have,
{\allowdisplaybreaks\begin{align}\label{D0} \begin{split}
&\abs{D_1\nu^k(x,\tilde{\mathcal{V}}^k(s,x))-D_1\nu^k(y,\tilde{\mathcal{V}}^k(s,y))}\\
&\le\abs{D_1\nu^k({x,}(\tilde{\mathcal{V}}^k(s,x)))-D_1\nu^k({x,}(\tilde{\mathcal{V}}^k(s,y)))}+\abs{D_1\nu^k({x,}(\tilde{\mathcal{V}}^k(s,y))-D_1\nu^k({y,}(\tilde{\mathcal{V}}^k(s,y))}\\
&\le \norma{D_{12}\nu^k}_{L^\infty(\R\times [0,1])} \abs{\tilde{\mathcal{V}}^k(s,x)-\tilde{\mathcal{V}}^k(s,y)}+\norma{D_{11}\nu^k}_{L^\infty(\R\times [0,1])}|y-x|\\
&\le \abs{\omega_{\eta}}_{\lip([0,\eta])}
\norma{v^k(s,\dott)}_{L^1(\R)}|y-x|\norma{D_{12}\nu^k}_{L^\infty(\R\times [0,1])} +\norma{D_{11}\nu^k}_{L^\infty(\R\times [0,1])}|y-x|
\end{split}
\end{align}}
and
{\allowdisplaybreaks
\begin{align}
\begin{split}\label{D-1}
&\abs{D_2\nu^k(x,\tilde{\mathcal{V}}^k(s,x))\tilde{\mathcal{V}}_x^k(s,x)-D_2\nu^k(y,\tilde{\mathcal{V}}^k(s,y))\tilde{\mathcal{V}}_y^k(s,y)}\\
&\le\abs{D_2\nu^k({x,}(\tilde{\mathcal{V}}^k(s,x)))-D_2\nu^k({x,}(\tilde{\mathcal{V}}^k(s,y)))}\abs{\tilde{\mathcal{V}}_x^k(s,x)}\\&\qquad+\abs{D_2\nu^k({x,}(\tilde{\mathcal{V}}^k(s,y))-D_2\nu^k({y,}(\tilde{\mathcal{V}}^k(s,y))}\abs{\tilde{\mathcal{V}}_x^k(s,x)}
\\
&\qquad+\abs{D_2\nu^k({y,}(\tilde{\mathcal{V}}^k(s,y))}\abs{\tilde{\mathcal{V}}_y^k(s,y)-\tilde{\mathcal{V}}_x^k(s,x)}\\
&\le \norma{D_{22}\nu^k}_{L^\infty(\R\times [0,1])} \abs{\tilde{\mathcal{V}}^k(s,x)-\tilde{\mathcal{V}}^k(s,y)}\abs{\tilde{\mathcal{V}}_x^k(s,x)}+\norma{D_{21}\nu^k}_{L^\infty(\R\times [0,1])}|y-x|\abs{\tilde{\mathcal{V}}_x^k(s,x)}\\
&\qquad+\abs{D_2\nu^k({y,}(\tilde{\mathcal{V}}^k(s,y))}\abs{\tilde{\mathcal{V}}_y^k(s,y)-\tilde{\mathcal{V}}_x^k(s,x)}\\
&\le \norma{D_{22}\nu^k}_{L^\infty(\R\times [0,1])} \abs{\omega_{\eta}}_{\lip([0,\eta])}
\norma{v^k(s,\dott)}_{L^1(\R)}|y-x|\abs{\tilde{\mathcal{V}}_x^k(s,x)}\\
&\qquad+\norma{D_{21}\nu^k}_{L^\infty(\R\times [0,1])}|y-x|\abs{\tilde{\mathcal{V}}_x^k(s,x)}+\abs{D_2\nu^k({y,}(\tilde{\mathcal{V}}^k(s,y))}\abs{\tilde{\mathcal{V}}_y^k(s,y)-\tilde{\mathcal{V}}_x^k(s,x)}. 
\end{split}
\end{align}}
The second term on the RHS of \eqref{u_b} can be estimated as follows using \eqref{V_t:Lip1}:
\begin{align*}
\abs{\mathcal{V}^{k}_x(s,x)-\mathcal{U}^{k}_x(s,x)} &\leq  \abs{D_1\nu^k(x,\tilde{\mathcal{U}}^k(s,x))-D_1\nu^k(x,\tilde{\mathcal{V}}^k(s,x))}\\&\quad +\abs{D_2\nu^k(x,\tilde{\mathcal{U}}^k(s,x))\tilde{\mathcal{U}}_x^k(s,x)-D_2\nu^k(x,\tilde{\mathcal{V}}^k(s,x))\tilde{\mathcal{V}}_x^k(s,x)},
\end{align*}
where 
\begin{align}
\begin{split}
    \label{D1}
&\abs{D_1\nu^k(x,\tilde{\mathcal{U}}^k(s,x))-D_1\nu^k(x,\tilde{\mathcal{V}}^k(s,x))}\\
&\le \norma{D_{12}\nu^k}_{L^\infty(\R\times [0,1])} \abs{\tilde{\mathcal{V}}^k(s,x)-\tilde{\mathcal{U}}^k(s,x)}\\&\le \norma{D_{12}\nu^k}_{L^\infty(\R\times [0,1])}\norma{\omega_{\eta}}_{L^{\infty}(\R)}\norma{v^k(s,\dott)-u^k(s,\dott)}_{L^1(\R)},
\end{split}
\end{align}
and 
\begin{align}\label{D2}
\begin{split}
&\abs{D_2\nu^k(x,\tilde{\mathcal{U}}^k(s,x))\tilde{\mathcal{U}}_x^k(s,x)-D_2\nu^k(x,\tilde{\mathcal{V}}^k(s,x))\tilde{\mathcal{V}}_x^k(s,x)}\\
&\le\abs{D_2\nu^k({x,}(\tilde{\mathcal{V}}^k(s,x)))-D_2\nu^k({x,}(\tilde{\mathcal{U}}^k(s,x)))}\abs{\tilde{\mathcal{V}}_x^k(s,x)}\\&\quad+\abs{D_2\nu^k({x,}(\tilde{\mathcal{U}}^k(s,x))}\abs{\tilde{\mathcal{U}}_x^k(s,x)-\tilde{\mathcal{V}}_x^k(s,x)}\\
&\le  \norma{D_{22}\nu^k}_{L^\infty(\R\times [0,1])} \abs{\tilde{\mathcal{V}}^k(s,x)-\tilde{\mathcal{U}}^k(s,x)}\abs{\tilde{\mathcal{V}}_x^k(s,x)}\\
&\quad +\abs{D_2\nu^k({x,}(\tilde{\mathcal{U}}^k(s,x))}\abs{\tilde{\mathcal{U}}_x^k(s,x)-\tilde{\mathcal{V}}_x^k(s,x)}\\
&\le \norma{D_{22}\nu^k}_{L^\infty(\R\times [0,1])} \norma{\omega_{\eta}}_{L^{\infty}(\R)}
\norma{v^k(s,\dott)-u^k(s,\dott)}_{L^1(\R)}\abs{\tilde{\mathcal{V}}_x^k(s,x)}\\&\quad+\abs{D_2\nu^k({x,}(\tilde{\mathcal{U}}^k(s,x))}\abs{\tilde{\mathcal{U}}_x^k(s,x)-\tilde{\mathcal{V}}_x^k(s,x)}.
\end{split}
\end{align}  
Finally, we consider the last term of \eqref{u_b}.
\begin{align*}
\abs{\mathcal{U}^{k}_x(t,x)-\mathcal{U}^{k}_x(s,x)} &\leq  \abs{D_1\nu^k(x,\tilde{\mathcal{U}}^k(s,x))-D_1\nu^k(x,\tilde{\mathcal{U}}^k(t,x))}\\&\quad +\abs{D_2\nu^k(x,\tilde{\mathcal{U}}^k(s,x))\tilde{\mathcal{U}}_x^k(s,x)-D_2\nu^k(x,\tilde{\mathcal{U}}^k(t,x))\tilde{\mathcal{U}}_x^k(t,x)},
\end{align*}
where,
\begin{align}\label{D3}
\begin{split}
&\abs{D_1\nu^k(x,\tilde{\mathcal{U}}^k(s,x))-D_1\nu^k(x,\tilde{\mathcal{U}}^k(t,x))}\\&\le \norma{D_{12}\nu^k}_{L^\infty(\R\times [0,1])} \abs{\tilde{\mathcal{U}}^k(s,x)-\tilde{\mathcal{U}}^k(t,x)}\\
&\le \norma{D_{12}\nu^k}_{L^\infty(\R\times [0,1])} \norma{\omega_{\eta}}_{L^{\infty}(\R)}
\norma{u^k(s,\dott)-u^k(t,\dott)}_{L^1(\R)}
\end{split}
\end{align}
and 
\begin{align}\label{D4}
\begin{split}
&\abs{D_2\nu^k(x,\tilde{\mathcal{U}}^k(s,x))\tilde{\mathcal{U}}_x^k(s,x)-D_2\nu^k(x,\tilde{\mathcal{U}}^k(t,x))\tilde{\mathcal{U}}_x^k(t,x)}\\
&\le\abs{D_2\nu^k({x,}(\tilde{\mathcal{U}}^k(t,x)))-D_2\nu^k({x,}(\tilde{\mathcal{U}}^k(s,x)))}\abs{\tilde{\mathcal{U}}_x^k(t,x)}\\& \quad +\abs{D_2\nu^k({x,}(\tilde{\mathcal{U}}^k(s,x))}\abs{\tilde{\mathcal{U}}_x^k(s,x)-\tilde{\mathcal{U}}_x^k(t,x)}\\
&\le \norma{D_{22}\nu^k}_{L^\infty(\R\times [0,1])} \abs{\tilde{\mathcal{U}}^k(s,x)-\tilde{\mathcal{U}}^k(t,x)}\abs{\tilde{\mathcal{U}}_x^k(t,x)}\\&\quad +\abs{D_2\nu^k({x,}(\tilde{\mathcal{U}}^k(s,x))}\abs{\tilde{\mathcal{U}}_x^k(s,x)-\tilde{\mathcal{U}}_x^k(t,x)}\\
&\le \norma{D_{22}\nu^k}_{L^\infty(\R\times [0,1])} \norma{\omega_{\eta}}_{L^{\infty}(\R)}
\norma{u^k(s,\dott)-u^k(t,\dott)}_{L^1(\R)}\abs{\tilde{\mathcal{U}}_x^k(t,x)}\\
&\quad+\abs{D_2\nu^k({x,}(\tilde{\mathcal{U}}^k(s,x))}\abs{\tilde{\mathcal{U}}_x^k(s,x)-\tilde{\mathcal{U}}_x^k(t,x)},
\end{split}
\end{align}
using \eqref{V_t:Lip1}. Substituting the above estimates in \eqref{IUx}, we get
\begin{align*}
I_{\mathcal{U}^{k}_x}
\leq 
I^1_{\mathcal{U}^k_x}+ \cdots +I^6_{\mathcal{U}^k_x},
\end{align*}
where $I^1_{\mathcal{U}^k_x} ,\ldots, I^6_{\mathcal{U}^k_x} $ satisfy the following estimates:
\begin{align*}
I^{1}_{\mathcal{U}^k_x}&=\abs{f^k}_{\lip(\R)}\int_{Q^2_T}\Theta_{\epsilon}(x-y)\abs{u^k(t,x)}\Theta_{{\epsilon}_0}(t-s)\\ & \qquad  \qquad \qquad \times \abs{D_1\nu^k(x,\tilde{\mathcal{V}}^k(s,x))-D_1\nu^k(y,\tilde{\mathcal{V}}^k(s,y))}\d{x}
\d{y}\d{t}\d{s}\\
&\le \abs{f^k}_{\lip(\R)}\norma{D_{12}\nu^k}_{L^\infty(\R\times [0,1])} \abs{\omega_{\eta}}_{\lip([0,\eta])}\\
&\qquad\times\int_{Q^2_T}\Theta_{\epsilon}(x-y)\abs{u^k(t,x)}\Theta_{{\epsilon}_0}(t-s)
\norma{v^k(s,\dott)}_{L^1(\R)}|y-x|\d{x}
\d{y}\d{t}\d{s}\\&\quad+\abs{f^k}_{\lip(\R)}\norma{D_{11}\nu^k}_{L^\infty(\R\times [0,1])}\int_{Q^2_T}\Theta_{\epsilon}(x-y)\abs{u^k(t,x)}\Theta_{{\epsilon}_0}(t-s)|y-x|\d{x}
\d{y}\d{t}\d{s}
\\& \le \epsilon  \abs{f^k}_{\lip(\R)} 
\abs{\omega_{\eta}}_{\lip([0,\eta])}
\norma{\boldsymbol{u}_0}_{(L^1(\R))^N}\norma{D_{12}\nu^k}_{L^\infty(\R\times [0,1])}\norma{v^k}_{L^1(\overline{Q}_T)}\\&\qquad+ 
\epsilon \abs{f^k}_{\lip(\R)}  
{\color{black}\norma{{u}^k}_{L^1(\overline{Q}_T)}}\norma{D_{11}\nu^k}_{L^\infty(\R\times [0,1])},
\end{align*}
\begin{align*}
I^{2}_{\mathcal{U}^k_x}&=\abs{f^k}_{\lip(\R)}\int_{Q^2_T}\abs{D_2\nu^k(y,\tilde{\mathcal{V}}^k(s,y))\tilde{\mathcal{V}}_y^k(s,y)-D_2\nu^k(x,\tilde{\mathcal{V}}^k(s,x))\tilde{\mathcal{V}}_x^k(s,x)}\\ &
 \qquad \qquad \qquad \times  \Theta_{\epsilon}(x-y)\abs{u^k(t,x)}\Theta_{{\epsilon}_0}(t-s) \d{x}
\d{y}\d{t}\d{s}\\&\le \abs{f^k}_{\lip(\R)}\norma{D_{22}\nu^k}_{L^\infty(\R\times [0,1])} \abs{\omega_{\eta}}_{\lip([0,\eta])} \\
& \qquad \quad \times \int_{Q^2_T}
\norma{v^k(s,\dott)}_{L^1(\R)}|y-x|\abs{\tilde{\mathcal{V}}_x^k(s,x)} \Theta_{\epsilon}(x-y)\abs{u^k(t,x)}\Theta_{{\epsilon}_0}(t-s) \d{x}
\d{y}\d{t}\d{s}\\
&\qquad+\abs{f^k}_{\lip(\R)}\norma{D_{21}\nu^k}_{L^\infty(\R\times [0,1])}\\
&\qquad \quad \times \int_{Q^2_T}\Theta_{\epsilon}(x-y)\abs{u^k(t,x)}\Theta_{{\epsilon}_0}(t-s)|y-x|\abs{\tilde{\mathcal{V}}_x^k(s,x)}\d{x}
\d{y}\d{t}\d{s}\\
&\qquad+\abs{f^k}_{\lip(\R)}\int_{Q^2_T}\Theta_{\epsilon}(x-y)\abs{u^k(t,x)}\Theta_{{\epsilon}_0}(t-s)\abs{D_2\nu^k({y,}(\tilde{\mathcal{V}}^k(s,y))}\\&
\qquad \qquad \qquad \qquad \times \abs{\tilde{\mathcal{V}}_y^k(s,y)-\tilde{\mathcal{V}}_x^k(s,x)}\d{x}
\d{y}\d{t}\d{s} \\
& \leq  \epsilon \abs{f^k}_{\lip(\R)} 
\norma{D_{22}\nu^k}_{L^\infty(\R\times [0,1])} \abs{\omega_{\eta}}_{\lip([0,\eta])}
\norma{v^k}_{L^1(\overline{Q}_T)} \norma{\tilde{\mathcal{V}}_x^k}_{L^{\infty}(\overline{Q}_T)}\norma{\boldsymbol{u}_0}_{(L^1(\R))^N}\\
&\qquad+ \epsilon \abs{f^k}_{\lip(\R)} \norma{u^k}_{L^1(\overline{Q}_T)}\norma{D_{21}\nu^k}_{L^\infty(\R\times [0,1])}  \norma{\tilde{\mathcal{V}}_x^k}_{L^{\infty}(\overline{Q}_T)}
 \\
&\qquad+ \epsilon T \abs{f^k}_{\lip(\R)} \norma{u^k}_{L^{\infty}(\overline{Q}_T)}\ \norma{D_2\nu^k}_{L^{\infty}(\R\times[0,1])}\abs{\tilde{\mathcal{V}}_x^k}_{L^\infty_t\BV_x},
\end{align*}
{\allowdisplaybreaks
\begin{align*}
I^3_{\mathcal{U}^k_x}&=\abs{f^k}_{\lip(\R)} \int_{Q^2_T}\abs{u^k(t,x)}
\abs{D_1\nu^k(x,\tilde{\mathcal{U}}^k(s,x))-D_1\nu^k(x,\tilde{\mathcal{V}}^k(s,x))} \\& \qquad \qquad  \qquad \times \Theta_{\epsilon}(x-y)\Theta_{{\epsilon}_0}(t-s) \d{x} \d{y}\d{t}\d{s}\\
& \leq \abs{f^k}_{\lip(\R)} \int_{Q^2_T}\abs{u^k(t,x)}
\norma{D_{12}\nu^k}_{L^\infty(\R\times [0,1])}  \abs{\tilde{\mathcal{U}}^k(s,x)-\tilde{\mathcal{V}}^k(s,x)}  \\& \qquad \qquad  \qquad \times \Theta_{\epsilon}(x-y)\Theta_{{\epsilon}_0}(t-s) \d{x} \d{y}\d{t}\d{s}
\\
&\le  \abs{f^k}_{\lip(\R)} \norma{\boldsymbol{u_0}}_{(L^{1}(\R))^N}\norma{\omega_{\eta}}_{L^{\infty}(\R)}\norma{D_{12}\nu^k}_{L^\infty(\R\times [0,1])} \\ & \qquad \qquad \qquad  \times \int_0^T
\norma{v^k(s,\dott)-u^k(s,\dott)}_{L^1(\R)}\d{s},
\end{align*}}
{\allowdisplaybreaks
\begin{align*} 
I^4_{\mathcal{U}^k_x}
&=\abs{f^k}_{\lip(\R)}\int_{Q^2_T} \abs{D_2\nu^k(x,\tilde{\mathcal{U}}^k(s,x))\tilde{\mathcal{U}}_x^k(s,x)-D_2\nu^k(x,\tilde{\mathcal{V}}^k(s,x))\tilde{\mathcal{V}}_x^k(s,x)}\\
&\qquad \qquad \qquad \times \Theta_{\epsilon}(x-y)\abs{u^k(t,x)}\Theta_{{\epsilon}_0}(t-s) \d{x}
\d{y}\d{t}\d{s}\\
&\le \abs{f^k}_{\lip(\R)}\norma{D_{22}\nu^k}_{L^\infty(\R\times [0,1])} \norma{\omega_{\eta}}_{L^{\infty}(\R)}\\ & \quad \qquad  \times \int_{Q^2_T}\Theta_{\epsilon}(x-y)\abs{u^k(t,x)}\Theta_{{\epsilon}_0}(t-s) 
\norma{v^k(s,\dott)-u^k(s,\dott)}_{L^1(\R)} \\ & \quad \qquad \qquad \times \abs{\tilde{\mathcal{V}}_x^k(s,x)} \d{x}
\d{y}\d{t}\d{s} \\&\qquad+\abs{f^k}_{\lip(\R)}\int_{Q^2_T}\Theta_{\epsilon}(x-y)\abs{u^k(t,x)}\Theta_{{\epsilon}_0}(t-s)\abs{D_2\nu^k({x,}(\tilde{\mathcal{U}}^k(s,x))} \\ & \qquad \qquad \qquad \qquad \times \abs{\tilde{\mathcal{U}}_x^k(s,x)-\tilde{\mathcal{V}}_x^k(s,x)}\d{x}
\d{y}\d{t}\d{s}\\ & \le \abs{f^k}_{\lip(\R)} \norma{\boldsymbol{u_0}}_{(L^1(\R))^N}\norma{D_{22}\nu^k}_{L^\infty(\R\times [0,1])} \norma{\omega_{\eta}}_{L^{\infty}(\R)} \norma{\tilde{\mathcal{V}}_x^k}_{L^{\infty}(\overline{Q}_T)} \\
& \qquad \times 
\int_0^T \norma{v^k(s,\dott)-u^k(s,\dott)}_{L^1(\R)} \d s \\&\qquad+\abs{f^k}_{\lip(\R)}\norma{u^k}_{L^{\infty}(\overline{Q}_T)}\norma{D_2\nu^k}_{L^{\infty}(\R \times [0,1])} \int_{Q_T}\abs{\tilde{\mathcal{U}}_x^k(s,x)-\tilde{\mathcal{V}}_x^k(s,x)}\d{x} \d{s}
\\ & \le \abs{f^k}_{\lip(\R)} \norma{\boldsymbol{u_0}}_{(L^1(\R))^N}\norma{D_{22}\nu^k}_{L^\infty(\R\times [0,1])} \norma{\omega_{\eta}}_{L^{\infty}(\R)} \norma{\tilde{\mathcal{V}}_x^k}_{L^{\infty}(\overline{Q}_T)} \\ & \qquad \qquad \times  \int_0^T \norma{v^k(s,\dott)-u^k(s,\dott)}_{L^1(\R)} \d s \\&\qquad+3\abs{f^k}_{\lip(\R)}\norma{u^k}_{L^{\infty}(\overline{Q}_T)}\norma{D_2\nu^k}_{L^{\infty}(\R \times [0,1])} \norma{\omega_{\eta}}_{L^{\infty}(\R)}\int_0^T \norma{v^k(s,\dott)-u^k(s,\dott)}_{L^1(\R)} \d s, 
\end{align*}}
\begin{align*}
I^5_{\mathcal{U}^k_x}
&=\abs{f^k}_{\lip(\R)}\int_{Q^2_T}\Theta_{\epsilon}(x-y)\abs{u^k(t,x)}\Theta_{{\epsilon}_0}(t-s)\\ & \qquad \qquad \qquad \times \abs{D_1\nu^k(x,\tilde{\mathcal{U}}^k(s,x))-D_1\nu^k(x,\tilde{\mathcal{U}}^k(t,x))}\d{x}
\d{y}\d{t}\d{s}\\
&\le\abs{f^k}_{\lip(\R)}\norma{D_{12}\nu^k}_{L^\infty(\R\times [0,1])}\norma{\omega_{\eta}}_{L^{\infty}(\R)}\\&\qquad \qquad \qquad \times \int_{Q^2_T}\Theta_{\epsilon}(x-y)\abs{u^k(t,x)}\Theta_{{\epsilon}_0}(t-s) \norma{u^k(s,\dott)-u^k(t,\dott)}_{L^1(\R)}\d{x}
\d{y}\d{t}\d{s}\\
& \leq  T\epsilon_0\abs{f^k}_{\lip(\R)} \norma{\boldsymbol{u}_0}_{(L^1(\R))^N} \norma{D_{12}\nu^k}_{L^\infty(\R\times [0,1])} \norma{\omega_{\eta}}_{L^{\infty}(\R)} \abs{u^k}_{\lip_tL^1_x},
\end{align*}
and
{\allowdisplaybreaks
\begin{align*}
I^6_{\mathcal{U}^k_x}
&=\abs{f^k}_{\lip(\R)}\int_{Q^2_T}\abs{u^k(t,x)}\abs{D_2\nu^k(x,\tilde{\mathcal{U}}^k(s,x))\tilde{\mathcal{U}}_x^k(s,x)-D_2\nu^k(x,\tilde{\mathcal{U}}^k(t,x))\tilde{\mathcal{U}}_x^k(t,x)}\\
& \qquad \qquad \qquad \times  
\Theta_{\epsilon}(x-y) \Theta_{{\epsilon}_0}(t-s) \d{x}
\d{y}\d{t}\d{s}\\
& \le \abs{f^k}_{\lip(\R)}\norma{D_{22}\nu^k}_{L^\infty(\R\times [0,1])} \norma{\omega_{\eta}}_{L^{\infty}(\R)} \\& \qquad \times \int_{Q^2_T}\Theta_{\epsilon}(x-y)\abs{u^k(t,x)}\Theta_{{\epsilon}_0}(t-s)
\norma{u^k(s,\dott)-u^k(t,\dott)}_{L^1(\R)}\abs{\tilde{\mathcal{U}}_x^k(t,x)}\d{x}
\d{y}\d{t}\d{s}\\
&\qquad+\abs{f^k}_{\lip(\R)}\int_{Q^2_T}\Theta_{\epsilon}(x-y)\abs{u^k(t,x)}\Theta_{{\epsilon}_0}(t-s)\abs{D_2\nu^k({x,}(\tilde{\mathcal{U}}^k(s,x))}\\
& \qquad \qquad \qquad \qquad \times \abs{\tilde{\mathcal{U}}_x^k(s,x)-\tilde{\mathcal{U}}_x^k(t,x)}
 \d{x}
\d{y}\d{t}\d{s}
\\ & \leq T\epsilon_0\abs{f^k}_{\lip(\R)}\norma{D_{22}\nu^k}_{L^\infty(\R\times [0,1])} \norma{\omega_{\eta}}_{L^{\infty}(\R)} \norma{\boldsymbol{u}_0}_{(L^1(\R))^N}
\abs{u^k}_{\lip_tL^1_x}\norma{\tilde{\mathcal{U}}_x^k}_{L^{\infty}(\overline{Q}_T)}\\
&\qquad+3\abs{f^k}_{\lip(\R)}\norma{D_2\nu^k}_{L^{\infty}(\overline{Q}_T)}\norma{\omega_{\eta}}_{L^{\infty}(\R)}\\&\qquad \quad \times\int_{Q_T}\int_{0}^T\Theta_{\epsilon}(x-y)\abs{u^k(t,x)}\Theta_{{\epsilon}_0}(t-s)\norma{u^k(s,\dott)-u^k(t,\dott)}_{L^1(\R)} \d{x}
\d{y}\d{t}\d{s},
\\ & \leq T\epsilon_0\abs{f^k}_{\lip(\R)}\norma{D_{22}\nu^k}_{L^\infty(\R\times [0,1])} \norma{\omega_{\eta}}_{L^{\infty}(\R)} \norma{\boldsymbol{u}_0}_{(L^1(\R))^N}
\abs{u^k}_{\lip_tL^1_x}\norma{\tilde{\mathcal{U}}_x^k}_{L^{\infty}(\overline{Q}_T)}\\
&\qquad+3\abs{f^k}_{\lip(\R)}\norma{D_2\nu^k}_{L^{\infty}(\overline{Q}_T)}\norma{\omega_{\eta}}_{L^{\infty}(\R)}\abs{u^k}_{\lip_tL^1_x}\\&\qquad\qquad\qquad\times\int_{0}^T\int_{0}^T\abs{u^k(t,x)}\Theta_{{\epsilon}_0}(t-s)|t-s| \d{x}
\d{t}\d{s},
\\ & \leq T\epsilon_0\abs{f^k}_{\lip(\R)}\norma{D_{22}\nu^k}_{L^\infty(\R\times [0,1])} \norma{\omega_{\eta}}_{L^{\infty}(\R)} \norma{\boldsymbol{u}_0}_{(L^1(\R))^N}
\abs{u^k}_{\lip_tL^1_x}\norma{\tilde{\mathcal{U}}_x^k}_{L^{\infty}(\overline{Q}_T)}\\
&\qquad+3T\epsilon_0\abs{f^k}_{\lip(\R)}\norma{u^k}_{L^{\infty}(\overline{Q}_T)}\norma{D_2\nu^k}_{L^{\infty}(\overline{Q}_T)}\norma{\omega_{\eta}}_{L^{\infty}(\R)}\abs{u^k}_{\lip_tL^1_x}.
\end{align*}}
The estimate \eqref{est:PP'}, follows by substituting the above estimates in \eqref{est:IUU_x}. Finally, we prove \eqref{est:I_S}.\\
\textbf{Proof of \eqref{est:I_S}}: Consider
\begin{align}\label{IS1}
    I^k_S
    &\le  \int_{Q^2_T}  \left|\mathcal{R}_{\boldsymbol{v}}^k(s,y)-\mathcal{R}_{\boldsymbol{u}}^k(t,x)\right| |\Phi(t,x,s,y)| \d x \d t  \d y \d s,
\end{align}
    where 
\begin{align*}
&\abs{{\color{black}\mathcal{R}}^k_{\boldsymbol{v}}(s,y)-{\color{black}\mathcal{R}}^k_{\boldsymbol{u}}(t,x)}\\
&\quad\leq  \abs{{\color{black}\mathcal{R}}^k_{\boldsymbol{v}}(s,y)-{\color{black}\mathcal{R}}^k_{\boldsymbol{v}}(s,x)}+\abs{{\color{black}\mathcal{R}}^k_{\boldsymbol{v}}(s,x)-{\color{black}\mathcal{R}}^k_{\boldsymbol{u}}(s,x)}+ \abs{{\color{black}\mathcal{R}}^k_{\boldsymbol{u}}(s,x)-{\color{black}\mathcal{R}}^k_{\boldsymbol{u}}(t,x)}\\ 
&\quad\le 2{\color{black}\abs{\boldsymbol{S}}_{(L^{\infty}(\R;\lip({[0,1]^4})))^N}}\sum_{j=k-1}^{k+1}\Big(\abs{v^{j}(s,y)-v^{j}(s,x)}+\abs{\mathcal{V}^{j}(s,y)-\mathcal{V}^{j}(s,x)}\\
    &\quad\qquad \qquad\qquad\qquad\qquad\qquad\quad+
  \abs{v^{j}(s,x)-u^{j}(s,x)}+\abs{\mathcal{V}^{j}(s,x)-\mathcal{U}^{j}(s,x)}\\
  &\quad\qquad\qquad \qquad\qquad\qquad\qquad\quad+
  \abs{u^{j}(s,x)-u^{j}(t,x)}+\abs{\mathcal{U}^{j}(s,x)-\mathcal{U}^{j}(t,x)}\Big).
\end{align*}
Now,
{\color{black}using the fact that $\boldsymbol{v}\in K$ and $\boldsymbol{u}$ is an entropy solution and properties of the mollifiers \eqref{mol}, we have for every $j\in\mathcal{N}$:
{\allowdisplaybreaks
\begin{align*}
\int_{Q_T^2} \abs{v^j(s,y)-v^j(s,x)} \Phi(t,x,s,y) \d t \d x \d s \d y &\leq \epsilon T \abs{v^j}_{L^{\infty}_tBV_x},\\
\int_{Q_T^2} \abs{v^j(s,x)-u^j(s,x)} \Phi(t,x,s,y) \d t \d x \d s \d y &\leq  T \norma{u^j(s,\dott)-v^j(s,\dott)}_{L^1(\R)},\\
\int_{Q_T^2} \abs{u^j(s,x)-u^j(t,x)} \Phi(t,x,s,y) \d t \d x \d s \d y & \leq  T \abs{u^j(s,\dott)-u^j(t,\dott)}_{\lip_tL^1_x}.\end{align*}}}
{\color{black}Using the fact that $\boldsymbol{u},\boldsymbol{v} \in (L^{\infty}_tL^1_x)^{N} \cap (L^{\infty}_t \BV_x)^{N}$ and assumption \ref{A3}, we have for every $j\in\mathcal{N},$
\begin{align*}
\abs{\mathcal{V}^j(s,y)- \mathcal{V}^j(s,x)} &\leq \abs{\boldsymbol{\nu}}_{(\lip(\R))^{N}} \int_{\R} \Big(\abs{v^j(s,{\color{black}y+z})-v^j(s,{x+z})}\\
&\qquad\qquad\qquad\qquad\qquad+\mathcal{C}_0\abs{a_0(x)-a_0(y)}\Big)\omega_{\eta}(z) \d z,\\[2mm]
\abs{\mathcal{V}^{j}(s,x)-\mathcal{U}^{j}(s,x)}&\leq \abs{\boldsymbol{\nu}}_{(\lip(\R))^{N}} \int_{\R} \abs{v^j(s,{x+z})-u^j(s,{x+z})}\omega_{\eta}(z) \d z,\\[2mm]
\abs{\mathcal{U}^{j}(s,x)-\mathcal{U}^{j}(t,x)}& \leq  \abs{\boldsymbol{\nu}}_{(\lip(\R))^{N}} \int_{\R} \abs{u^j(s,{x+z})-u^j(t,{x+z})}\omega_{\eta}(z) \d z.
\end{align*}}
Now, using Fubini's theorem and similar calculations as above, we have for every $j\in\mathcal{N}$,
\begin{align*}
\int_{Q_T^2} \abs{\mathcal{V}^j(s,y)-\mathcal{V}^j(s,x)} &\Phi(t,x,s,y) \d t \d x \d s \d y \\
&\qquad\leq \abs{\boldsymbol{\nu}}_{(\lip(\R))^{N}}  \epsilon T \abs{v^j}_{L^{\infty}_tBV_x}+\mathcal{C}_0 T \epsilon {\abs{a_{0}}_{\BV(\R)}},\\
\int_{Q_T^2} \abs{\mathcal{V}^j(s,x)-\mathcal{U}^j(s,x)}& \Phi(t,x,s,y) \d t \d x \d s \d y\\ 
&\qquad\leq  \abs{\boldsymbol{\nu}}_{(\lip(\R))^{N}}  T \norma{u^j(s,\dott)-v^j(s,\dott)}_{L^1(\R)},\\
\int_{Q_T^2} \abs{\mathcal{U}^j(s,x)-\mathcal{U}^j(t,x)}& \Phi(t,x,s,y) \d t \d x \d s \d y \\
&\qquad \leq  \abs{\boldsymbol{\nu}}_{(\lip(\R))^{N}}  T \abs{u^j(s,\dott)-u^j(t,\dott)}_{\lip_tL^1_x}.
\end{align*}}
The estimate \eqref{est:I_S}, follows by substituting the above estimates in \eqref{IS1}.
{\color{black}Finally, combining \eqref{est:I_T^k}--\eqref{est:I_S}, we obtain \eqref{est:kuz}, thereby completing the proof.}
\end{proof}
\begin{theorem}
[Uniqueness]\label{uniqueness}
Let $\boldsymbol{u},\boldsymbol{v}$ 
be the entropy solutions of the IVP for the system \eqref{PDE_NL_k}
  with initial data $\boldsymbol{u}_0,\boldsymbol{v}_0{\color{black}\in ((L^1 \cap \BV) (\R;[0,1]))^{N}}$, respectively. Then, the following holds:
 \begin{align*}
\norma{\boldsymbol{u}(T,\dott)-\boldsymbol{v}(T,\dott)}_{(L^1(\R))^{N}} &\leq \mathcal{C}_T \norma{\boldsymbol{u}_0-\boldsymbol{v}_0}_{(L^1(\R))^{N}}.
  \end{align*}
 In particular, if $\boldsymbol{u}_0=\boldsymbol{v}_0$, then $\boldsymbol{u}=\boldsymbol{v}$ a.e. in $\overline{Q}_T.$
\end{theorem}
\begin{proof}
Since $\boldsymbol{v}$ is an entropy solution of \eqref{PDE_NL_k} with the initial data $\boldsymbol{v}_0$, for each $k\in\mathcal{N}$, $\Lambda^k_{\epsilon,\epsilon_0}(v^k,u^k) \geq 0$.
Further, in \eqref{est:kuz}, \color{black}$\lim_{\epsilon_0 \rightarrow 0 }\gamma(v^k,\epsilon_0)=0$ and $\mathcal{C}_T$ is independent of $\epsilon,\epsilon_0.$ Now, the result follows by sending $\epsilon,\epsilon_0\rightarrow 0$ in \eqref{est:kuz}.
\end{proof}

The following section  is dedicated to the finite volume schemes approximating the  entropy solutions for the IVP \eqref{PDE_NL_k},\eqref{eq:u11A}, and establishes  their existence.
\section{The numerical scheme and its convergence}\label{exis}
Choose a space step $\Delta x,$ such that $\eta=N_{\eta}\Delta x,$ for {\color{black}some integer} $N_{\eta}.$
For $\Delta t>0$, and $\lambda:=\Delta t/\Delta x$, consider the equidistant spatial grid points $x_i:=i\Delta x$ for $i\in\Z$ and the temporal grid points $t^n:=n\Delta t$ 
for integers in $\mathcal{N}_T:=\{0, \ldots, N_T\}$, such that $T=N_T \D t$.  Let $\chi_i(x)$ denote the indicator function of $C_i:=[x_{i-1/2}, x_{i+1/2})$, where $x_{i+1/2}=\frac12(x_i+x_{i+1})$ and let
$\chi^{k,n}(t)$ denote the indicator function of $C^{n}:=[t^n,t^{n+1})$. {\color{black} Finally, write} $C_i^{k,n}:=C^{k,n}\times C_i$. 
{\color{black}In the rest of the article, for any vector $\{s_i\}_{i\in\Z},$ we let $\Delta_-s_i:=s_i-s_{i-1}$, $i\in \Z$, denote the backward difference operator.}\\
\textcolor{black}{Consider initial data $\boldsymbol{u}_0\in ((\BV\cap L^1)(\R ;[0,1]))^{N}$}. For each $k \in \mathcal{N}$, we approximate the initial data according to:
\begin{align*}
&u^{k}_{\Delta}(0,x):=\sum_{i\in\Z}\chi_i(x)u^{k,0}_i, \quad x\in \R, \text{ where }  u^{k,0}_i=\frac1{\Delta x}\int_{C_i}u^k_0(x) \d x,\quad i\in\Z.
\end{align*}
Throughout this section,  {\color{black}since $\boldsymbol{u}_0\in ((\BV\cap L^1)(\R ;[0,1]))^{N},$ we have} $0\leq u^{k,0}_i \leq 1$ for each $(i,k) \in \Z \times \mathcal{N}$.
Furthermore, we define an approximate solution  to ~\eqref{PDE_NL_k} that  is piecewise constant by $\boldsymbol{u}_{\Delta} =
\left(u_{\Delta}^{1}, \ldots, u_{\Delta}^{N}\right)$ where
\begin{align}\label{apx}
  u_{\Delta}^{k} (t,x) =  u^{k,n}_{i}, \quad (t,x)\in C^{k,n}_i,\,  (i,k,n) \in \Z \times \mathcal{N} \times \mathcal{N}_T,
\end{align} 
through the following marching formula:
\begin{align}\label{MF}
u^{k,n+1}_i&=u^{k,n}_i-\lambda \left[ \mathcal{F}_{i+1/2}^{k,n}(u^{k,n}_i,u^{k,n}_{i+1}) - \mathcal{F}_{i-1/2}^{k,n}(u^{k,n}_{i-1},u^{k,n}_{i})\right] + \D t \mathcal{R}_i^{k,n}, 
\end{align}
{\color{black} where $\mathcal{R}$ and $\mathcal{F}$ will be defined in the following.} Further, we write \textcolor{black}{$u^{k,n}=u^k_{\Delta}(t^n, \dott)$ and $\boldsymbol{u}^{n}=\boldsymbol{u}_{\Delta}(t^n,\dott)$}, and for a fixed $\beta\in\left(0,2/3\right)$, we define {\color{black}a Lax–Friedrichs type scheme by}
\begin{align}
\nonumber \mathcal{F}^{k,n}_{i+1/2}(u^{k,n}_i,u^{k,n}_{i+1})&:= 
  \frac{1}{2}\nu^k({x_{i+1/2}},c^{k,n}_{i+1/2})\left( f^k(u^{k,n}_i)
    +
    f^k(u^{k,n}_{i+1})\right)
  -
  \frac{\beta}{2\, \lambda}{\color{black}\Delta_{+}u^{k,n}_{i}},
  \end{align}
  and
  \begin{align}
\label{num:R}
\begin{split}
\mathcal{R}_i^{k,n}&= R^{k}({x_i}, {\color{black}u^{k-1,n}_i,u^{k,n}_{i},u^{k+1,n}_i},c^{k,n}_i,c^{k-1,n}_i,c^{k+1,n}_i)\\
&:=S^{k-1}({x_i,}u^{k-1,n}_i,u^{k,n}_i,c_i^{k-1,n},c_i^{k,n})-S^{k}({x_i,}u^{k,n}_i,u^{k+1,n}_i,c_i^{k,n},c_i^{k+1,n}),
\end{split}
\end{align}
 \textcolor{black}{where $S^k$ is defined by \eqref{eq:Sk}}. Further,
\begin{align}
\label{num:c_ip1/2}
\begin{split}c_{i+1/2}^{k,n}&:=\sum\limits_{p=0}^{N_{\eta}-1}\zeta_{p+1/2}u_{i+p+1}^{k,n}={\color{black}\sum\limits_{p=0}^{N_{\eta}-1}\zeta_{p-i+1/2}u_{p+1}^{k,n}},\\ \zeta_{p+1/2}&:=\int_{p\Delta x}^{(p+1)\Delta x}\omega_{\eta}(y)\d y,\qquad p\in\mathcal{N}_{\eta}:=\{0,1,\dots,N_{\eta}-1\}.
\end{split}
\end{align}
Throughout, $\Delta t$ is chosen in order to satisfy
the CFL condition \begin{align}\label{CFL}
 \lambda \le \frac{\min(1, 4-6{\beta},6{\beta},1-2\Delta t\abs{\boldsymbol{S}}_{(L^{\infty}(\R;\lip({[0,1]^4})))^N})}{1+6{\abs{f^k}_{\lip(\R)}}{\norma{\nu^k}_{L^{\infty}(\R\times [0,1])}}}.\end{align}
 
{\color{black}Note that \eqref{MF} can be written in the following form {\color{black} for a suitable $\mathcal{H}^k$}:
\begin{align*}
u_i^{k,n+1}&:=\mathcal{H}^k({\color{black}x_{i-1/2},x_i,x_{i+1/2}, c^{k,n}_{i-1/2},c^{k,n}_{i+1/2},c^{k,n}_{i},c^{k-1,n}_{i},c^{k+1,n}_{i}},\\
&\qquad{\color{black}u_{i-1}^{k,n},u_{i}^{k,n},u_{i+1}^{k,n},u_i^{k-1,n},u_i^{k+1,n}})\\
&=\mathcal{H}_{C}^k({x_{i-1/2},x_{i+1/2},}c^{k,n}_{i-1/2},c^{k,n}_{i+1/2},u_{i-1}^{k,n},u_{i}^{k,n},u_{i+1}^{k,n}) \\&\quad +\mathcal{H}_S^k({x_{i}},c^{k,n}_{i},c^{k-1,n}_{i},c^{k+1,n}_{i},{\color{black}u_{i}^{k,n},u_i^{k-1,n},u_i^{k+1,n}}),
\end{align*}
where  $\mathcal{H}_S^k({x_{i}},c^{k,n}_{i},c^{k-1,n}_{i},c^{k+1,n}_{i}, u_{i}^{k,n},u_i^{k-1,n},u_i^{k+1,n})= \D t \mathcal{R}_i^{k,n}$ comes from the source term, and $\mathcal{H}_{C}^k$ comes from the (remaining) convective terms of \eqref{MF}. We now show that the numerical approximations obtained by the algorithm \eqref{MF} satisfy various expected physical properties that implies that  it  converges to the unique entropy solution of 
the IVP \eqref{PDE_NL_k},\eqref{eq:u11A}. 
\begin{remark}\label{rem:monH_c}\normalfont 
Note that $\mathcal{H}_C^k$ is same as the numerical flux used for nonlocal conservation law in \cite{ACG2015,ACT2015} in one dimension and for an
arbitrary fixed sequence $\{x_{i+1/2}, c^{k,n}_{i+1/2}\}_{i\in \Z}$, it is monotone in last three arguments under the CFL condition \eqref{CFL}, see\cite[eq.~(4.17)]{ACT2015}.
\end{remark}
\begin{lemma}[Monotonicity] \label{lem:monotone}For any $(i,k,n) \in \Z \times \mathcal{N} \times \mathcal{N}_T$, and for given sequences
$\{x_i, x_{i\pm 1/2}, c_{i\pm 1/2}^{k,n}\}_{i\in \Z}$, $\{c_{i\pm 1/2}^{k\pm 1,n}\}_{i\in \Z}$, $\mathcal{H}^k$ is increasing in the last five arguments, under the CFL condition \eqref{CFL}.
\end{lemma}
\begin{proof}
Fix $k \in \mathcal{N}$.  We first prove that $S^k$ is increasing in the second variable and decreasing in the third variable.
Using \eqref{eq:Sk},
\begin{align*}
S^k(x,a,b,A,B)&:= \max(\mathcal{M}^k(x,a,b,A,B),0)a+\min(\mathcal{M}^k(x,a,b,A,B),0)b,
\end{align*}
where $\mathcal{M}^k(x,a,b,A,B):=g^{k+1}(b)\nu^{k+1}({x, }B)-g^k(a)\nu^k({x, }A).$ Now, note that
\begin{align}
\mathcal{M}^k_a=-D_1g^k(a)\nu^k({x, }A)\ge0, \quad\mathcal{M}^k_b=D_1g^{k+1}(b)\nu^{k+1}({x, }B)\le 0,\end{align} owing to \ref{A1} and non-increasingness of $g^k$. Hence, if $\mathcal{M}^k\ge0,$\begin{align}
S^k_a=-aD_1g^k(a)\nu^k({x, }A)\ge0, \quad S^k_b=D_1g^{k+1}(b)\nu^{k+1}({x, }B)\le 0,
\end{align} 
since $a,b\ge 0,$ being the variables denoting $u^k_n.$ This proves that if $\mathcal{M}^k\ge0,$ $S^k$ is increasing in the second variable and decreasing in the third variable. Similar proof follows when $\mathcal{M}^k\le0.$ We now prove that $\mathcal{H}^k$ is increasing in  $u_{i}^{k\pm 1,n}$. Using the fact that $S^k$ is increasing in second and third arguments in \eqref{num:R}, we have $\mathcal{R}_i^{k,n}$ is increasing  in $u_{i\pm 1}^{k\pm 1,n}$. Hence using \eqref{MF}, we have $\mathcal{H}_S^k$ and hence $\mathcal{H}^k$ is increasing in $u_{i}^{k\pm 1,n}.$ We next prove that $\mathcal{H}^k$ is increasing in $u_{i\pm 1}^{k,n}$. Using Remark ~\ref{rem:monH_c},
 $\mathcal{H}^k_{C}$ is increasing in $u_{i\pm 1}^{k,n}$ under the CFL condition \eqref{CFL}. Thus, $\mathcal{H}^k$ is increasing in $u_{i\pm 1}^{k,n}$.  We finally prove that $\mathcal{H}^k$ is increasing in $u_{i}^{k,n}$. To that end we have
\begin{align*}
\begin{split}
\partial_{u_{i}^{k,n}} \mathcal{H}^k &=\partial_{u_{i}^{k,n}}  \mathcal{H}^k_{C}+\Delta t  \partial_{u_{i}^{k,n}}  S^{k-1}-\Delta t \partial_{u_{i}^{k,n}}  S^{k}\\
     &=1-D_1f^k(u^{k,n}_i)\frac{\lambda }{2}\left(\nu^k(x_{i+1/2}, c^{k,n}_{i+1/2})-\nu^k(x_{i-1/2},c^{k,n}_{i-1/2})\right)-\Delta t \partial_{u_{i}^{k,n}} ( S^{k}- S^{k-1})\\
     &\ge 1-\lambda \norma{f^k}_{\lip(\R)}{\norma{\nu^k}_{L^{\infty}(\R\times [0,1])}}-2\Delta t\abs{\boldsymbol{S}}_{(L^{\infty}(\R;\lip({[0,1]^4})))^N})
     \ge 0,
 \end{split}   
 \end{align*}
 under the CFL condition \eqref{CFL}, finishing the proof.
\end{proof}}
\begin{lemma}\label{lem:stability} Fix an initial data ${\boldsymbol{u}_0} \in ((L^1  \cap \BV)
  (\R;[0,1))^{N}$. Let {\color{black} the CFL condition }\eqref{CFL} hold. Then, the approximate solution
  ${\boldsymbol{u}_\Delta}$ defined by the {marching formula}~\eqref{MF} satisfies:
\begin{enumerate}[(a)]
\item \label{IRP}(Invariant region principle) If $0\leq {u}^0_{\Delta}(t,x) \leq 1, (t,x,k)\in \overline{Q}_T\times\mathcal{N}$, then $0\leq {u}^k_{\Delta}(t,x) \leq 1, \,\,\,{\color{black}\forall}(t,x,k)\in \overline{Q}_T\times\mathcal{N}.$
  \item \label{pos}(Conservation) {\color{black} For $n \in \mathcal{N}_T$,}
\begin{align*}
 \sum\limits_{k\in \mathcal{N}} \sum\limits_{i\in \Z}u_i^{k,n} =  \sum\limits_{k\in \mathcal{N}}\sum\limits_{i\in \Z}u_i^{k,0}, n\in \mathcal{N}_T.
\end{align*}
    \item (Total variation bound) {\color{black} For $n \in \mathcal{N}_T$,}
  \begin{align}\label{TV_u}
\sum_{k\in\mathcal{N}}\TV(u^{k,n+1})\le    \exp({\color{black}\mathcal{C}_{TV}} t^n) \sum_{k\in\mathcal{N}} \TV(u^{k,0})
     + \exp({\color{black}\mathcal{C}_{TV}} t^n)-1, 
\end{align}
where {\color{black}${\color{black}\mathcal{C}_{TV}}$ depends on $N,\abs{\boldsymbol{S}}_{(L^{\infty}(\R;\lip({[0,1]^4})))^N},\abs{\omega_{\eta}}_{\BV(\R)},\norma{\boldsymbol{u}_0}_{(L^1(\R))^N},\abs{f^k}_{\lip(\R)},$\\$\norma{D_1 \nu^k}_{L^{\infty}(\R \times[0,1])}$,
$\norma{D_2 \nu^k}_{L^{\infty}(\R \times[0,1])},\norma{D_{11} \nu^k}_{L^{\infty}(\R \times [0,1])}, \norma{D_{12} \nu^k}_{L^{\infty}(\R \times [0,1])} ,\\\norma{D_{2} \nu^k}_{L^{\infty}(\R \times [0,1])},\norma{D_{2 1} \nu^k}_{L^{\infty}(\R \times [0,1])}$ and $\norma{D_{2 2} \nu^k}_{L^{\infty}(\R \times [0,1])}.$}
\item (Lipschitz continuity in time estimate)\label{lem:TE}
For $(k,m,n)\in \mathcal{N}\times\mathcal{N}_T\times\mathcal{N}_T$, 
\begin{eqnarray*}
\Delta x \sum \limits_{i\in \mathbb{Z}} \abs{u_i^{k,m}-u_i^{k,n}} \leq {\color{black}\mathcal{C}_{L}} \Delta t (m-n), \quad m>n, 
\end{eqnarray*}
where {\color{black}${\color{black}\mathcal{C}_{L}}$ depends on $\norma{\boldsymbol{u}_0}_{(L^1(\R))N}, \norma{D_1 \nu^k}_{L^{\infty}(\R \times[0,1])},\norma{D_2 \nu^k}_{L^{\infty}(\R \times[0,1])}$,\\$\abs{\boldsymbol{S}}_{(L^{\infty}(\R;\lip({[0,1]^4})))^N},\norma{\boldsymbol{\nu}}_{(L^{\infty}(\R \times [0,1]))^N}$ and $ \TV(\boldsymbol{u}_{\D}(t^n,\dott)).$}
\item (Discrete entropy inequality)
For any $\alpha\in\R$ and for any $(i,k,n)\in\Z\times \mathcal{N}\times \mathcal{N}_T$,
\begin{align} 
\begin{split}
&\abs{u_i^{k,n+1}-\alpha}-\abs{u_i^{k,n}- \alpha}+\lambda\big(\mathcal{G}^{k,n}_{i+1/2}(u_i^{k,n} ,u_{i+1}^{k,n},\alpha )-\mathcal{G}^{k,n}_{i-1/2}(u_{i-1}^{k,n} ,u_i^{k,n},\alpha)\big)\\
 &\qquad \qquad \qquad +\lambda\sgn(u_i^{k,n+1}-\alpha) f^k(\alpha)(\nu^k(x_{i+\frac{1}{2}},c_{i+\frac{1}{2}}^{k,n})-\nu^k(x_{i-\frac{1}{2}},c_{i-\frac{1}{2}}^{k,n})) \\
 &\qquad \qquad\qquad\qquad\qquad\qquad\qquad\qquad\qquad\leq  \D t\sgn(u_i^{k,n+1}-\alpha) \mathcal{R}_i^{k,n}, \label{apx:de}
 \end{split}
 \end{align}
where $\mathcal{G}^{k,n}_{i+1/2}(a,b,\alpha)=\mathcal{F}_{i+1/2}^{k,n}(a \vee \alpha,b \vee \alpha)-\mathcal{F}_{i+1/2}^{k,n}(a \wedge \alpha,b \wedge \alpha)$ for all $(i,n)\in\Z\times \mathcal{N}_T$,  with  $a \vee b:= \max(a,b),$ and $ a \wedge b:=\min(a,b)$.
 \item \label{CON}(Convergence)
The numerical approximations $\boldsymbol{u}_{\Delta}$  converge to the entropy solution $ \boldsymbol{u} \in (C([0,T];L^1(\R)) \cap L^{\infty}([0,T];\BV(\R)))^{N}$ as $\Delta x, \Delta t \to 0.$ 
\end{enumerate}
\end{lemma}
 \begin{proof}
\begin{enumerate}[(a)]
\item {\color{black} Since $f^k(1)=f^k(0)=0$, for any given sequences
$\{x_{i\pm 1/2},x_{i},c_{i\pm 1/2}^{k,n},c_{i\pm 1/2}^{k\pm 1,n}\}_{i\in \Z}$, we have
\begin{align*}
\mathcal{H}^k(\dott,\dott,\dott,\dott,\dott,\dott,\dott,\dott,0,0,0,0,0,)&=0,
\mathcal{H}^k(\dott,\dott,\dott,\dott,\dott,\dott,\dott,\dott,1,1,1,1,1,)&=1.
\end{align*}
Now, the result follows from the monotonicity of $\mathcal{H}^k$ in last five arguments (cf.~Lemma~\ref{lem:monotone}).}
\item {\color{black}This follows from the definition  of the scheme \eqref{MF} and positivity of $u^{k,n}$ proved in (a).}
\item 
For each $k\in\mathcal{N}$, writing \eqref{MF} in the incremental form and using the TVB property for the space dependent fluxes (cf.~\cite[Lemma~2.6]{ACG2015}), we get:

\begin{align} \label{HKCTV}
\begin{split}
   &\sum_{i\in \Z} \big|\mathcal{H}_{C}^k({x_{i+1/2},x_{i+3/2},}{\color{black}c^{k,n}_{i+1/2},c^{k,n}_{i+3/2},u_{i}^{k,n},u_{i+1}^{k,n},u_{i+2}^{k,n}})\\&\qquad -\mathcal{H}_{C}^k({x_{i-1/2},x_{i+1/2},}{\color{black}c^{k,n}_{i-1/2},c^{k,n}_{i+1/2},u_{i-1}^{k,n},u_{i}^{k,n},u_{i+1}^{k,n}})\big|
   \\ &\quad = {\color{black}\Big|u^{k,n}_{i+1}-\lambda(\mathcal{F}_{i+3/2}^{k,n}(u^{k,n}_{i+1},u^{k,n}_{i}) - \mathcal{F}_{i+1/2}^{k,n}(u^{k,n}_{i},u^{k,n}_{i+1}))}-\\&\qquad {\color{black}-\big(u^{k,n}_i-\lambda( \mathcal{F}_{i+1/2}^{k,n}(u^{k,n}_i,u^{k,n}_{i+1}) - \mathcal{F}_{i-1/2}^{k,n}(u^{k,n}_{i-1},u^{k,n}_{i}))\big) \Big|}
   \\ &\quad 
   \leq \TV(u^{k,n})+\sum_{i\in \Z} \abs{A^{k,n}_i}+\sum_{i\in \Z} \abs{B^{k,n}_i},
   \end{split}
   \end{align}
   where
    \begin{align*}\nonumber A^{k,n}_i&=\nonumber  
      - \lambda(f^k(u^{k,n}_{i+1})- f^k(u^{k,n}_{i}))\left(
\nu^k(x_{i+3/2},c^{k,n}_{{i+3/2}})-
      \nu^k(x_{i+1/2},c^{k,n}_{{i+1/2}})
      \right),\\
     B^{k,n}_i&=-\lambda  f^k(u^{k,n}_{i})\left(
\nu^k(x_{i+3/2},c^{k,n}_{{i+3/2}})-2
        \nu^k(x_{i+1/2},c^{k,n}_{{i+1/2}})+
       \nu^k(x_{i-1/2},c^{k,n}_{{i-1/2}})
      \right).
    \end{align*}
We make some remarks about $c^{k,n},$ specific to non-smoothness of $\omega_{\eta}$, following the arguments of \cite[Thm.~3.2]{FKG2018}).  {\color{black}Since $ u_i^{k,n}\in[0,1],$}  $\omega_{\eta}$ is decreasing on $[0,\eta],$ we have $\Delta_{-}{\zeta_{p+1/2}}\le 0 $
implying that, \begin{align*}
   {\color{black}-\Delta_{-}c^{k,n}_{{i+1/2}}}
    &=
\left(\zeta_{1/2}u^{{\color{black}k},n}_{i}+\sum\limits_{p=1}^{N_{\eta}-1/2}\Delta_{-}{\zeta_{p+1/2}}u^{k,n}_{p+i}- \zeta_{N_{\eta}-1}u^{k,n}_{N_{\eta}+i}\right)\le 
\zeta_{1/2}u^{{\color{black}k},n}_{i}\le 
\zeta_{1/2}
    \end{align*}
and
     \begin{align*}
   {\color{black}-\Delta_{-}c^{k,n}_{{i+1/2}}}&\ge 
\left(\zeta_{1/2}u^{{\color{black}k},n}_{i}+\sum\limits_{p=1}^{N_{\eta}-1/2}\Delta_{-}{\zeta_{p+1/2}}- \zeta_{N_{\eta}-1}\right)=\zeta_{1/2}u^{{\color{black}k},n}_{i}-\zeta_{1/2}\\
& \ge 
\zeta_{1/2}(u_i^{{\color{black}k},n}-1)\ge -\zeta_{1/2}.
    \end{align*}
Finally, we have,
\begin{align}\label{C}
  \abs{\Delta_{-}c^{{\color{black}k},n}_{{i+1/2}}}
&\le 
\zeta_{1/2}\le 
\omega_{\eta}(0)\Delta x.
    \end{align} 
    Further since,
\begin{align*}
  c^{k,n}_{{i-1/2}} - 2c^{k,n}_{{i+1/2}}+c^{k,n}_{{i+3/2}}
 &=-\zeta_{1/2}\Delta_+u_i^{k,n}-\sum\limits_{p=1}^{N_{\eta}-1}\Delta_{-}{\zeta_{p+1/2}}\Delta_+u_{i+p}^{k,n}\\
& \quad +\zeta_{N_{\eta}-1/2}\Delta_+u_{i+N_{\eta}}^{k,n},
\end{align*} 
we infer,
\begin{align}
\sum_i\abs{c^{k,n}_{{i-1/2}} - 2c^{k,n}_{{i+1/2}}+c^{k,n}_{{i+3/2}}}\nonumber
              & \le \sum_i\abs{\Delta_+u_i^{k,n}}\left(\zeta_{1/2} -\sum\limits_{p=1}^{N_{\eta}-1}\Delta_{-}{\zeta_{p+1/2}}+\zeta_{N_{\eta}-1/2}\right)\\
&=2\sum_i\abs{\Delta_+u_i^{k,n}}\zeta_{1/2}\label{double}\le2\sum_i\abs{\Delta_+u_i^{k,n}}\omega_{\eta}(0)\Delta x .
\end{align} 
To estimate the term $|A_i^{k,n}|$, we first consider 
\begin{align}\label{NuD}
\begin{split}
    &\abs{
\nu^k(x_{i+3/2},c^{k,n}_{{i+3/2}})-\nu^k(x_{i+1/2},c^{k,n}_{{i+1/2}})
   } \\
      &\qquad \quad \leq  \abs{\nu^k(x_{i+3/2},c^{k,n}_{{i+3/2}})-\nu^k(x_{i+1/2},c^{k,n}_{{i+3/2}})}\\
      &\qquad\qquad+ \abs{\nu^k(x_{i+1/2},c^{k,n}_{{i+3/2}})-\nu^k(x_{i+1/2},c^{k,n}_{{i+1/2}})}\\
      &\qquad\quad\le \abs{\Delta x D_1\nu^k (\overline{x}_{i+1}, c^{k,n}_{i+3/2})}+
   \abs{D_2 \nu^k (x_{i+1/2}, \overline{c}^{k,n}_{i+1})(c^{k,n}_{i+3/2}-c^{k,n}_{i+1/2})}\\
      &\qquad\quad \leq \norma{D_1 \nu^k}_{L^{\infty}(\R \times[0,1])} \D x + \norma{D_2 \nu^k}_{L^{\infty}(\R \times[0,1])} \omega_{\eta}(0)\Delta x,
\end{split}
\end{align}
  for suitable $\overline{x}_{i} \in (x_{i-1/2},
    x_{i+1/2})$ and $\overline{c}^{k,n}_{i}\in
  I(c^{k,n}_{i-1/2},c^{k,n}_{i+1/2})$. 
{\color{black}Thus,}
 \begin{align} 
 \begin{split}\label{B_i}
    \sum_{i\in \Z} {\abs{A_i^{k,n}}} &\leq   {\color{black}\mathcal{C}_{21}}\TV(u^{{\color{black}k},n}) \D t,
    \end{split}
\end{align}
where ${\color{black}\mathcal{C}_{21}}=\abs{f^{\color{black}k}}_{\lip(\R)}\max\{\norma{D_1 \nu^k}_{L^{\infty}(\R \times[0,1])}, \norma{D_2 \nu^k}_{L^{\infty}(\R \times[0,1])}, {\color{black}\omega_{\eta}(0)}\}.$

Now, to estimate the term $\abs{B^{k,n}_i}$, we first note that
  \begin{align} \begin{split} 
    &
      \nu^{k}(x_{i+3/2},c^{k,n}_{i+3/2})  - 2 \nu^{k}(x_{i+1/2},c^{k,n}_{i+1/2})  + \nu^{k}(x_{i-1/2},c^{k,n}_{i-1/2})
 \\ 
    &=(
      \nu^{k}(x_{i+3/2},c^{k,n}_{i+3/2})  -  \nu^{k}(x_{i+1/2},c^{k,n}_{i+1/2})) - 
   ( \nu^{k}(x_{i+1/2},c^{k,n}_{i+1/2})-\nu^{k}(x_{i-1/2},c^{k,n}_{i-1/2})
 )\\ 
    & =   D_1\nu^k(\overline{x}_{i+1},c^{k,n}_{i+3/2}) \D x + D_2\nu^k(x_{i+1/2},\overline{c}^{k,n}_{i+1})  (c^{k,n}_{i+3/2}-c^{k,n}_{i+1/2}) \\ 
&\quad -D_1\nu^k(\overline{x}_{i},c^{k,n}_{i+1/2}) \D x - D_2\nu^k(x_{i-1/2},\overline{c}^{k,n}_{i})  (c^{k,n}_{i+1/2}-c^{k,n}_{i-1/2})\\ 
&= (D_1\nu^k(\overline{x}_{i+1},c^{k,n}_{i+3/2})- D_1\nu^k(\overline{x}_{i},c^{k,n}_{i+1/2})) \D x\\
&\quad+ D_2\nu^k(x_{i+1/2},\overline{c}^{k,n}_{i+1})  (c^{k,n}_{i+3/2}-c^{k,n}_{i+1/2})  - D_2\nu^k(x_{i-1/2},\overline{c}^{k,n}_{i})  (c^{k,n}_{i+1/2}-c^{k,n}_{i-1/2})\\ \label{C_{1,i}+C_{2,i}}
&=B_{1,i}^{k,n}+B_{2,i}^{k,n}.
\end{split}
\end{align}
{\color{black}Further, the} terms $B_{1,i}^{k,n},B_{2,i}^{k,n}$ can be estimated by invoking the mean-value theorem 
once again:
 \begin{align}
\begin{split}
 \label{C_{1,i}}
    \abs{B_{1,i}^{k,n}}&=  |D_1\nu^k(\overline{x}_{i+1},c^{k,n}_{i+3/2})-D_1\nu^k(\overline{x}_{i},c^{k,n}_{i+1/2})|\D x\\ 
    & \leq \norma{D_{11}\nu^k}_{L^{\infty}(\R\times [0,1])}\D x^2+ \norma{D_{12}\nu^k}_{L^{\infty}(\R\times [0,1])}w_{\eta}(0) \D x^2,\\
\abs{B_{2,i}^{k,n}}&=| D_2\nu^k(x_{i+1/2},\overline{c}^{k,n}_{i+1})  (c^{k,n}_{i+3/2}-2c^{k,n}_{i+1/2}-c^{k,n}_{i-1/2})  \\& \quad + ( D_2\nu^k(x_{i+1/2},\overline{c}^{k,n}_{i+1}) -D_2\nu^k(x_{i-1/2},\overline{c}^{k,n}_{i}) ) (c^{k,n}_{i+1/2}-c^{k,n}_{i-1/2})| \\
 & \leq \norma{D_2 \nu^k}_{L^{\infty}(\R \times[0,1])}\abs{c^{k,n}_{i+3/2}-2c^{k,n}_{i+1/2}-c^{k,n}_{i-1/2}} \\&\quad + \norma{D_{21}\nu^k}_{L^{\infty}(\R \times[0,1])} \D x \abs{c^{k,n}_{i+1/2}-c^{k,n}_{i-1/2}}\\& \quad + \norma{D_{22}\nu^k}_{L^{\infty}(\R \times[0,1])} \abs{c^{k,n}_{i+1/2}-c^{k,n}_{i-1/2}}\abs{\overline{c}^{k,n}_{i}-\overline{c}^{k,n}_{i+1}}.
 \end{split}
\end{align} 
Since $\overline{c}^{k,n}_{i}\in I(c^{k,n}_{i-1/2},{c}^{k,n}_{i+1/2})$, $\abs{\overline{c}^{k,n}_{i+1}-\overline{c}^{k,n}_{i}}\le \abs{c^{{\color{black}k},n}_{{i+3/2}}-c^{{\color{black}k},n}_{{i-1/2}}}\le 2\omega_{\eta}(0)\Delta x$,
invoking the estimates \eqref{C_{1,i}+C_{2,i}}--\eqref{C_{1,i}} to get:
\begin{align}\begin{split}\label{C_i}
 \sum_{i\in \Z}\abs{C_i^{k,n}}&= \lambda 
 \sum\limits_{i\in \Z}\abs{f^k(u_i^{k,n})} \abs{B_{1,i}^{k,n}+B_{2,i}^{k,n}}\leq \lambda \abs{f^k}_{\lip(\R)} \sum\limits_{i\in \Z}  \abs{u_i^{k,n}} \left(\abs{B_{1,i}^{k,n}}+\abs{B_{2,i}^{k,n}} \right) 
 \\  &\leq  \bigg( \norma{D_{11} \nu^k}_{L^{\infty}(\R \times [0,1])} \norma{\boldsymbol{u}_0}_{{\color{black}(L^1(\R))^N}}+ \norma{D_{12} \nu^k}_{L^{\infty}(\R \times [0,1])} \omega_{\eta}(0) \norma{\boldsymbol{u}_0}_{{\color{black}(L^1(\R))^N}}  \\ & \quad + 2 \norma{D_{2} \nu^k}_{L^{\infty}(\R \times [0,1])} \TV(u^{{\color{black}k},n})   + \norma{D_{2 1} \nu^k}_{L^{\infty}(\R \times [0,1])} \norma{\boldsymbol{u}_0}_{{\color{black}(L^1(\R))^N}} \omega_{\eta}(0) \\  & \quad + \norma{D_{2 2} \nu^k}_{L^{\infty}(\R \times [0,1])} \norma{\boldsymbol{u}_0}_{{\color{black}(L^1(\R))^N}} (\omega_{\eta}(0))^2  \bigg)\abs{f^k}_{\lip(\R)} \D t
 \\
    &:= {\color{black}\mathcal{C}_{22}} \TV(u^{{\color{black}k},n}) \D t + {\color{black}\mathcal{C}_{23}} \D t.
    \end{split}
 \end{align}
 Substituting the estimates \eqref{B_i} and \eqref{C_i} in \eqref{HKCTV} we get the total variation bound:
 \begin{align}\label{TVC}
 \begin{split}
     &{\color{black}\sum_{i\in \Z} \big|\Delta_+\mathcal{H}_{C}^k({x_{i-1/2},x_{i+1/2},}c^{k,n}_{i-1/2},c^{k,n}_{i+1/2},u_{i-1}^{k,n},u_{i}^{k,n},u_{i+1}^{k,n})\big| }\\&\qquad{\color{black}\leq (1+{\color{black}\mathcal{C}_{24}} \D t) \TV(u^{k,n}) + {\color{black}\mathcal{C}_{24}} \D t},
 \end{split}
 \end{align}
 where ${\color{black}\mathcal{C}_{24}}={\color{black}\mathcal{C}_{21}}+{\color{black}\mathcal{C}_{22}}+{\color{black}\mathcal{C}_{23}}.$ 
 {\color{black}Using \eqref{num:R} and 
 the Lipschitz continuity of $\boldsymbol{S}$, cf.~\eqref{est:Lip_S}, we have
\begin{align}\label{RTV}
\begin{split}
    \abs{{\Delta_+\mathcal{R}_{i}^{k,n}}}&\leq 2\abs{\boldsymbol{S}}_{(L^{\infty}(\R;\lip({[0,1]^4})))^N} \Big( \abs{\Delta_+u_{i}^{k-1,n}} + \abs{\Delta_+u_{i}^{k,n}}+ \abs{\Delta_+u_{i}^{k+1,n}}\\ &\qquad \qquad  \qquad +\abs{\Delta_+c_{i}^{k-1,n}}+\abs{\Delta_+c_{i}^{k,n}}+ \abs{\Delta_+c_{i}^{k+1,n}}\Big),
\end{split}
\end{align}
where for every $(i,k,n)\in\Z\times\mathcal{N}\times\mathcal{N}_T$, $c^{k,n}_{{i}}=\frac{1}{2}(c^{k,n}_{{i+1/2}}+c^{k,n}_{{i+3/2}})$ and
{\allowdisplaybreaks {\color{black}
 \begin{align}  
 \begin{split}\label{TV_c}
  \sum_{i\in\Z} \abs{\Delta_+c^{k,n}_{{i}}}&=\sum_{i\in\Z}\frac{1}{2}\abs{(c^{k,n}_{{i+1/2}}+c^{k,n}_{{i+3/2}}) - (c^{k,n}_{{i-1/2}}+c^{k,n}_{{i+1/2}})}=\frac{1}{2}  \sum_{i\in\Z}\abs{c^{k,n}_{{i+3/2}} - c^{k,n}_{{i-1/2}}} \\ 
&=\frac{1}{2}\sum_{i\in\Z}\abs{\sum\limits_{p=0}^{N_{\eta}-1}\zeta_{p+1/2}(u_{i+p+2}^{k,n}-u_{i+p}^{k,n})}\\
&\le \frac{1}{2}\sum\limits_{p=0}^{N_{\eta}-1}\sum_{i\in\Z}\left(\abs{\Delta_+u^{k,n}_{i+p+1}}+\abs{\Delta_+u^{k,n}_{i+p}}\right)\abs{\zeta_{p+1/2}}\\ &\leq  \norma{\omega_{\eta}}_{L^1(\R)}\TV(u^{k,n})=\TV(u^{k,n}).\end{split} 
    \end{align}	}
}
 Now, using \eqref{TVC}--\eqref{TV_c},
\begin{align}\label{TVC1}
\begin{split}
\TV(u^{k,n+1})
&\leq (1+{\color{black}\mathcal{C}_{24}} \D t)  \TV(u^{k,n}) + {\color{black}\mathcal{C}_{24}} \D t \\
&\quad + 4\abs{\boldsymbol{S}}_{(L^{\infty}(\R;\lip({[0,1]^4})))^N}\D t \bigg(  \TV(u^{k,n}) + \TV(u^{k-1,n})+ \TV(u^{k+1,n})\bigg) \\
&\leq  (1+{\color{black}\mathcal{C}_{24}} \D t) \TV(u^{k,n}) + 4\abs{\boldsymbol{S}}_{(L^{\infty}(\R;\lip({[0,1]^4})))^N}\D t \sum_{k\in\mathcal{N}} \TV(u^{k,n})\\&\quad + {\color{black}\mathcal{C}_{24}} \D t.
\end{split}\end{align}
Consequently,
\begin{align*}
\sum_{k\in \mathcal{N}} \TV(u^{k,n+1}) & \leq (1+\mathcal{C}_{25} \D t) \sum_{k\in \mathcal{N}}\TV(u^{k,n}) +\mathcal{C}_{26}\D t,
\end{align*}
where $
\mathcal{C}_{25}={\color{black}\mathcal{C}_{24}}+4N\abs{\boldsymbol{S}}_{(L^{\infty}(\R;\lip({[0,1]^4})))^N}$ and  $
\mathcal{C}_{26}=N{\color{black}\mathcal{C}_{24}}.$
Setting ${\color{black}\mathcal{C}_{TV}}=\mathcal{C}_{25}+\mathcal{C}_{26}$, and applying mathematical induction, the result follows.
\item Consider 
\begin{align*}
\sum\limits_{i\in \Z}\abs{u_i^{k,n+1}-u_i^{k,n}} \leq \lambda \sum\limits_{i\in \Z} \abs{ \mathcal{F}_{i+1/2}^{k,n}(u^{k,n}_i,u^{k,n}_{i+1}) - \mathcal{F}_{i-1/2}^{k,n}(u^{k,n}_{i-1},u^{k,n}_{i})} + \D t \sum\limits_{i\in \Z} \abs{\mathcal{R}_i^{k,n}},
\end{align*}
where
\begin{align*}
    \abs{ \mathcal{F}_{i+1/2}^{k,n}(u^{k,n}_i,u^{k,n}_{i+1}) - \mathcal{F}_{i-1/2}^{k,n}(u^{k,n}_{i-1},u^{k,n}_{i})} & \leq \abs{\nu^k(x_{i+1/2},c_{i+1/2}^{k,n})} \left(\abs{\Delta_{-}u^{k,n}_{i}}+ \abs{\Delta_{+}u^{k,n}_{i}} \right) \\
    &\quad + 
    \abs{
\nu^k(x_{i+1/2},c^{k,n}_{{i+1/2}})-\nu^k(x_{i-1/2},c^{k,n}_{{i-1/2}})
   }\\
   &\qquad\times\left(\abs{u^{k,n}_{i}} +\abs{u^{k,n}_{i+1}}\right).
\end{align*}

Using the total variation bounds on $\boldsymbol{u}_{\Delta}$ established in $(c)$, \eqref{C}, 
 and \eqref{NuD}, we have
\begin{align*}
    &\sum\limits_{i\in \Z}\abs{ \mathcal{F}_{i+1/2}^{k,n}(u^{k,n}_i,u^{k,n}_{i+1}) - \mathcal{F}_{i-1/2}^{k,n}(u^{k,n}_{i-1},u^{k,n}_{i})} \\&\quad\leq 2\norma{\boldsymbol{\nu}}_{(L^{\infty}(\R \times [0,1]))^N} \TV(\boldsymbol{u}_{\D}(t^n,\dott))\\
    &\qquad+ 2(\norma{D_1 \nu^k}_{L^{\infty}(\R \times[0,1])}  + \norma{D_2 \nu^k}_{L^{\infty}(\R \times[0,1])} \omega_{\eta}(0))\norma{\boldsymbol{u}_0}_{(L^1(\R))N}.
\end{align*}
Further, since $R^k$ is Lipschitz and $R^k(\boldsymbol{0})=0$, we have
\begin{align*}
   &  \abs{\mathcal{R}_i^{k,n}}\\
   &\quad\leq  2 \abs{\boldsymbol{S}}_{(L^{\infty}(\R;\lip({[0,1]^4})))^N}\left( \abs{u_i^{k,n}}+ \abs{u_i^{k-1,n}} + \abs{u_i^{k+1,n}}+ \abs{c_i^{k,n}}+ \abs{c_i^{k-1,n}} + \abs{c_i^{k+1,n}}\right).
\end{align*}
Consequently, using the definition of \eqref{num:c_ip1/2}, we have 
\begin{align*}
    \D t \sum_{i\in\Z} \abs{\mathcal{R}_i^{k,n}} \leq  12 \lambda \abs{\boldsymbol{S}}_{(L^{\infty}(\R;\lip({[0,1]^4})))^N}\norma{\boldsymbol{u}_0}_{(L^1(\R))N}  ,
\end{align*}
   giving the required estimate,   with 
  \begin{align*}
     {\color{black}\mathcal{C}_{L}}&=\norma{\boldsymbol{u}_0}_{(L^1(\R))N} \Big(2(\norma{D_1 \nu^k}_{L^{\infty}(\R \times[0,1])}  + \norma{D_2 \nu^k}_{L^{\infty}(\R \times[0,1])} \omega_{\eta}(0))\\
     &\qquad\qquad\qquad+12 \lambda \abs{\boldsymbol{S}}_{(L^{\infty}(\R;\lip({[0,1]^4})))^N}\Big)\\
&\quad+2\norma{\boldsymbol{\nu}}_{(L^{\infty}(\R \times [0,1]))^N} \TV(\boldsymbol{u}_{\D}(t^n,\dott)).
   \end{align*}
\item 
 Using Remark~\ref{rem:monH_c}, since $\mathcal{H}^k_C$ is increasing in its last three arguments, we have  
\begin{align}  \nonumber
u_{i}^{k,n+1} &= \mathcal{H}_{C}^k(x_{i-1/2},x_{i+1/2},c^{k,n}_{i-1/2},c^{k,n}_{i+1/2},u_{i-1}^{k,n},u_{i}^{k,n},u_{i+1}^{k,n}) +\Delta t \mathcal{R}_i^{k,n}  \\ \nonumber 
&\le \mathcal{H}_{C}^k(x_{i-1/2},x_{i+1/2},c^{k,n}_{i-1/2},c^{k,n}_{i+1/2},u_{i-1}^{k,n}\vee \alpha ,u_{i}^{k,n}\vee \alpha ,u_{i+1}^{k,n}\vee \alpha ) +\Delta t \mathcal{R}_i^{k,n}\\
 &
 =u_i^{k,n}\vee \alpha-\lambda\Delta_{-}(\mathcal{F}_{i+1/2}^{k,n}(u_i^{k,n}\vee \alpha,u_{i+1}^{k,n}\vee \alpha))+\Delta t \mathcal{R}_i^{k,n}. \label{Mon1}
 \end{align}
 Similarly,
 \begin{align}
     \label{Mon2}
u_i^{k,n+1} &\geq u_i^{k,n}\wedge \alpha-\lambda\Delta_{-}(\mathcal{F}_{i+1/2}^{k,n}(u_i^{k,n}\wedge \alpha,u_{i+1}^{k,n}\wedge \alpha))+\Delta t \mathcal{R}_i^{k,n}.
 \end{align}
 Further, once again invoking the monotonicity of $\mathcal{H}_{C}^k$ in its last three arguments, we have, 
 \begin{align} \nonumber
 \alpha & = \alpha -\lambda\Delta_{-}(\mathcal{F}_{i+1/2}^{k,n}(\alpha,\alpha)) + \lambda\Delta_{-}(\mathcal{F}_{i+1/2}^{k,n}(\alpha,\alpha))  \\
 &= \mathcal{H}_{C}^k(x_{i-1/2},x_{i+1/2},c^{k,n}_{i-1/2},c^{k,n}_{i+1/2},\alpha,\alpha, \alpha)  +\lambda\Delta_{-}(\mathcal{F}_{i+1/2}^{k,n}(\alpha,\alpha)) \nonumber \\
 \label{Mon3}
 &\le 
 u_i^{k,n}\vee \alpha-\lambda\Delta_{-}(\mathcal{F}_{i+1/2}^{k,n}(u_i^{k,n}\vee \alpha,u_{i+1}^{k,n}\vee \alpha))+\lambda\Delta_{-}(\mathcal{F}_{i+1/2}^{k,n}(\alpha,\alpha)). 
\end{align}
Similarly,
\begin{align}
\label{Mon4}
\alpha &\geq u_i^{k,n}\wedge \alpha-\lambda\Delta_{-}(\mathcal{F}_{i+1/2}^{k,n}(u_i^{k,n}\wedge \alpha,u_{i+1}^{k,n}\wedge \alpha)) +\lambda\Delta_{-}(\mathcal{F}_{i+1/2}^{k,n}(\alpha,\alpha)).
\end{align}
Now setting
\begin{align*}
	a&=u_{i}^{k,n+1}-\alpha,\\[2mm] 
			c &=-\lambda\Delta_{-}(\mathcal{F}_{i+1/2}^{k,n}(\alpha,\alpha))+\Delta t \mathcal{R}_i^{k,n} ,\\[2mm] 
			b&=\lambda\Delta_{-}(\mathcal{F}_{i+1/2}^{k,n}(u_i^{k,n}\wedge\alpha,u_{i+1}^{k,n}\wedge\alpha))+\abs{u_i^{k,n}-\alpha}-\lambda\Delta_{-}(\mathcal{F}_{i+1/2}^{k,n}(u_i^{k,n}\vee\alpha,u_{i+1}^{k,n}\vee\alpha)),
		\end{align*}
and then, subtracting \eqref{Mon4} from \eqref{Mon1}, we have $a \leq b+c$, and subtracting \eqref{Mon3} from \eqref{Mon2}  we get
 $-a \leq b-c.$
 Now, the discrete entropy inequality follows from the conditional inequality
\begin{align*}
a\leq b+c \quad \text{and }-a \leq b-c \implies \abs{a} \leq b+\sgn(a)c.
 \end{align*}}
\end{enumerate}
\end{proof}
{\color{black}The following is a direct consequence of the above lemma and the uniqueness result Lemma~\ref{lemma:kuz}.}
  \begin{corollary}[Regularity of the entropy solution]\label{tC}
For $0\leq t\leq T$, 
{\color{black} there exists a unique entropy solution} of the IVP~\eqref{PDE_NL_k}, \eqref{eq:u11A} satisfying the following:
\begin{align*}
\boldsymbol{u}(t,\dott) &\in [0,1]^{N},\\
\norma{\boldsymbol{u}(t,\dott)}_{(L^{1}(\R))^{N}}&=\norma{\boldsymbol{u}_0}_{(L^{1}(\R))^{N}},\\
\sum_{k=1}^{N} \TV(u^k(t,\dott)) &\leq 
 \exp({\color{black}\mathcal{C}_{TV}}t) \sum_{k=1}^{N} \TV(u^k_0)  + \exp({\color{black}\mathcal{C}_{TV}}t)-1, \\
\norma{\boldsymbol{u}(t_2,\dott)-\boldsymbol{u}(t_1,\dott)}_{(L^1(\R))^{N}} &\leq N{\color{black}\mathcal{C}_{L}}\abs{t_2-t_1},\quad 0\leq t_1, t_2 \leq T.
\end{align*}
\end{corollary}
{\color{black}
\section{Nonlocal to local limits}\label{NLL}
In this section we prove that the entropy solution of the nonlocal system
\begin{align}\label{PDE_NL_k_linear}
\partial_t(u^{k,\eta})+\partial_x(u^{k,\eta}\nu^k( u^{k,\eta}*\omega_{\eta}))&=R^k(\boldsymbol{u}^{\eta}, \boldsymbol{u}^{\eta} \circledast \omega_{\eta}),\, \,&&(t,x) \in Q_T, k\in \mathcal{N},\\\label{intNL}
{u}_0^k&=u^k(x,0), &&x\in\R, \, \, k\in \mathcal{N},
\end{align}
converges to the entropy solution of the local system, namely,
\begin{align}\label{PDE_NL_k_linearL}
\partial_t(u^k)+\partial_x(u^k\nu^k( u^k))&=R^k (\boldsymbol{u}, \boldsymbol{u}),  &&(t,x) \in Q_T, k\in \mathcal{N},
\\\label{intL}
{u}_0^k&=u^k(x,0), &&x\in\R, \, \, k\in \mathcal{N},
\end{align} in a suitable topology
as $\omega_{\eta}\rightarrow \delta_{x=0}$, where $R^k(\boldsymbol{u}, \boldsymbol{u})$ {\color{black} (cf.~\eqref{eq:Rk}--\eqref{eq:Sk})} is same as the source term taken in \cite[Sec.~3]{HR2019}.

In comparison to the setup of \eqref{PDE_NL_k}, we assume that for $k \in \mathcal{N}$, $f^k(u)=u$, the velocity functions $\nu^k$ is independent of space variable with $D \nu^k \leq 0$ in $[0,1]$ {\color{black} and the kernel $\omega_\eta$ is convex}. 
Consequently, the source term $R^k$ is independent of the space variable. 

To this end consider the convective part of \eqref{PDE_NL_k_linear}:
\begin{align}\label{PDE_NL_k_linearC}
    v^{k,\eta}_t + (v^{k,\eta} \nu^k (v^{k,\eta} * \omega_\eta))_x = 0, \,\,v^{k}(0,x)=v_0^{k}(x), \, x \in \R, \, \, k\in \mathcal{N},
\end{align}
where $v^{k,\eta} = v^{k,\eta}(t,x)$, is its unique entropy solution satisfying \eqref{kruz2} with $R^k=0$, and  with 
  $\boldsymbol{v}_0\in ((L^1 \cap \BV) (\R;[0,1]))^{N}$, see also \cite{ACT2015,ACG2015,CK24,keimer2017existence}. Let $$\mathcal{S}^{k,\eta}_{C}(t, v_0^k) = v^{k,\eta}(t,x)$$ denote the solution operator at time $t$ with initial data $v_0^k$ of the convective part \eqref{PDE_NL_k_linearC}.  In view of \cite[Thm.~1.1]{CCMS2023}, for each $k \in \mathcal{N},$ we have:
\begin{align}\label{Con_TVD} \TV \big( \mathcal{S}^{k,\eta}_{C}(t, v^k_0) * \omega_\eta \big) \leq \TV (v^k_0 * \omega_\eta).
\end{align}
Next, consider a coupled system with a source term with the $k$th equation given by:
\begin{align} \label{eq:Source}
w^{k,\eta}_t = R^k \big( \boldsymbol{w}^{\eta}, (\boldsymbol{w}^{\eta} * \omega_{\eta})\big),\,\,\,w^{k}(0,x)=w_0^{k}(x), x \in \R, k\in \mathcal{N},
\end{align} 
with $\boldsymbol{w}_0\in ((L^1 \cap \BV) (\R;[0,1]))^{N}$ and 
where $\boldsymbol{w}^{\eta} = (w^{1,\eta}, \dots, w^{N,\eta})$. Let 
$$
\mathcal{S}^{k,\eta}_{\text{source}}(t, w_0^k) = w^{k,\eta}(t,x)
$$ 
denote the solution operator at time $t$ with initial data $\boldsymbol{w}_0$. Repeating the convergence analysis of the numerical scheme presented in Section \ref{exis} and the uniqueness result (cf.~Thm.~\ref{uniqueness}) for the case with zero convective flux, we get that the IVP \eqref{eq:Source}  admits a unique entropy solution $\boldsymbol{w}^{\eta}\in(\lip([0,T];L^1(\R;[0,1]))\cap L^{\infty}([0,T]; \BV(\R)))^{N}$. Note that the uniqueness can also be established without the entropy condition using the ODE theory in Banach spaces, but here we use the entropy approach for a self-contained proof, see also \cite{CK24}.

\begin{lemma}The entropy solution $\boldsymbol{w}^{\eta}$ of IVP \eqref{eq:Source} satisfies \begin{align} \label{eq:STVB}
\sum_{k \in \mathcal{N}}  \TV \big( \mathcal{S}^{k,\eta}_{\text{source}}(t, w^k_0) * \omega_\eta \big) \leq \exp \big( 12t \, \abs{\boldsymbol{S}}_{(\lip({[0,1]^4}))^N} \big) \sum_{k \in \mathcal{N}}  \TV (w^k_0).
\end{align}
\end{lemma}
\begin{proof}
 Consider numerical approximations to \eqref{eq:Source} as done before (cf.~Section \ref{exis}), that the approximations satisfy (cf.~\eqref{TVC}--\eqref{TVC1}),
\begin{align*}
\TV(w^{k,n+1,\eta})
&\leq \TV(w^{k,n,\eta})+ 4\abs{\boldsymbol{S}}_{(\lip({[0,1]^4}))^N}\D t \bigg(  \TV(w^{k,n,\eta}) + \TV(w^{k-1,n,\eta})+ \TV(w^{k+1,n,\eta})\bigg). 
\end{align*} 
Summing over $k \in \mathcal{N}$, we obtain:
\begin{align*}
\sum_{k \in \mathcal{N}} \TV(w^{k,n+1,\eta}) &\leq \sum_{k \in \mathcal{N}} \TV(w^{k,n,\eta}) +12\abs{\boldsymbol{S}}_{(\lip({[0,1]^4}))^N}\D t \sum_{k \in \mathcal{N}}\TV(w^{k,n,\eta})   \\
&\leq (1+12\abs{\boldsymbol{S}}_{(\lip({[0,1]^4}))^N}\D t) \sum_{k \in \mathcal{N}} \TV(w^{k,n,\eta})   \\
&\leq \exp \big( 12 n \Delta t \, \abs{\boldsymbol{S}}_{(\lip({[0,1]^4}))^N} \big) \sum_{k \in \mathcal{N}} \TV(w_0^{k}).
\end{align*}

Furthermore, for any function $z$, since $\TV (z*\omega_{\eta}) = \TV (z)$, we get that entropy solution, which is the limit of this numerical approximation satisfies:
\begin{align*} \sum_{k \in \mathcal{N}} \TV \big( (w^{k,\eta} * \omega_{\eta})(t,\dott) \big) \leq \exp \big( 12t \, \abs{\boldsymbol{S}}_{(\lip({[0,1]^4}))^N} \big) \sum_{k \in \mathcal{N}} \TV (w^k_0),\end{align*} proving the lemma.
\end{proof}
The lemma implies that $\mathcal{S}^{k,\eta}_{\text{source}}(\dott, w^k_0) * \omega_\eta $ is uniformly total variation bounded in $[0,T]$ and the bound is  independent of $\eta$. In the following lemma,  we now show that a similar uniform bound on $u^{k,\eta}*\omega_{\eta}$ can be obtained on the entropy solution of the system \eqref{PDE_NL_k_linear} as well.

\begin{lemma}
The entropy solution $\boldsymbol{u}^{\eta}$  of \eqref{PDE_NL_k_linear}--\eqref{intNL} with $\boldsymbol{u}_0\in ((L^1 \cap \BV) (\R;[0,1]))^{N},$ satisfies
\begin{align}\label{conv_bv}
\sum_{k \in \mathcal{N}} \TV(u^{k,\eta}*\omega_{\eta})(t,\dott) \leq \exp \big(12 t \, \abs{\boldsymbol{S}}_{(\lip({[0,1]^4}))^N} \big) \sum_{k \in \mathcal{N}} \TV(u^k_0*\omega_{\eta}).
\end{align}
\end{lemma}
\begin{proof}
To prove the lemma, we employ the so-called source-splitting technique \cite{HKLR2010}. Define $t^n = n \Delta t$ and a sequence of functions $\{ \boldsymbol{u}^{n,\eta} \}$ in $L^1(\R;\R^N)$ by:
\begin{align*} 
u^{k,0,\eta} &:= u^k_0, \\
 {u}^{k,n+1,\eta} &:= \mathcal{S}^k_{C} \big( \Delta t, \mathcal{S}^k_{\text{source}}(\Delta t, u^{k,n,\eta}) \big),  \quad \text{for } k \in \mathcal{N}. 
 \end{align*}
Consider the splitting approximations $\boldsymbol{u}^{\eta}_{\D t}=(u^{1,\eta}_{\D t},\ldots,u^{N,\eta}_{\D t})$ to the system of nonlocal PDEs \eqref{PDE_NL_k_linear} defined by:
\begin{align*} u^{k,\eta}_{\Delta t}(t,\dott) :=
\begin{cases}\mathcal{S}^k_{C} \big( 2(t - t_n), u^{k,n,\eta} \big)
, & t \in [t_n, t_{n+1/2}), \\
\mathcal{S}^k_{\text{source}} \big( 2(t - t_n), S^k_{C}(\Delta t, u^{k,n,\eta}) \big), & t \in [t_{n+1/2}, t_{n+1}), \quad \text{for } k\in \mathcal{N}. 
\end{cases} 
\end{align*}
Clearly, $\boldsymbol{u}^{\eta}_{\Delta t} \in (\lip([0,T];L^1(\R;[0,1])))^{N}$, and for each $t\in[0,T]$, from \eqref{TV_u}, 
$ \sum_{k\in\mathcal{N}}\TV (u^{k,\eta}_{\Delta t}(t,\dott))$ is uniformly bounded with the uniform bound being independent of $\D t$, but could potentially depend on $\eta$.  {\color{black}Consequently, using BV compactness, as $\Delta t \rightarrow 0$, up to a subsequence, the sequence of functions $\boldsymbol{u}^{\eta}_{\Delta t}$ generated by the splitting algorithm converges pointwise almost everywhere and in $(L^1_{\text{loc}}([0,T] \times \mathbb{R}))^N$ to a function $\boldsymbol{u}^\eta$. This limit can be shown to be the unique entropy solution of the nonlocal system \eqref{PDE_NL_k_linear}--\eqref{intNL}, using a gluing argument as detailed in \cite[p.~216]{HR2015}. Now, by uniqueness of the entropy solution, infact the entire sequence $\boldsymbol{u}^{\eta}_{\Delta t}$ converges to the entropy solution $\boldsymbol{u}^\eta$.}

On the other hand, from \eqref{Con_TVD} and \eqref{eq:STVB}, we have,
\begin{align} \label{eq:TVB_split}
 \sum_{k\in \mathcal{N}} \TV ((u_{\Delta t}^k * \omega_\eta)(t,\dott)) \leq \exp \big( 12 t \, \abs{\boldsymbol{S}}_{(\lip({[0,1]^4}))^N} \big) \sum_{k\in \mathcal{N}} \TV (u_0^k),
\end{align} 
and hence, the limit, viz., the entropy solution $\boldsymbol{u}^{\eta}$ of \eqref{PDE_NL_k_linear}--\eqref{intNL} satisfies the claim of the lemma.
\end{proof}

\begin{theorem}
Let $\boldsymbol{u}^{\eta}$ be the entropy solution of the  system \eqref{PDE_NL_k_linear}--\eqref{intNL}. As the  support of the kernel $\eta \rightarrow 0$, $\boldsymbol{u}^{\eta}$ converges to the entropy solution $\boldsymbol{u}$ of the local conservation law \eqref{PDE_NL_k_linearL}--\eqref{intL}
in weakly* in $(L^\infty([0,T] \times \mathbb{R}))^N$.
\end{theorem}
\begin{proof}
The uniform total variation bound obtained in \eqref{conv_bv} implies that as $\eta \to 0$, up to a subsequence, 
$\boldsymbol{u}^{\eta}\circledast \omega_{\eta} \to \boldsymbol{u}$ strongly in $(L^1_{\text{loc}}([0,T] \times \mathbb{R}))^N$. To complete the proof of the theorem, we show that the limit $\boldsymbol{u}$ is the entropy solution of the local conservation law \eqref{PDE_NL_k_linearL}--\eqref{intL}, and that $\boldsymbol{u}$ is indeed the weak* limit of $\boldsymbol{u}^{\eta} $ which we show by modifying the arguments of \cite{CCMS2023}.

To show that $\boldsymbol{u}$ is indeed the entropy solution of \eqref{PDE_NL_k_linearL}--\eqref{intL} {\color{black}in the sense of \cite[Def.~3.1]{HR2019}}, we consider the smooth entropy flux pair $(\boldsymbol{\Upsilon}, \boldsymbol{\Xi})$ for \eqref{PDE_NL_k_linearL}, where $\boldsymbol{\Upsilon}=(\Upsilon^1,\ldots,\Upsilon^N)$ and $\boldsymbol{\Xi}=(\Xi^1,\ldots,\Xi^N)$. We prove that for each $ k \in \mathcal{N}$, {\color{black}$u^k$ satisfies the entropy condition,
\begin{multline*}
\int_{Q_T}\big(\Upsilon^k(u^{k})\phi_t+ \Xi^k(u^{k})\phi_x\big)  \d t \d x+\int_{\R} u_0^{k}(x)\phi(0,x) \d x 
\\ +\int_{Q_T}D\Upsilon^k(u^{k}) R^k(\boldsymbol{u},\boldsymbol{u})  \phi \d t \d x\geq 0,
\end{multline*} 
for all nonnegative $\phi \in C_c^{\infty}([0,T) \times \R)$.
To this end, we consider} 
\begin{align}\label{De}
\begin{split}
D^{\boldsymbol{\Upsilon},k}_{\eta}(\phi):= & \int_{Q_T}\Upsilon^k(u^{k,\eta}*\omega_{\eta})\phi_t+ \Xi^k(u^{k,\eta}*\omega_{\eta})\phi_x  \d t \d x\\ & \quad  +\int_{\R} (u_0^{k}*\omega_{\eta})(x)\phi(0,x) \d x 
  \\ & \quad +  \int_{Q_T}D\Upsilon^k(u^{k,\eta}*\omega_{\eta}) R^k(\boldsymbol{u}^\eta\circledast\omega_{\eta},\boldsymbol{u}^\eta\circledast\omega_{\eta})  \phi \d t \d x,
 \end{split}
\end{align}
 and prove that \begin{align}\label{Dr}
   \lim\limits_{\eta \rightarrow 0}  D^{\boldsymbol{\Upsilon},k}_{\eta}(\phi) \geq 0, \quad \text{for all nonnegative }\phi \in C_c^{\infty}([0,T) \times \R).
\end{align}

Convolving \eqref{PDE_NL_k_linear} with $ \omega_{\eta} $, and then multiplying the result with $ D\Upsilon^k(u^{k,\eta}*\omega_{\eta}) $, we get:
\begin{align*}
&\partial_t \Upsilon^k(u^{k,\eta}*\omega_{\eta}) + D\Upsilon^k(u^{k,\eta}*\omega_{\eta}) \partial_x( \big(\nu^k(u^{k,\eta}*\omega_{\eta}) u^{k,\eta}\big) * \omega_{\eta})\\
&\quad-D\Upsilon^k(u^{k,\eta}*\omega_{\eta})(R^k(\boldsymbol{u}^\eta,\boldsymbol{u}^\eta\circledast\omega_{\eta})*\omega_{\eta})=0.
\end{align*}

Next, we multiply the above equation with the test function $ \phi $, integrate over $Q_T$ and then apply the integration by parts formula to obtain
\begin{align}\label{Dw}
\begin{split}
&\int_{Q_T}\Upsilon^k(u^{k,\eta}*\omega_{\eta})\phi_t + \partial_x(D\Upsilon^k(u^{k,\eta}*\omega_{\eta})\phi) ([\nu^k(u^{k,\eta}*\omega_{\eta}) u^{k,\eta}] * \omega_{\eta} )\d x\d t\\
&+\int_{\R}\Upsilon^k(u_0^{k}*\omega_{\eta})(x)\phi(0,x)\d x-\int_{Q_T}D\Upsilon^k(u^{k,\eta}*\omega_{\eta})(R^k(\boldsymbol{u}^{\eta},\boldsymbol{u}^{\eta}\circledast\omega_{\eta})*\omega_{\eta})\d x\d t=0 .
\end{split}
\end{align}
Combining  \eqref{De} and \eqref{Dw}, we get
\begin{align}\label{De1}
\begin{split}
D^{\boldsymbol{\Upsilon},k}_{\eta}(\phi) &= \left( \int_{Q_T} \Xi^k\left( u^{k,\eta} * \omega_{\eta} \right) \phi_x \, \mathrm{d} t \, \mathrm{d} x \right) \\
&\quad - \left( \int_{Q_T} \partial_x \left( D\Upsilon^k\left( u^{k,\eta} * \omega_{\eta} \right) \phi \right) \left[ \nu^k\left( u^{k,\eta} * \omega_{\eta} \right) u^{k,\eta} \right] * \omega_{\eta} \, \mathrm{d} x \, \mathrm{d} t \right) \\
&\quad + \left( \int_{Q_T} D\Upsilon^k\left( u^{k,\eta} * \omega_{\eta} \right) R^k\left( \boldsymbol{u}^\eta \circledast \omega_{\eta}, \boldsymbol{u}^\eta \circledast \omega_{\eta} \right) \phi \, \mathrm{d} t \, \mathrm{d} x \right) \\
&\quad - \left( \int_{Q_T} D\Upsilon^k\left( u^{k,\eta} * \omega_{\eta} \right) \left( R^k\left( \boldsymbol{u}^{\eta}, \boldsymbol{u}^{\eta} \circledast \omega_{\eta} \right) * \omega_{\eta} \right) \phi \, \mathrm{d} x \, \mathrm{d} t \right) \\
&:= I^{\eta}_{D_1} + I^{\eta}_{D_2}.
\end{split}
\end{align}
Arguing as in \cite[Thm.~1.2, cf.~(4.4), (4.7), and (4.8)]{CCMS2023}, $I^{\eta}_{D_1}$, coming from the convective part satisfies $\lim\limits_{\eta \rightarrow 0 }I^{\eta}_{D_1} \geq 0$, for all nonnegative $\phi \in C_c^{\infty}([0,T) \times \R)$.

We now deal with $I^{\eta}_{D_2}$. For any {\color{black}two} functions, $Z \in L^1([0,T]\times \R)$ with $|Z|_{L^\infty_t\BV_x}< \infty$, and $z\in L^1([0,T]\times \R) \cap L^{\infty}([0,T]\times \R)$, using the properties of the convolution, we first note the following: 
\begin{align}\label{zZ}
\begin{split}
&\int_{Q_T} \left| \left( \left( Z z \right) * \omega_\eta \right)(t,x) - Z(t,x) \left( z * \omega_\eta \right)(t,x) \right| \, \d t \d x \\
&= \int_{Q_T} \left| \int_{\mathbb{R}} z(t,y) \left[ Z(t,y) - Z(t,x) \right] \omega_\eta(x-y) \, \mathrm{d} y \right| \, \d t \d x \\
&\leq \int_{Q_T} \int_{\mathbb{R}} \left| z(t,y) \left[ Z(t,y) - Z(t,x) \right] \omega_\eta(x-y) \right| \, \d y  \d t \d x \\
&\leq \eta T |Z|_{L^\infty_t \mathrm{BV}_x}.
\end{split}
\end{align}
Now, referring to \eqref{eq:Sk} and \eqref{De1}--\eqref{zZ}, for every $k\in\mathcal{N}\setminus N
$, we conclude the following estimates with $a=u^k,b=u^{k+1},A=u^k*\omega_{\eta},B=u^{k+1}*\omega_{\eta},$
\begin{align*}
\begin{split}
&\int_{Q_T}\left| \left( \left( \nu^{k+1}(B) - \nu^k(A) \right)^+ a \right) * \omega_{\eta}(t,x) - \left( \left( \nu^{k+1}(B) - \nu^k(A) \right)^+ A(t,x) \right) \right| \d t \d x\\
&\quad \le \eta \sum_{i=k}^{k+1} \|\nu^i\|_{\text{Lip}(\mathbb{R})} |u^i * \omega_\eta|_{L^\infty_t \BV_x}, \\[2mm]
&\int_{Q_T} \left| \left( \left( \nu^{k+1}(B) - \nu^k(A) \right)^- b \right) * \omega_{\eta}(t,x) - \left( \left( \nu^{k+1}(B) - \nu^k(A) \right)^- B(t,x) \right) \right| \d t \d x
\\
&\quad \le \eta \sum_{i=k}^{k+1} \|\nu^i\|_{\text{Lip}(\mathbb{R})} |u^i * \omega_\eta|_{L^\infty_t \BV_x},
\end{split}
\end{align*}
 with $z=a,Z=(\nu^{k+1}(B)-\nu^k({ }A))^+$, and $z=b,Z=(\nu^{k+1}(B)-\nu^k({ }A))^-$ respectively. Consequently,  as $\eta \rightarrow 0$, we get $I^{\eta}_{D_2} \rightarrow 0$ for all nonnegative $\phi \in C_c^{\infty}([0,T) \times \R)$. This implies that the limit $\boldsymbol{u}$ is indeed the entropy solution for \eqref{PDE_NL_k_linearL}--\eqref{intL} in the sense of \cite[Def.~3.1]{HR2019}. The uniqueness of the entropy solutions to the weakly coupled system \eqref{PDE_NL_k_linearL}--\eqref{intL}(see \cite[Thm.~3.2]{HR2019}) implies that, in fact the entire sequence {\color{black}$\boldsymbol{u}^{\eta}\circledast \omega_{\eta}$} converges to the entropy solution $\boldsymbol{u}$ of \eqref{PDE_NL_k_linearL}--\eqref{intL}.

{\color{black} Note that, since $\boldsymbol{u}^{\eta} \in [0,1]^N $ (cf.~Cor.~\ref{tC}), the family $ \{ \boldsymbol{u}^{\eta} \} $ is pre-compact in $ (L^{\infty}([0,T] \times \mathbb{R}))^N$ endowed with the $\text{weak}^*$ topology.  
We can now use the properties of convolution, as in \cite[Thm.~1.3, Step 3, p.~16]{CCMS2023}, to show that any subsequential limit of the family $ \{ \boldsymbol{u}^{\eta} \} $ in the $\text{weak}^*$ topology
is the entropy solution of \eqref{PDE_NL_k_linearL}--\eqref{intL}. Hence, $ \boldsymbol{u}^{\eta} \xrightarrow{*} \boldsymbol{u} $ weakly$^*$ in $ (L^{\infty}([0,T] \times \mathbb{R}))^N $, completing the proof.}
\end{proof}}


\section{A relative entropy estimate and convergence rate}\label{sec:error}
{\color{black} The aim of this section is to establish the $L^1$ difference between the numerical approximation $\boldsymbol{u}_{\Delta}$ and the entropy solution $\boldsymbol{u}$. This will be achieved by  estimating each term of the right-hand side of \eqref{est:kuz}, in particular $\Lambda^k_{\epsilon,\epsilon_0} (u^{k}_{\Delta},u^k)$. Further, we introduce some notations, similar to that of \cite{HR2015,AHV2023,GTV2022}, {\color{black} in order to try to keep the notations slightly simpler for the upcoming proofs.}} For $(i,n,k)\in \Z \times \mathcal{N}_T\times \mathcal{N},\alpha \in \R$, and $(t,x)\in Q_T$, we define
\begin{enumerate}[(i)]
    \item $\eta_{i}^{k,n}(\alpha):=\abs{u_i^{k,n}-\alpha}$,
    \item $p_i^{k,n}(\alpha):=G(u_i^{k,n},\alpha)=\sgn(u_i^{k,n}-\alpha) (f^k(u_i^{k,n})-f^k(\alpha))$,
    \item $
\mathcal{U}^{k}_{\Delta}(t,x)
    := \nu^k(x,\omega_{\eta}*u^{k}_{\Delta}(t))(x)$,
    \item 
${\color{black}\mathcal{R}}_{\boldsymbol{u},\Delta}^{k}(t,x)
:=R^k(x,\boldsymbol{u}_{\Delta}(t,x),(\boldsymbol{u}_{\Delta} \circledast \omega_{\eta})(t,x))$.
\end{enumerate}
\begin{lemma}\label{Lem:est_Lambda}
For $k \in \mathcal{N}$, 
\begin{align}\label{est:Lambda}
-\Lambda^k_{\epsilon,\epsilon_0} (u^{k}_{\Delta},u^k)\leq \mathcal{L}\left(\frac{\D x}{\epsilon}+\frac{\D t}{\epsilon_0}+ \D t\right),
\end{align}
where $\mathcal{L}$ depends on {\color{black}$T,N, \norma{D_1 \nu^k}_{L^{\infty}(\R \times[0,1])},\norma{D_2 \nu^k}_{L^{\infty}(\R \times[0,1])},\norma{D_{11} \nu^k}_{L^{\infty}(\R \times [0,1])}$, $\norma{D_{12} \nu^k}_{L^{\infty}(\R \times [0,1])}$, $\norma{D_{2} \nu^k}_{L^{\infty}(\R \times [0,1])},\norma{D_{2 1} \nu^k}_{L^{\infty}(\R \times [0,1])},\norma{D_{2 2} \nu^k}_{L^{\infty}(\R \times [0,1])},\abs{\boldsymbol{f}}_{(\lip(\R))^{N}}$, $\abs{\omega_{\eta}}_{\BV(\R)}$, $\norma{\boldsymbol{u}_0}_{(L^{1}(\R))^{N}}, \abs{\boldsymbol{u}_0}_{(\BV(\R))^{N}}$, and $\abs{\boldsymbol{S}}_{(L^{\infty}(\R;\lip({[0,1]^4})))^N}$.}
\end{lemma}
\begin{proof}
{\color{black}Write $\sum\limits_{i,n}=\sum\limits_{i\in \Z}\sum\limits_{n=0}^{N_T-1}$}. For the piecewise constant function $u^{k}_{\D}$ (cf.~\eqref{apx}--\eqref{MF}), the relative entropy can be written as
{\allowdisplaybreaks\begin{align*}
-\Lambda^k_{\epsilon,\epsilon_0}(u^{k}_{\Delta},u^k) 
 &= -\int_{Q_T}\sum_{i,n}   \int_{C_i^{k,n}}\eta_{i}^{k,n}(u^k(s,y))\Phi_t(s,y,t,x) \d t  \d x \d s \d y \\
 &\quad-\int_{Q_T}\sum_{i,n}   \int_{C_i^{k,n}} p_i^{k,n}(u^k(s,y))\mathcal{U}^{k}_{\Delta}(t,x)\Phi_x(s,y,t,x) \d t  \d x \d s \d y \\
 &\quad + \int_{Q_T}\sum_{i,n}  \int_{C_i^{k,n}}\sgn(u_i^{k,n}-u^k(s,y))\\
 &\qquad\qquad\qquad \times f^k(u^k(s,y))\partial_x\mathcal{U}^{k}_{\Delta}(t,x)\Phi(s,y,t,x) \d t  \d x \d s \d y \\
    &\quad- \int_{Q_T}\sum_{i\in\Z}   \int_{C_i}  \eta_{i}^{k,0}(u^k(s,y)) \Phi(s,y,0,x)  \d x\d s  \d y \\
    &\quad+\int_{Q_T} \sum_{i\in\Z}  \int_{C_i}  \eta_{i}^{k,N}(u^k(s,y))  \Phi(s,y,T,x)  \d x\d s  \d y
 \d t\\
    &\quad +\int_{Q_T}\displaystyle\sum_{i,n}  \int_{C_i^{k,n}}\sgn(u_i^{k,n}-u^k(s,y))\\
 &\qquad\qquad\qquad \times {\color{black}\mathcal{R}}_{\boldsymbol{u},\Delta}^{k}(t,x) \Phi(s,y,t,x) \d x\d s  \d y\d t.
    \end{align*}}
Here ${\color{black}\mathcal{R}}_{\boldsymbol{u},\Delta}^{k}(t,x)$ denotes the numerical approximation of ${\color{black}\mathcal{R}}_{\boldsymbol{u}}^{k}(t,x)$, i.e., by replacing $\boldsymbol{u}$ with $\boldsymbol{u}_\Delta$ in ${\color{black}\mathcal{R}}_{\boldsymbol{u}}^{k}$.
Repeat the arguments as in \cite[Lemma~3.3]{AHV2023}, with the discrete entropy inequality \eqref{apx:de} to get:
\begin{align}\label{Lam}
-\Lambda^k_{\epsilon,\epsilon_0}(u^{k}_{\Delta},u) \leq &\mathcal{L}_0 \left( \frac{\D x}{\epsilon}+ \frac{\D t}{\epsilon_0} \right) + \mathcal{E},
\end{align}
where 
$\mathcal{L}_0$  depends on {\color{black}$N,\abs{\boldsymbol{S}}_{(L^{\infty}(\R;\lip({[0,1]^4})))^N},\abs{\omega_{\eta}}_{\BV(\R)},\norma{\boldsymbol{u}_0}_{(L^1(\R))^N},\abs{f^k}_{\lip(\R)},$ \\ $\norma{D_1 \nu^k}_{L^{\infty}(\R \times[0,1])}$, $\norma{D_2 \nu^k}_{L^{\infty}(\R \times[0,1])},\norma{D_{11} \nu^k}_{L^{\infty}(\R \times [0,1])}, \norma{D_{12} \nu^k}_{L^{\infty}(\R \times [0,1])}$, \\ $\norma{D_{2} \nu^k}_{L^{\infty}(\R \times [0,1])},\norma{D_{2 1} \nu^k}_{L^{\infty}(\R \times [0,1])}$, and $\norma{D_{2 2} \nu^k}_{L^{\infty}(\R \times [0,1])}$}. Further, 
{\allowdisplaybreaks \begin{align*}
    \mathcal{E}&:= \int_{Q_T}\sum_{i,n}  \int_{C_i^{k,n}}\sgn(u_i^{k,n}-u^{k}(s,y)){\color{black}\mathcal{R}}_{\boldsymbol{u},\Delta}^{k}(t,x) \Phi(s,y,t,x)\d x \d s \d y \d t  
    \\ &\quad-\Delta t\int_{Q_T}\sum_{i,n}  \sgn(u_i^{k,n+1}-u^{k}(s,y))\int_{C_i}\mathcal{R}_i^{k,n}\Phi(s,y,t^{n+1},x)\d x \d s \d y \\
    &=\int_{Q_T}\sum_{i,n}  \int_{C_i^{k,n}}\Big(\sgn(u_i^{k,n}-u^{k}(s,y)){\color{black}\mathcal{R}}_{\boldsymbol{u},\Delta}^{k}(t,x) \Phi(s,y,t,x)
    \\&\qquad \qquad \qquad \qquad-\sgn(u_i^{k,n+1}-u^{k}(s,y))\mathcal{R}_i^{k,n}\Phi(s,y,t^{n+1},x)\Big)\d x \d s \d y \d t\\
     &=\int_{Q_T}\sum_{i,n}  \int_{C_i^{k,n}}\Big[\sgn(u_i^{k,n}-u^{k}(s,y)){\color{black}\mathcal{R}}_{\boldsymbol{u},\Delta}^{k}(t,x) \Phi(s,y,t,x) \\
     & \qquad \qquad \qquad \qquad -\sgn(u_i^{k,n}-u^{k}(s,y)){\color{black}\mathcal{R}}_{\boldsymbol{u},\Delta}^{k}(t,x) \Phi(s,y,t^{n+1},x)\\
     & \qquad \qquad \qquad \qquad + \sgn(u_i^{k,n}-u^{k}(s,y)){\color{black}\mathcal{R}}_{\boldsymbol{u},\Delta}^{k}(t,x) \Phi(s,y,t^{n+1},x)\\
     & \qquad \qquad \qquad \qquad -\sgn(u_i^{k,n}-u^{k}(s,y))\mathcal{R}_i^{k,n} \Phi(s,y,t^{n+1},x)\\
     & \qquad \qquad \qquad \qquad + \sgn(u_i^{k,n}-u^{k}(s,y))\mathcal{R}_i^{k,n} \Phi(s,y,t^{n+1},x)
    \\&\qquad \qquad \qquad \qquad-\sgn(u_i^{k,n+1}-u^{k}(s,y))\mathcal{R}_i^{k,n}\Phi(s,y,t^{n+1},x)\Big]\d x \d s \d y \d t.
\end{align*}}
{\color{black}Now, we aim to show
\begin{align}\label{E}
    \mathcal{E} \leq \mathcal{L}_4 \left( \D t +\frac{\D t} {\epsilon_0} \right),
\end{align} for a suitable $\mathcal{L}_4$ depending on {\color{black}$T,N, \norma{D_1 \nu^k}_{L^{\infty}(\R \times[0,1])},\norma{D_2 \nu^k}_{L^{\infty}(\R \times[0,1])},\norma{D_{11} \nu^k}_{L^{\infty}(\R \times [0,1])}$, $\norma{D_{12} \nu^k}_{L^{\infty}(\R \times [0,1])}$, $\norma{D_{2} \nu^k}_{L^{\infty}(\R \times [0,1])},\norma{D_{2 1} \nu^k}_{L^{\infty}(\R \times [0,1])},\norma{D_{2 2} \nu^k}_{L^{\infty}(\R \times [0,1])},\abs{\boldsymbol{f}}_{(\lip(\R))^{N}},$\\$ \norma{\boldsymbol{u}_0}_{(L^{1}(\R))^{N}}, \abs{\boldsymbol{u}_0}_{(\BV(\R))^{N}}$, $\abs{\boldsymbol{S}}_{(L^{\infty}(\R;\lip({[0,1]^4})))^N}$, and $\abs{\omega_{\eta}}_{\BV(\R)}$,} so that the lemma follows from \eqref{Lam} with $\mathcal{L}=\mathcal{L}_0+\mathcal{L}_4.$ To this end, we write $\mathcal{E}=\mathcal{E}_1+\mathcal{E}_2+\mathcal{E}_3$, where} 
{\allowdisplaybreaks
\begin{align*}
\mathcal{E}_1&=\int_{Q_T}\sum_{i,n}  \int_{C_i^{k,n}}
\sgn(u_i^{k,n}-u^{k}(s,y))  {\color{black}\mathcal{R}}_{\boldsymbol{u},\Delta}^{k}(t,x) \\
&\qquad\qquad\qquad\qquad\qquad\qquad \times( \Phi(s,y,t,x) - \Phi(s,y,t^{n+1},x) ) \d x \d s \d y \d t, \\[2mm]
 \mathcal{E}_2& = \int_{Q_T}\sum_{i,n}  \int_{C_i^{k,n}} \sgn(u_i^{k,n}-u^{k}(s,y))
( {\color{black}\mathcal{R}}_{\boldsymbol{u},\Delta}^{k}(t,x) - \mathcal{R}_i^{k,n} )\Phi(s,y,t^{n+1},x) \d x \d s \d y \d t,\\[2mm]
  \mathcal{E}_3& = \int_{Q_T}\sum_{i,n}  \int_{C_i^{k,n}} \mathcal{R}_i^{k,n}
(\sgn(u_i^{k,n}-u^{k}(s,y)) 
   -\sgn(u_i^{k,n+1}-u^{k}(s,y)))\\
   &\qquad\qquad\qquad\qquad\qquad\qquad \times\Phi(s,y,t^{n+1},x)\d x \d s \d y \d t.
 \end{align*}}
 {\color{black}{\color{black}From the assumed properties of the mollifier, see Section \ref{uni}}, note that,\begin{align*}
 \int_{Q_T} \left(\int_{C^{n}} \Phi(s,y,t,x_{i+1/2})\d t-\lambda\int_{C_i}\Phi(s,y,t^{n+1},x)\d x \right) \d s \d y&=\mathcal{O} \left(\frac{\D x^2}{\epsilon} +\frac{\D t^2}{\epsilon_0}\right),\\ 
\int_{Q_T}\left(\int_{C_{i+1}}\Phi(s,y,t^{n+1},x)\d x-\int_{C_i}\Phi(s,y,t^{n+1},x)\d x \right) \d s \d y&= \mathcal{O}\left( \frac{\D x^2}{\epsilon}\right),\\
\int_{0}^T\abs{\Theta_{\epsilon_0}(s-t)-\Theta_{\epsilon_0}(s-t^{n+1})} \d s \leq \D t \int_{0}^T\abs{\Theta'_{\epsilon_0}(s-\tau)} \d s &\leq  \D t/ \epsilon_0,\,\,t\in C^n.
\end{align*}}
Details of these estimates can be found in \cite[Sec.~3.3]{HR2015}. Consequently
\begin{align*}
    \mathcal{E}_1 \leq   \norma{\mathcal{R}_{\boldsymbol{u},\Delta}^{k}}_{L^1(Q_T)} \frac{\D t}{\epsilon_0} \leq 2 
\abs{\boldsymbol{S}}_{(L^{\infty}(\R;\lip({[0,1]^4})))^N}\norma{\boldsymbol{u_0}}_{(L^1(\R))^{N}}\frac{\D t}{\epsilon_0}.
\end{align*}
To estimate $\mathcal{E}_2$, we first note that,
for $(n,i)\in \mathcal{N}_T \times \Z$  and $(t,x)\in C_i^{k,n}$, 
\begin{align*}
\left|{\color{black}\mathcal{R}}_{\boldsymbol{u},\Delta}^{k}(t,x)-\mathcal{R}_i^{k,n}\right|
&\le \left|{\color{black}\mathcal{R}}_{\boldsymbol{u},\Delta}^{k}(t,x)-{\color{black}\mathcal{R}}_{\boldsymbol{u},\Delta}^{k}(t^n,x)\right| +\left|{\color{black}\mathcal{R}}_{\boldsymbol{u},\Delta}^{k}(t^n,x)-{\color{black}\mathcal{R}}_{\boldsymbol{u},\Delta}^{k}(t,x_i)\right| \\
&\quad+\left|{\color{black}\mathcal{R}}_{\boldsymbol{u},\Delta}^{k}(t,x_i)-\mathcal{R}_i^{k,n}\right|.
\end{align*}
In what follows, we estimate the summands on the RHS of the above inequality for any $t\in C^{k,n}$.
\begin{align*}
\begin{split}
     \sum_{i\in \Z} &\int_{C_i} {\color{black}\abs{{\color{black}\mathcal{R}}_{\boldsymbol{u},\Delta}^{k}(t,x) -{\color{black}\mathcal{R}}_{\boldsymbol{u},\Delta}^{k}(t^n,x)}} \d x\\
   &\quad \leq 2 \abs{\boldsymbol{S}}_{(L^{\infty}(\R;\lip({[0,1]^4})))^N} \\
   &\qquad \times\sum_{i\in \Z} \int_{C_i} \sum_{l=k-1}^{k+1} 
   \abs{\nu^k(x,(u^{l,\D}*\omega_{\eta})(t,x))- \nu^k(x,(u^{l,\D}*\omega_{\eta})(t^n,x))}\d x
   \\
   & \quad \leq 2 \abs{\boldsymbol{S}}_{(L^{\infty}(\R;\lip({[0,1]^4})))^N} \norma{D_2\nu^k}_{L^\infty(\R\times [0,1])}
   \D t \sum_{l=k-1}^{k+1} {{\color{black}\mathcal{C}_{L}}}  \\
    & \quad =6\Delta t\abs{\boldsymbol{S}}_{(L^{\infty}(\R;\lip({[0,1]^4})))^N} \norma{D_2\nu^k}_{L^\infty(\R\times [0,1])}
   {\color{black}\mathcal{C}_{L}},
   \end{split}
   \end{align*}
   where the last inequality follows from the time estimate for $u^{k}_{\D}$ 
   (cf.~Lemma~\ref{lem:stability}\ref{lem:TE}). Further,
    {\allowdisplaybreaks
  \begin{align*}
   \sum_{i\in \Z} \int_{C_i} &\abs{\mathcal{R}_{\boldsymbol{u},\Delta}^{k}(t^n,x) -\mathcal{R}_{\boldsymbol{u},\Delta}^{k}}(t^n,x_i) \d x
   \\
   & \leq 2 \abs{\boldsymbol{S}}_{(L^{\infty}(\R;\lip({[0,1]^4})))^N} \\
   &\qquad \times\sum_{i\in \Z} \int_{C_i} \sum_{l=k-1}^{k+1} 
     \abs{\nu^k(x,(u^{l,\D}*\omega_{\eta})(t^n,x))- \nu^k(x_i,(u^{l,\D}*\omega_{\eta})(t^n,x_i)}\d x\\
   & \leq 2 \abs{\boldsymbol{S}}_{(L^{\infty}(\R;\lip({[0,1]^4})))^N}  \D x \\
   &\qquad \times\left(3\abs{a_{0}}_{\BV(\R)}+\norma{D_2\nu^k}_{L^\infty(\R\times [0,1])}\abs{\omega_{\eta}}_{\BV(\R)}\sum_{l=k-1}^{k+1} \norma{u^{l,\D}(t^n)}_{L^1(\R)} \right)  \\
    & \leq 2 \abs{\boldsymbol{S}}_{(L^{\infty}(\R;\lip({[0,1]^4})))^N}  \D x \\
    &\qquad\times\left(3\abs{a_{0}}_{\BV(\R)}+\norma{D_2\nu^k}_{L^\infty(\R\times [0,1])}\abs{\omega_{\eta}}_{\BV(\R)}\norma{\boldsymbol{u}^0}_{(L^{1}(\R))^{N}} \right),\\[2mm]
     \sum_{i\in \Z} \int_{C_i}&\Big( \abs{\mathcal{R}_{\boldsymbol{u},\Delta}^{k}}(t^n,x_i) - \mathcal{R}_i^{k,n}\Big)\d x \\
     & \leq 2\D x\abs{\boldsymbol{S}}_{(L^{\infty}(\R;\lip({[0,1]^4})))^N} \sum_{l=k-1}^{k+1}  \sum\limits_{i\in \Z} \abs{\nu^k(x_i,u^{l,\D}*\omega_{\eta})-\nu^k(x_i,C_i^{k,n})}\\
     & \leq 2 \D x\norma{D_2\nu^k}_{L^\infty(\R\times [0,1])}\abs{\boldsymbol{S}}_{(L^{\infty}(\R;\lip({[0,1]^4})))^N} \\&\qquad\times\sum_{l=k-1}^{k+1}  \sum\limits_{i,p\in \Z} u_p^{l,n} \int_{\mathcal{C}_p} \abs{\omega_{\eta}({\color{black}y-x_i})-\omega_{\eta}({x_p-x_i})} \d y\\
  &  \leq 6\abs{\omega_{\eta}}_{\BV(\R)} \norma{D_2\nu^k}_{L^\infty(\R\times [0,1])}\sum_{l=k-1}^{k+1} \norma{u^{l,\D}(t^n)}_{L^{1}(\R)} \D x  
  \\
   & \leq 6\abs{\omega_{\eta}}_{\BV(\R)} \norma{D_2\nu^k}_{L^\infty(\R\times [0,1])}\norma{\boldsymbol{u}^0}_{(L^1(\R))^{N}} \D x,
\end{align*}}
which implies, 
 {\color{black}using the CFL condition\eqref{CFL}},
\begin{align}
\sum_{i\in \Z}\int_{C_i} \left|{\color{black}\mathcal{R}}_{\boldsymbol{u},\Delta}^{k}(t,x)-\mathcal{R}_i^{k,n}\right| \d x \leq  \mathcal{L}_1 \D t, 
\end{align}
where, 
\begin{align*}
    \mathcal{L}_1 &= 6\abs{\boldsymbol{S}}_{(L^{\infty}(\R;\lip({[0,1]^4})))^N}
   {\color{black}\mathcal{C}_{L}}+2 \abs{\boldsymbol{S}}_{(L^{\infty}(\R;\lip({[0,1]^4})))^N} \abs{\omega_{\eta}}_{\BV(\R)} \lambda  \norma{\boldsymbol{u}^0}_{(L^{1}(\R))^{N}} \\
   &\quad+ \lambda\abs{\omega_{\eta}}_{\BV(\R)} \norma{\boldsymbol{u}^0}_{(L^1(\R))^{N}},
\end{align*}
and hence 
$\mathcal{E}_2 \leq \mathcal{L}_1T \D t$. Now, we consider the term $\mathcal{E}_3$ which can be rewritten as
{\allowdisplaybreaks
\begin{align*}
\mathcal{E}_3 &=
\int_{Q_T} \sum_{i\in\Z}\sum_{n=1}^{N_T}\int_{C_i^{k,n}}\sgn(u_i^{k,n}-u^k(s,y))\mathcal{R}_i^{k,n}
\Phi(s,y,t^{n+1},x)  \d x\d s  \d y \d t\\
& \quad - \int_{Q_T} \sum_{i\in\Z}\sum_{n=1}^{N_T}\int_{C_i^{k,n}}\sgn(u_i^{k,n}-u^k(s,y))\mathcal{R}_i^{k,n-1}
\Phi(s,y,t^{n},x)  \d x\d s  \d y \d t\\
&\quad+ \int_{Q_T} \sum_{i\in\Z}\int_{C_i^{k,n}}\sgn(u_i^{k,0}-u^k(s,y))\mathcal{R}_i^{k,0}
\Phi(s,y,t^1,x)  \d x\d s  
\d y \d t\\
&\quad- \int_{Q_T} \sum_{i\in\Z}\int_{C_i^{k,n}} \sgn(u_i^{N}-u^k(s,y))\mathcal{R}_i^{k,n}
\Phi(s,y,t^{N+1},x)  \d x\d s  \d y \d t.
\end{align*}
}
Adding and subtracting the term
we have,
{\allowdisplaybreaks
\begin{align*}
\mathcal{E}_3&=
\int_{Q_T} \sum_{i\in\Z}\sum_{n=1}^{N_T}\int_{C_i^{k,n}}\sgn(u_i^{k,n}-u^k(s,y))\Big(\mathcal{R}_i^{k,n}-\mathcal{R}_i^{k,n-1}\Big)\Phi(s,y,t^{n+1},x) \d x\d s  \d y\\
&\quad+ \int_{Q_T} \sum_{i\in\Z}\sum_{n=1}^{N_T}\int_{C_i^{k,n}}\sgn(u_i^{k,n}-u^k(s,y))\mathcal{R}_i^{k,n-1}\\
&\qquad\qquad \qquad\qquad\qquad\qquad\times
(\Phi(s,y,t^{n+1},x)-\Phi(s,y,t^{n},x))  \d x  \d s \d y\\
&\quad+ \int_{Q_T}  \sum_{i\in\Z}\int_{C_i^{k,n}}\sgn(u_i^{k,0}-u^k(s,y))\mathcal{R}_i^{k,0}
\Phi(s,y,t^1,x)  \d x \d s \d y\\
&\quad- \int_{Q_T}  \sum_{i\in\Z}\int_{C_i^{k,n}}\sgn(u_i^{k,N_T}-u^k(s,y))\mathcal{R}_i^{k,n}
\Phi(s,y,t^{n+1},x) \d x \d s \d y\\
&:=\mathcal{E}_{31}+\mathcal{E}_{32} +\mathcal{E}_{33} +\mathcal{E}_{34}.
\end{align*}}
Note that
\begin{align}
\begin{split}
\label{TER}\abs{\mathcal{R}_i^{k,n}-\mathcal{R}_i^{k,n-1}} & \leq 2 \abs{\boldsymbol{S}}_{(L^{\infty}(\R;\lip({[0,1]^4})))^N}\bigg( \abs{u^{k-1,n}_i-u^{k-1,n-1}_i}+ \abs{u^{k,n}_i-u^{k,n-1}_i} \\
  & \quad  + \abs{u^{k+1,n}_i-u^{k+1,n-1}_i} +\norma{D_2\nu^k}_{L^\infty(\R\times [0,1])}\abs{c^{k,n}_i-c^{k,n-1}_i}\\
&\quad+\norma{D_2\nu^k}_{L^\infty(\R\times [0,1])}\abs{c^{k-1,n}_i-c^{k-1,n-1}_i}\\
&\quad+\norma{D_2\nu^k}_{L^\infty(\R\times [0,1])}\abs{c^{k+1,n}_i-c^{k+1,n-1}_i} \bigg).
\end{split}
\end{align}
Hence, using \eqref{TV_c} and the time estimate (cf.~Lemma~\ref{lem:stability}\ref{lem:TE}) in \eqref{TER}, we have,
\begin{align*}
    \sum_{i\in\Z} \abs{\mathcal{R}_i^{k,n}-\mathcal{R}_i^{k,n-1}} & \leq 24  {\color{black}\mathcal{C}_{L}}  \lambda \abs{\boldsymbol{S}}_{(L^{\infty}(\R;\lip({[0,1]^4})))^N}
\end{align*}
and consequently,
\begin{align*}
    \mathcal{E}_{31} \leq 24 \abs{\boldsymbol{S}}_{(L^{\infty}(\R;\lip({[0,1]^4})))^N} {\color{black}\mathcal{C}_{L}} T \D t.
\end{align*}
Now, we consider
{\allowdisplaybreaks\begin{align*}
   \mathcal{E}_{32}
   &=- \int_{Q_T}  \sum_{i\in\Z}\sum_{n=1}^{N_T}\int_{C_i^{k,n}}\sgn(u_i^{k,n}-u^k(s,y))\mathcal{R}_i^{k,n-1} \\
&\qquad\qquad \times\big(\Phi(s,y,t^{n},x)-\Phi(s,y,t^{n+1},x)\big) \d x  \d s  \d y \d t\\
&\le \int_{0}^T  \sum_{i\in\Z}\sum_{n=1}^{N_T} \int_{C_i^{k,n}}
\abs{\mathcal{R}_i^{k,n-1}}
(\Theta_{\epsilon_0}(s-t^n)-\Theta_{\epsilon_0}(s-t^{n+1}))  \d s  \d y \d t \\
& \leq \D x  \D t \sum_{i,n} \abs{\mathcal{R}_i^{k,n-1}}   \frac{ \D t}{\epsilon_0} \\
&\leq  2T \abs{\boldsymbol{S}}_{(L^{\infty}(\R;\lip({[0,1]^4})))^N} (1+\norma{D_2\nu^k}_{L^\infty(\R\times [0,1])})\norma{\boldsymbol{u}_0}_{(L^1(\R))^{N}}  \frac{ \D t}{\epsilon_0} \\
& :=\mathcal{L}_2\frac{\Delta t}{\epsilon_0}.
\end{align*}}

Finally, we estimate on the remaining boundary terms $\mathcal{E}_{33}$ and $\mathcal{E}_{34}$.  Specifically,
\begin{align*}
      \mathcal{E}_{33}&= \int_{Q_T} \sum_{i\in\Z}\sgn(u_i^{k,0}-u^k(s,y))\mathcal{R}_i^{k,0}
\int_{C_i^{k,n}}\Phi(s,y,t^1,x)  \d x\d s  \d y \\
&\le  \D x \sum_{i\in\Z}\abs{\mathcal{R}_i^{k,0}}  \Delta t \\
&\leq 4 \abs{\boldsymbol{S}}_{(L^{\infty}(\R;\lip({[0,1]^4})))^N}  \norma{\boldsymbol{u}_0}_{(L^1(\R))^{N}} \D t\\ 
&:=\mathcal{L}_3\D t.
\end{align*}
Similarly,
\begin{align}\mathcal{E}_{34}
\le \D x \sum_{i\in\Z}\abs{\mathcal{R}_i^{k,N_T}} \Delta t\le \mathcal{L}_3 \D t.
\end{align}
Collecting the estimates on $\mathcal{E}_1,\mathcal{E}_2$ and $\mathcal{E}_3$, we have
\begin{align*}
    \mathcal{E} \leq \mathcal{L}_4 \left( \D t +\frac{\D t} {\epsilon_0} \right),
\end{align*}
where
\begin{align*}
\mathcal{L}_4&= 2 \abs{\boldsymbol{S}}_{(L^{\infty}(\R;\lip({[0,1]^4})))^N} \norma{\boldsymbol{u_0}}_{(L^1(\R))^{N}}+\mathcal{L}_1T\\&\quad+24  \abs{\boldsymbol{S}}_{(L^{\infty}(\R;\lip({[0,1]^4})))^N} {\color{black}\mathcal{C}_{L}} T+\mathcal{L}_2+2\mathcal{L}_3.
\end{align*}
Finally, the lemma follows from \eqref{Lam}--\eqref{E} with $\mathcal{L}=\mathcal{L}_0+\mathcal{L}_4.$
\end{proof}
\begin{theorem}\label{CR}
    Let $\boldsymbol{u}$ be the entropy solution and $\boldsymbol{u}_{\Delta}$ be the numerical approximation obtained via the marching formula \eqref{MF}. Then the following convergence rate estimate holds:
\begin{align}\label{rate}
    \norma{\boldsymbol{u}(T,\dott)-\boldsymbol{u}_{\Delta}(T,\dott)}_{(L^1(\R))^{N}} = \tilde{\mathcal{C}}_T(\D t)^{1/2},
\end{align}
where
$\tilde{\mathcal{C}}_T$ depends on {\color{black}$T,N, \norma{D_1 \nu^k}_{L^{\infty}(\R \times[0,1])},\norma{D_2 \nu^k}_{L^{\infty}(\R \times[0,1])},\norma{D_{11} \nu^k}_{L^{\infty}(\R \times [0,1])}$, $\norma{D_{12} \nu^k}_{L^{\infty}(\R \times [0,1])}$, $\norma{D_{2} \nu^k}_{L^{\infty}(\R \times [0,1])},\norma{D_{2 1} \nu^k}_{L^{\infty}(\R \times [0,1])},\norma{D_{2 2} \nu^k}_{L^{\infty}(\R \times [0,1])},\abs{\boldsymbol{f}}_{(\lip(\R))^{N}}$, $\norma{\boldsymbol{u}_0}_{(L^{1}(\R))^{N}}, \abs{\boldsymbol{u}_0}_{(\BV(\R))^{N}}$, $\abs{\boldsymbol{S}}_{(L^{\infty}(\R;\lip({[0,1]^4})))^N}$, and $\abs{\omega_{\eta}}_{\BV(\R)}$}, but is independent of $\Delta t.$
\end{theorem}
\begin{proof}
In view of the  Kuznetsov-type estimate
\eqref{est:kuz}  and the estimate on the relative entropy functional \eqref{est:Lambda}, the theorem follows by repeating the arguments of \cite[Thm.~5.2]{AHV2023}.
\end{proof}
{\color{black}
\begin{remark}\normalfont Theorem \ref{CR} demonstrates the rate of convergence being at least 1/2. This result should be regarded as a worst-case estimate, meaning the rate cannot fall below 1/2. An example from \cite{Sab1997} illustrates that, in general, this bound is optimal for local conservation laws, as a rate of 1/2 is achieved in that case. However, in many instances, the method exhibits rates higher than 1/2 {\color{black}(see Figure \ref{fig:table})}.
\end{remark}}
\section{A short note for a general system of nonlocal balance laws}\label{exten}
In this section, 
{\color{black}we briefly describe how to} extend our analysis  to a general system of balance laws \eqref{PDE_NL_kd} and \eqref{eq:u11A}, with 
$R^k$ having a very general form with dependence on all $N$ components of $\boldsymbol{u}$ and $(\tilde{\boldsymbol{\omega}_{\eta}}\circledast\boldsymbol{u})^k$ and $\nu^k$ having dependence on all components of $(\boldsymbol{\omega}_{\eta}\circledast\boldsymbol{u})^k$, making the system strongly coupled through the nonlocal convective part. 
\begin{definition}[Entropy Condition]\label{def:entropy1}
A function $\boldsymbol{u} \in ({\color{black}\operatorname{Lip}}([0,T];L^1(\R;[0,1]))\cap L^{\infty}([0,T]; \BV(\R)))^{N}$  is an entropy
 solution of IVP \eqref{PDE_NL_kd} and \eqref{eq:u11A}, 
if for all $k \in\mathcal{N}$, and for all $\alpha\in \R$, 
\begin{multline} \label{kruz21}
\int_{Q_T}\left|u^k(t,x)- \alpha\right|\phi_t(t,x)  \d{x} \d t  \\ 
+ \int_{Q_T}\sgn (u^k(t,x)-\alpha) \nu^k((\boldsymbol{\omega}_{\eta}\circledast\boldsymbol{u})^k(t,x))
(f^k(u^k(t,x))-f^k(\alpha))\phi_x(t,x) \d{x} \d{t}\\ 
 -\int_{Q_T} f^k(\alpha) (\sgn (u^k(t,x)-\alpha)) \partial_x\nu^k((\boldsymbol{\omega}_{\eta}\circledast\boldsymbol{u})^k(t,x))\phi(t,x)\d{x} \d{t}\\ 
+\int_{\R} \left|u_0^k(x)- \alpha \right|\phi(0,x)  \d{x} \geq \int_{Q_T} \sgn(u^k(t,x)-\alpha) {\color{black}\mathcal{R}}^k_{\boldsymbol{u}}(t,x)\phi(t,x) \d t \d{x}, 
\end{multline}
for all  $0\leq \phi\in C_c^{\infty}([0,T)\times \R)$, where
 ${\color{black}\mathcal{R}}^k_{\boldsymbol{u}}(t,x):= R^k(x,\boldsymbol{u}(t,x),(\boldsymbol{u} \circledast \tilde{\boldsymbol{\omega}_{\eta}})^k(t,x)))$, with $(t,x)\in Q_T$.
\end{definition}

The Kuznetsov-type lemma (cf.~Lemma~\ref{lemma:kuz}) and hence the uniqueness of the entropy solution (cf.~Theorem~\ref{uniqueness}) remain valid, and the proof follows by combining the estimates presented in \cite[Sec.~4]{AHV2023} and Section \ref{uni}. Following the notations in  Section \ref{exis}, the marching algorithm can be modified as
\begin{align*}
u_{i}^{k,n+1} &=\mathcal{H}^k({x_{i-1/2},x_i,x_{i+1/2}},u_{i-1}^{k,n},u_{i}^{k,n},u_{i+1}^{k,n},\boldsymbol{u}_{i}^{n},\boldsymbol{c}^{k,n}_{i-1/2},\boldsymbol{c}^{k,n}_{i+1/2},\boldsymbol{\tilde{c}}^{k,n}_{i})\\
&=u_{i}^{k,n}- \lambda\Big(\mathcal{F}^{k,n}_{i+1/2}({x_{i+1/2}},u_{i}^{k,n},u_{i+1}^{k,n})-\mathcal{F}^{k,n}_{i-1/2}({x_{i-1/2}},u_{i-1}^{k,n},u_{i}^{k,n}) \Big) + \D t \mathcal{R}_i^{k,n},
\end{align*}
where, for every $k\in\mathcal{N},i\in\Z,n\in\{0,1,\dots,N_T\}$, and for fixed $\alpha\in (0,2/3),$
{\allowdisplaybreaks
\begin{align*}
\mathcal{F}^{k,n}_{i+1/2}(a,b)&:= 
\frac{1}{2}\nu^{k}({x_{i+1/2}},\boldsymbol{c}^{k,n}_{i+1/2})\left( f^k(a)   +  f^k(b)\right) -
  \frac{\beta}{2\, \lambda}(b-a),\\[2mm]
\mathcal{R}^{k,n}_{i}&:= R^{k}({x_i},\boldsymbol{u}^{k,n}_{i}, \tilde{\boldsymbol{c}}^{k,n}_{i}),\\[2mm]
{c}_{i+1/2}^{k,n}&:=\sum\limits_{p=0}^{N_{\eta}-1}\zeta^k_{p+1/2}u_{i+p+1}^{k,n}, \\
\zeta_{p+1/2}^k
&:=\int_{p\Delta x}^{(p+1)\Delta x}\omega^k_{\eta}(y)\d y,\quad p\in\mathcal{N}_{\eta}:=\{0,1,\dots,N_{\eta}-1\},\\[2mm]
\tilde{c}_{i+1/2}^{k,n}&:= \sum\limits_{p=0}^{N_{\tilde{\eta}}-1}\tilde{\zeta}^k_{p+1/2}u_{i+p+1}^{k,n},\\
 \tilde{\zeta}^k_{p+1/2}&:=\int_{p\Delta x}^{(p+1)\Delta x}\tilde{\omega}^k_{\tilde{\eta}}(y)\d y,\quad p\in\mathcal{N}_{\tilde{\eta}}:=\{0,1,\dots,N_{\tilde{\eta}}-1\},\\[2mm]
{\boldsymbol{{z}}}^{k,n}_{i+1/2}&=({z}_{ i+1/2}^{k,n},\ldots, {z}_{i+1/2}^{k,n}), \quad z=c,\tilde{c},\\[2mm]
{\boldsymbol{z}}^{k,n}_{i}&=\frac{1}{2}({\boldsymbol{{z}}}^{k,n}_{i+1/2}+{\boldsymbol{{z}}}^{k,n}_{i-1/2}), \quad z=c,\tilde{c}.
\end{align*}}
The existence of the entropy solution can now be proved {\color{black} by establishing the convergence of the above finite volume approximation. 
The proof runs exactly on the similar lines of the Lemma~\ref{lem:stability}, but by 
 borrowing the estimates for the  coupled nonlocal terms in the convective part from \cite[Thm.~3.2]{AHV2023_1} and the estimates from the previous sections for the source part (cf.~Section \ref{uni}). 
 
Furthermore, a Kuznetsov-type lemma for the system \eqref{PDE_NL_kd} can be proved as in Lemma~\ref{lemma:kuz}, again borrowing the estimates for the  convective part (coupled through the nonlocal terms) from \cite[Lemma~4.1]{AHV2023_1} (see also \cite[Lemma~5.1]{AHV2023}) and the estimates for the source part from Lemma~\ref{lemma:kuz} (see pp.~16--17). Lemma~\ref{lemma:kuz} and Lemma~\ref{lem:stability} imply the well-posedness for the system \eqref{PDE_NL_kd}.

Finally, the estimates on the relative entropy can be obtained  as in Lemma~\ref{Lem:est_Lambda}, yet again borrowing the estimates for the convective part (coupled through the nonlocal terms) from \cite[Lemma~5.1]{AHV2023_1} (see also \cite[Lemma~3.3]{AHV2023}).

The above results imply a convergence rates estimates of the order $\mathcal{O}(\sqrt{\D t})$ (see Theorem~\ref{CR}) for the above numerical scheme.}
\section{Numerical Experiments}
\label{num}
We now present some numerical experiments to illustrate the theory presented in the previous section, and {\color{black}the ``nonlocal to local" dynamics of the system as the convolution radius $\eta\rightarrow0$}. Throughout the section, we choose $\beta=0.3333$, and $\lambda=0.1286$ to satisfy the CFL condition \eqref{CFL}. We consider a road
with two lanes, each with its own velocity function, with the second lane being the fastest.  It is also assumed that the lanes are homogeneous and the
traffic on the road is unidirectional. To this end, we consider the IVP \eqref{PDE_NL_k}, \eqref{eq:u11A} with $N=2$ with
$g^k(x)=1-x,$
\begin{align*}
  {\color{black} \nu^1(x)=1.5g^1(x), \nu^2(x)=2.5g^1(x),  \mu^1(x)=\mu^2(x)=L(\eta+x)\mathbbm{1}_{(-\eta,0)}(x),}
\end{align*}
for $x\in\R$. In addition, 
\begin{figure}[ht!]
 \centering
 \begin{subfigure}{.45\textwidth}
\includegraphics[width=\textwidth,keepaspectratio]{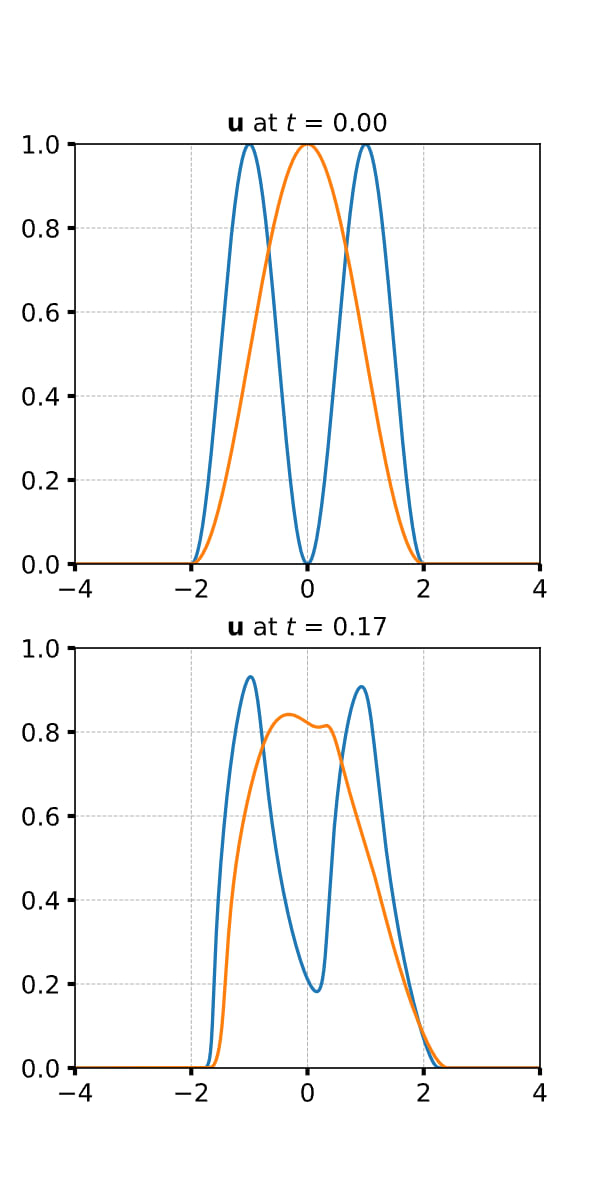}
\end{subfigure}
\begin{subfigure}{.45\textwidth}
\includegraphics[width=\textwidth,keepaspectratio]{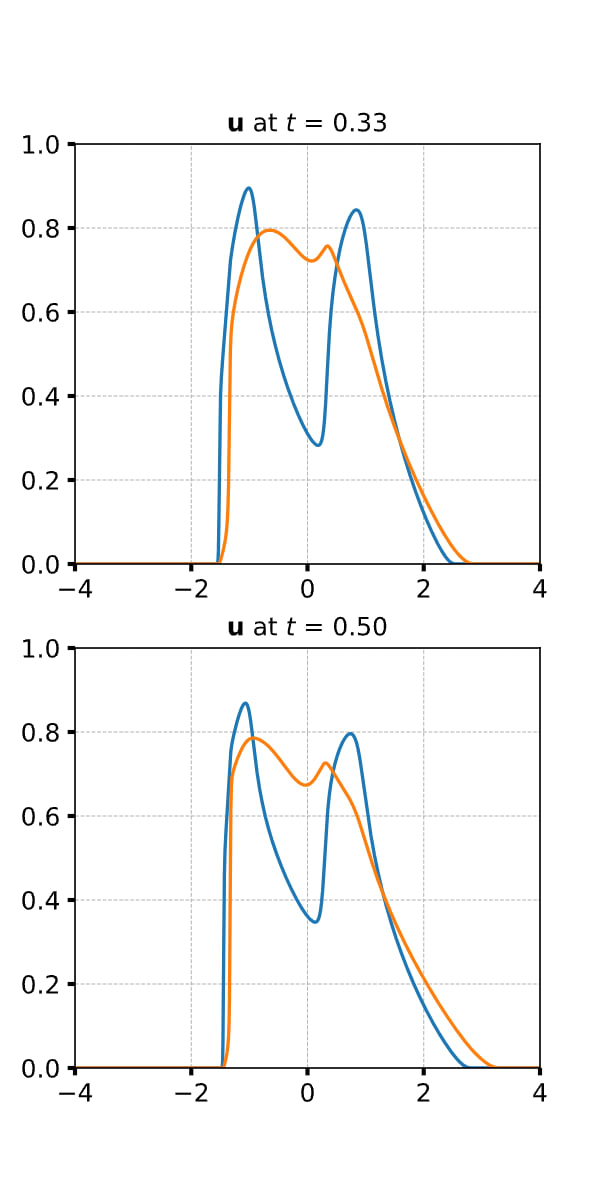}
\end{subfigure}
\hfill
\caption{Solution to the nonlocal conservation law~\eqref{PDE_NL_k}, \eqref{eq:ex1} on the domain $[-4,\,4]$ at times $t =
    0.00,\; 0.017,\;0.33, \: 0.5$, with mesh size $\Delta x=0.00078125$. $u^1$({\color{blue}\full}),\,\,$u^2$({\color{orange}\full}).}
  \label{fig:ex212}
\end{figure}$L$ is chosen such that $\int_{\R}\mu^k(x)\d{x}=1.$ This PDE fits the hypotheses of this article.
Further, the domain of integration is chosen to be the interval $[-4, 4]$ with $t\in[0, 0.5]$, and 
\begin{align}
    \label{eq:ex1} u^1_0(x)&=\sin^2(0.5\pi x)\mathbbm{1}_{(-2,2)}(x),\\ 
    u^2_0(x)&=\cos^2(0.25\pi x)\mathbbm{1}_{(-2,2)}(x),\,\,x\in\R. \notag
\end{align} 

Figure \ref{fig:ex212} displays the numerical approximations of \eqref{PDE_NL_k}, \eqref{eq:ex1} 
generated by the numerical scheme \eqref{apx}--\eqref{MF},
with $\Delta x=0.00078125$ and $\eta=80\Delta x=0.0625$. It can be seen that the numerical scheme is able to capture both shocks and rarefactions well, and that there is the expected change of lanes to the faster lane.
\begin{figure}[ht!]
 \centering
 \begin{subfigure}{.45\textwidth}
\includegraphics[width=\textwidth,keepaspectratio]{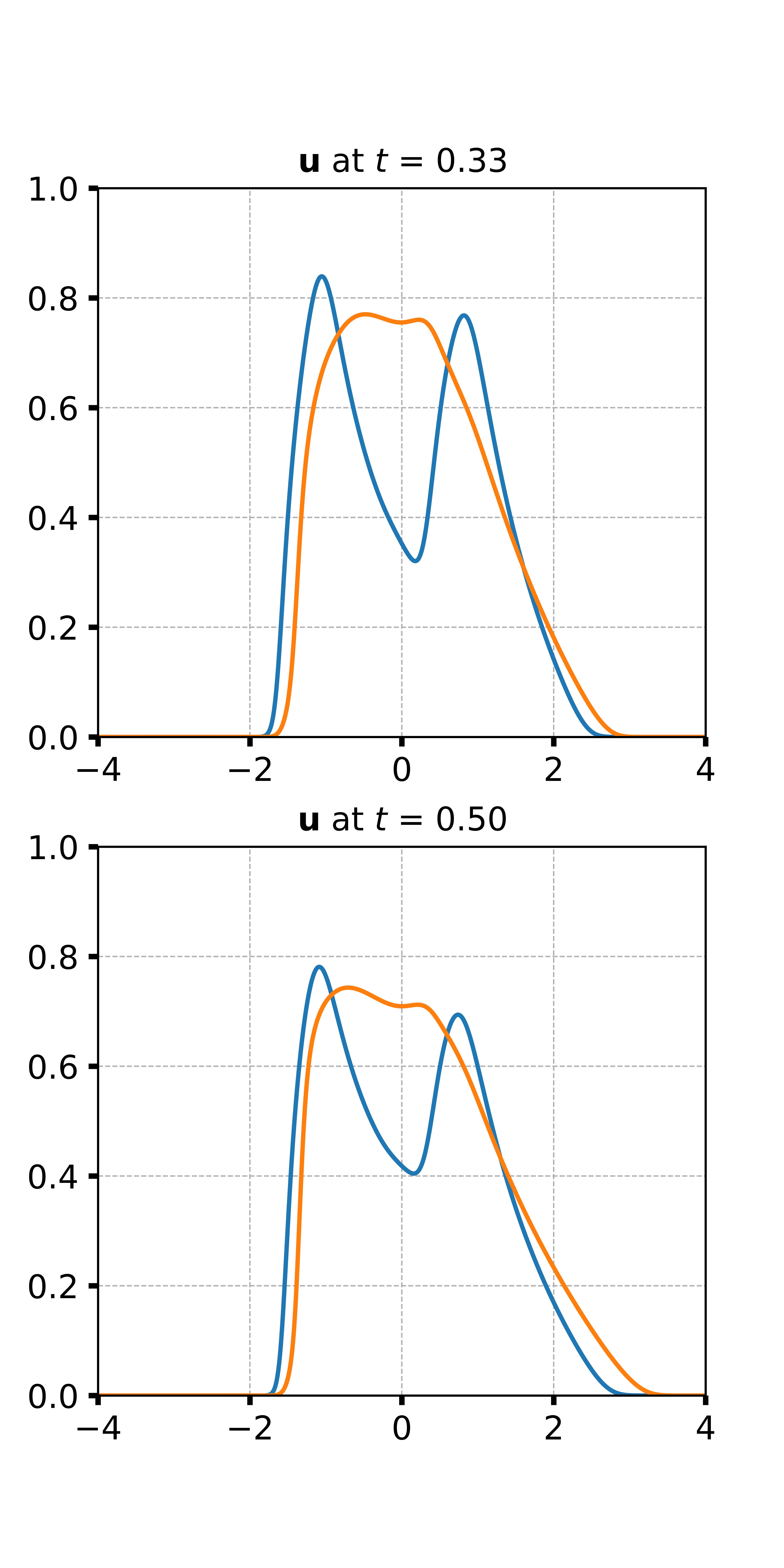}
\end{subfigure}
\begin{subfigure}{.45\textwidth}
\includegraphics[width=\textwidth,keepaspectratio]{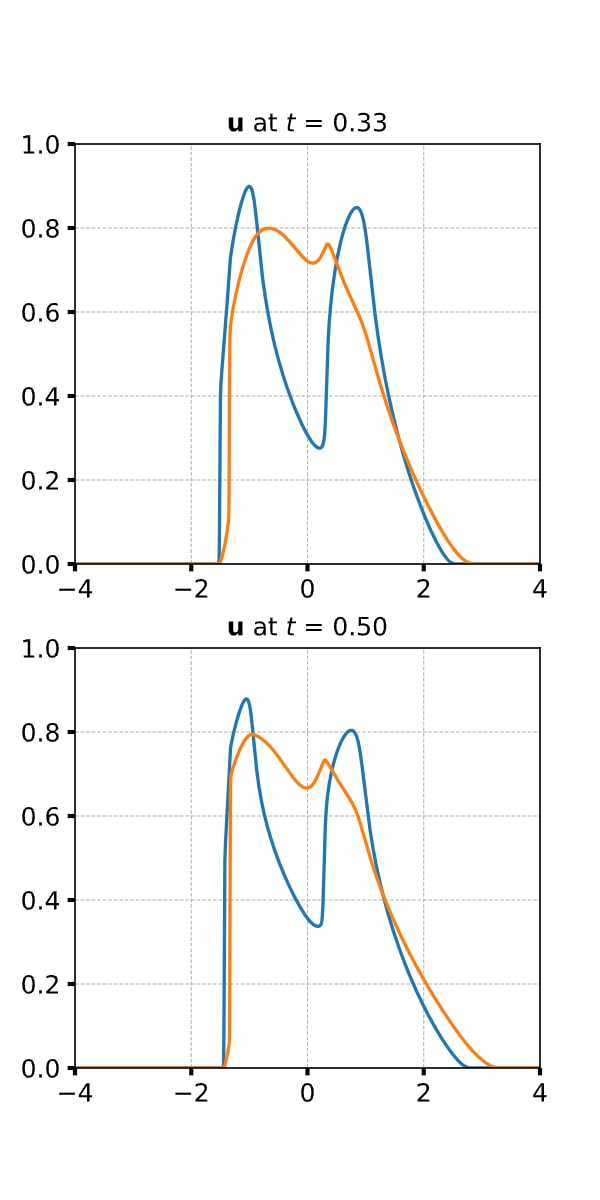}
\end{subfigure}
\hfill
\caption{Domain $[-4,\,4],\; t=0.33,\;0.5,\;\Delta x=0.00625$: (Left) Solution to the nonlocal conservation law~\eqref{PDE_NL_k}, \eqref{eq:ex1} with $\eta=100\Delta x$; (Right) Solution to the local balance law\cite[(2.2a), (2.2b)]{HR2019}, \eqref{eq:ex1}{\color{blue},} $u^1$({\color{blue}\full}),\,\,$u^2$({\color{orange}\full}).}
  \label{fig:ex2121}
\end{figure} 
{\color{black}It can be seen in Figure \ref{fig:ex2121} that due to the presence of the nonlocal terms with forward kernel, the blue slower lane observes a lower density of the faster lane and shifts to the faster lane at a relatively higher rate as compared to the local model. Further, in $[-1,1],$ the blue lane sees a higher average density and decides to shift {\color{black}at} a lower rate, thus increasing its density, while the faster red lane noticing a lower density shifts to slower lane. Similar behavior is observed for $x\ge 1.$} 

Figure \ref{fig:ex21} displays 
 \begin{figure}[ht!]
 \centering
\begin{subfigure}{.45\textwidth}
\includegraphics[width=\textwidth,keepaspectratio]{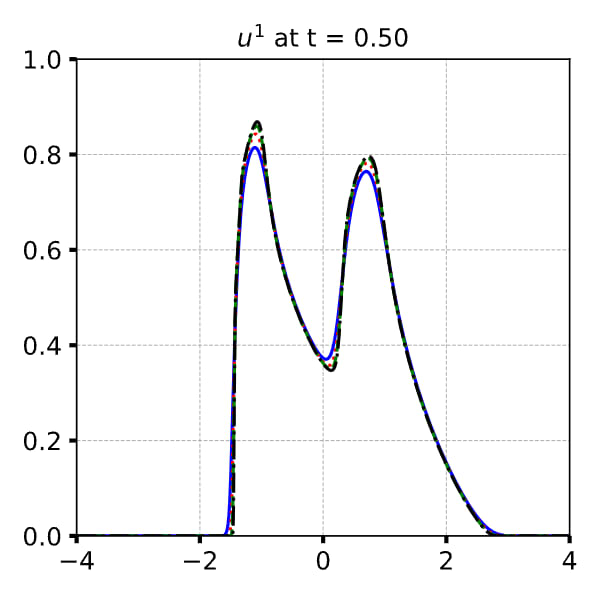}
\end{subfigure}
\begin{subfigure}{.45\textwidth}
\includegraphics[width=\textwidth,keepaspectratio]{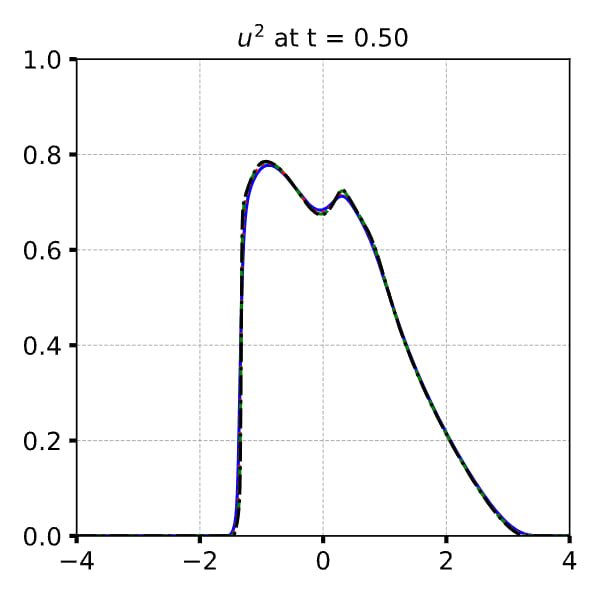}
\end{subfigure}
\caption{Solution to the nonlocal conservation law~\eqref{PDE_NL_k}, \eqref{eq:ex1} on the domain $[-4,\,4]$ at time $t = 0.5$, with decreasing mesh size  $\Delta x=0.00625$({\color{blue}\full}), $\Delta x=0.003125 $({\color{magenta}\dotted}), $\Delta x=0.0015625$({\color{black}\dashed}) and $\Delta x=0.00078125$({\color{darkgreen}\chainn}).}
  \label{fig:ex21}
\end{figure} \begin{figure}[ht!]  
  \centering
\noindent\begin{minipage}{0.45\textwidth}
    \centering
    \begin{tabular}{|c|c|c|c|c|c|c|c|c|c|}\hline
     \multicolumn{1}{|c|}{ $\displaystyle\frac{\D x}{0.00625}$}&\multicolumn{1}{|c|}{$\displaystyle\frac{e_{\Delta x}(T)}{100}$}\vline & \multicolumn{1}{|c|}{$\alpha$}\vline\\
     \hline
     $1$&$6.095$&$1.9386$\tabularnewline
     \hline
     $1/2$&$1.590$&$ 1.9664$\tabularnewline
     \hline
     $1/4$&$0.0407$&$1.9804$\tabularnewline
     \hline
     $1/8$&$0.0103$&$1.9862$\tabularnewline
     
     \hline
     $1/16$&$0.0026$&\tabularnewline
     
     \hline
 \end{tabular}
  \end{minipage}
  \noindent\begin{minipage}{0.45\textwidth}
\includegraphics[width=.9\textwidth, trim = 40 25 20 5]{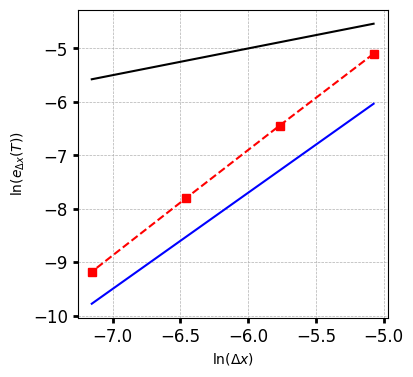}
  \end{minipage}
  \caption{Convergence rate $\alpha$ for the numerical scheme~\eqref{apx}--\eqref{MF}  for the approximate solutions
  to the problem~\eqref{PDE_NL_k}, \eqref{eq:ex1} on the domain $[-4,\,4]$ at time $T=0.5$. {\color{black}Optimal rate of convergence} 0.5({\color{black}\full}), Observed Convergence Rate ({\color{blue}\ch}), Reference Slope 1.9 ({\color{blue}\full}).}\label{fig:my_label211}\label{fig:table}
\end{figure}\begin{figure}[ht!]
 \centering
\begin{subfigure}{.45\textwidth}
\includegraphics[width=\textwidth,keepaspectratio]{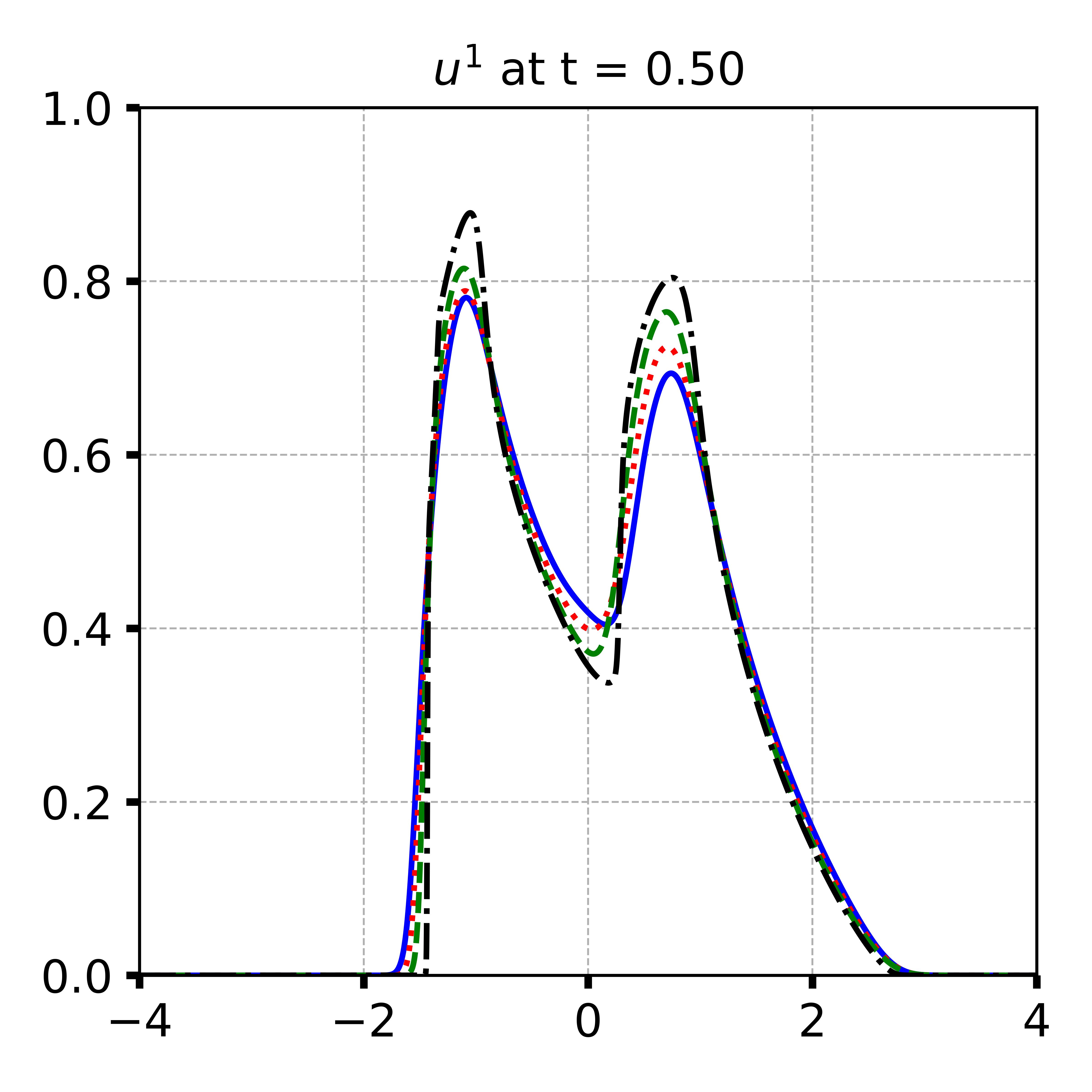}
\end{subfigure}
\hfill\begin{subfigure}{.45\textwidth}
\includegraphics[width=\textwidth,keepaspectratio]{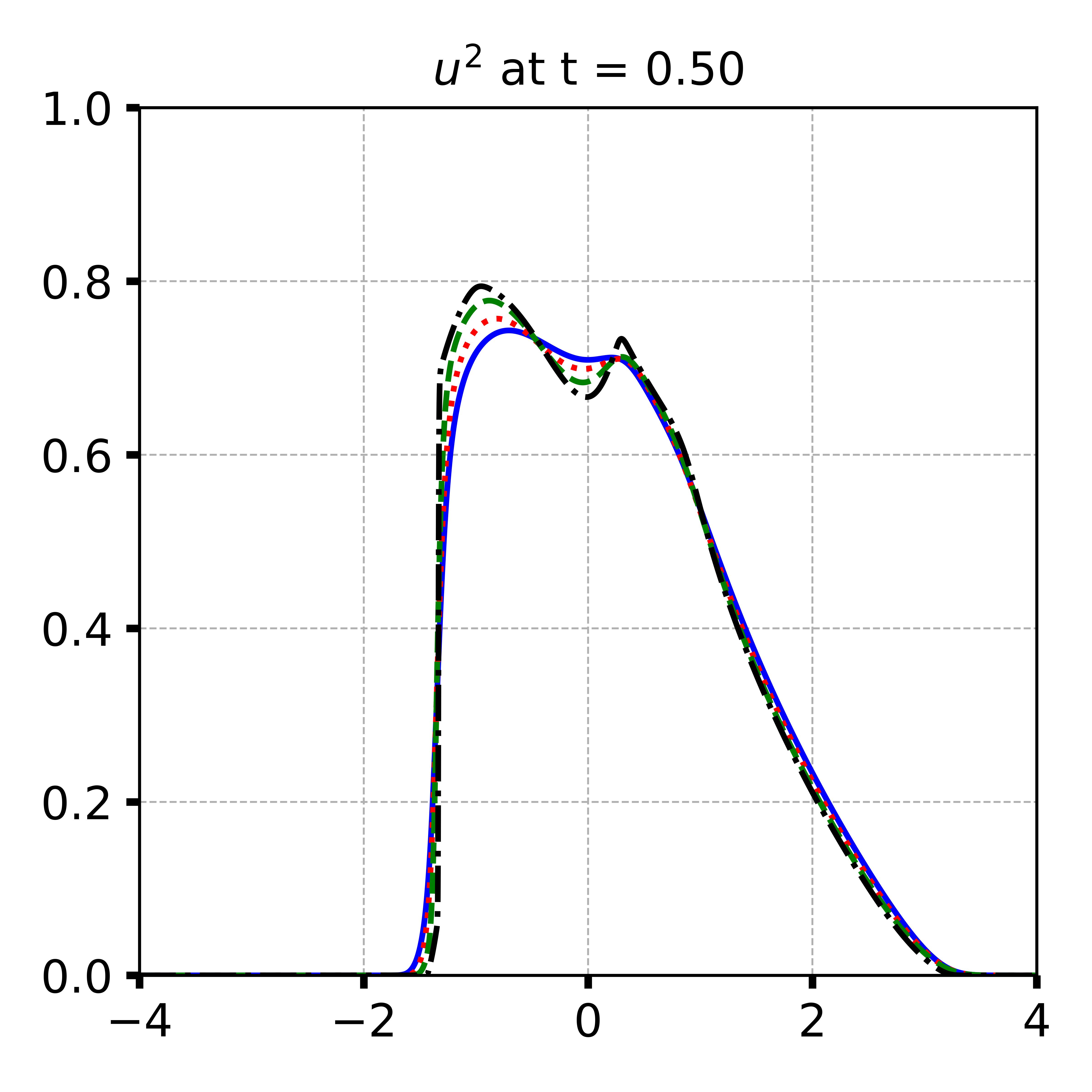}
\end{subfigure}
\caption{Domain $[-4,\,4],   T=0.5,\Delta x =0.00625$: Solution to the local conservation law~\eqref{PDE_NL_k}, \eqref{eq:ex1} with decreasing convolution radii   $100\Delta x$({\color{blue}\full}), $50\Delta x $({\color{magenta}\dotted}), $10\Delta x$({\color{green}\dashed}); Solution to the local balance law\cite[(2.2a), (2.2b)]{HR2019}, \eqref{eq:ex1} ({\color{black}\chainn}).}
  \label{fig:ex211}
\end{figure}the numerical approximations of \eqref{PDE_NL_k}, \eqref{eq:ex1}  by the numerical scheme \eqref{apx}--\eqref{MF}, with a decrease in grid size $\Delta x$, starting with $\Delta x =0.00625$. Next, we compute the observed
convergence rate $\alpha$ of the scheme at time $T=0.5$ by computing the $L^1$ distance
between the numerical solutions ${\boldsymbol{u}}_{\Delta x}(T,\dott)$ and $ {\boldsymbol{u}}_{\Delta x/2}(T,\dott)$ obtained
for the grid size $\Delta x$ and $\Delta x/2$, for each grid size $\Delta x$. Let $e_{\Delta x}(T)=\sum_{k=1}^{2}\norma{u^k_{\Delta x}(T,\dott)-u^k_{\Delta x/2}(T,\dott)}_{L^1(\R)}$. 
The observed convergence rate $\alpha$ is given by $\log_{2}(e_{\Delta x}(T)/{e_{\Delta x/2}(T)}).$ The results recorded in Figure \ref{fig:my_label211}  
show
 that $\alpha>0.5$. The present numerical integration resonates well with the theoretical convergence rate obtained in Theorem \ref{CR} in this article.

Figure \ref{fig:ex211} illustrates
the nonlocal to local limit, and suggests that the entropy solutions of the nonlocal conservation laws seem to converge to the entropy solution of the corresponding local conservation law as the radius of the kernel goes to zero, starting with $\eta=100\Delta x$, {\color{black} which could very well be due to the numerical viscosity present in the scheme, see \cite{CCMS2024, colombo2019singular}. 
\ack{A part of this work was carried out during the GV's tenure of the ERCIM ‘Alain Bensoussan’ Fellowship Programme at NTNU.  The project was supported in part by, the project \textit{IMod --- Partial differential equations, statistics and data:
An interdisciplinary approach to data-based modelling}, project number 65114, from the Research Council of Norway, by AA's faculty development allowance 2023-24, funded by IIM Indore, and  the Swedish Research Council under grant no. 2021-06594 while HH was in residence at Institut Mittag-Leffler in Djursholm, Sweden during the fall semester of 2023. {\color{black} GV would like to thank Elio Marconi and Laura Spinolo for useful discussions on their co-authored article on `nonlocal-to-local' dynamics \cite{CCMS2023}.}
\section*{References}

\begin{thebibliography}{10}
\expandafter\ifx\csname url\endcsname\relax
  \def\url#1{{\tt #1}}\fi
\expandafter\ifx\csname urlprefix\endcsname\relax\def\urlprefix{URL }\fi
\providecommand{\eprint}[2][]{\url{#2}}

\bibitem{CG2019}
Chiarello F~A and Goatin P 2019 {\em Netw. Heterog. Media\/} {\bf 14} 371--387

\bibitem{GHS+2014}
G{\"o}ttlich S, Hoher S, Schindler P, Schleper V and Verl A 2014 {\em Appl.
  Math. Model.\/} {\bf 38} 3295--3313

\bibitem{KLS2018}
Keimer A, Leugering G and Sarkar T 2018 {\em J. Hyperbolic Differ. Eq.\/} {\bf
  15} 375--406

\bibitem{CMR2015}
Colombo R~M, Marcellini F and Rossi E 2015 {\em Netw. Heterog. Media\/} {\bf
  11} 49--67

\bibitem{BBKT2011}
Betancourt F, B{\"u}rger R, Karlsen K~H and Tory E~M 2011 {\em Nonlinearity\/}
  {\bf 24} 855--885

\bibitem{ANT2007}
Aletti G, Naldi G and Toscani G 2007 {\em SIAM J. Appl. Math.\/} {\bf 67}
  837--853

\bibitem{AS2012}
Amadori D and Shen W 2012 {\em J. Hyperbolic Differ. Equ.\/} {\bf 9} 105--131

\bibitem{CHM2011}
Colombo R~M, Herty M and Mercier M 2011 {\em ESAIM Math. Model. Numer. Anal.
  Control Optim. Calc. Var.\/} {\bf 17} 353--379

\bibitem{BHL2023}
B{\"u}rger R, Contreras H~D and Villada L~M 2023 {\em Netw. Heterog. Media\/}
  {\bf 18} 664--693

\bibitem{ACT2015}
Amorim P, Colombo R~M and Teixeira A 2015 {\em ESAIM Math. Model. Numer.
  Anal.\/} {\bf 49} 19--37

\bibitem{ACG2015}
Aggarwal A, Colombo R~M and Goatin P 2015 {\em SIAM J. Numer. Anal.\/} {\bf 53}
  963--983

\bibitem{AG2016}
Aggarwal A and Goatin P 2016 {\em Bull. Braz. Math. Soc. (N.S.)\/} {\bf 47}
  37--50

\bibitem{BG2016}
Blandin S and Goatin P 2016 {\em Numer. Math.\/} {\bf 132} 217--241

\bibitem{FGKP2022}
Friedrich J, G{\"o}ttlich S, Keimer A and Pflug L 2024 {\em SIAM J. Appl.
  Math.\/} {\bf 84} 497--522

\bibitem{CGL2012}
Colombo R~M, Garavello M and L{\'e}cureux-Mercier M 2012 {\em Math. Mod. Met.
  Appl. Sci.\/} {\bf 22} 1150023--34

\bibitem{CL2011}
Colombo R~M and L{\'e}cureux-Mercier M 2011 {\em Acta Math. Sci. Ser. B\/} {\bf
  32} 177--196

\bibitem{AHV2023}
Aggarwal A, Holden H and Vaidya G 2024 {\em Num. Math.\/} {\bf 156} 237--271

\bibitem{AV2023}
Aggarwal A and Vaidya G 2025 {\em Math. Comp.\/} {\bf 94} 585--610

\bibitem{FCV2023}
Chiarello F~A, Contreras H~D and Villada L~M 2023 {\em J. Engrg. Math.\/} {\bf
  141} 9

\bibitem{KP2021}
Keimer A and Pflug L 2023 {\em C. R. Math.\/} {\bf 361} 1723--1760

\bibitem{HR2019}
Holden H and Risebro N~H 2019 {\em SIAM J. Math. Anal.\/} {\bf 51} 3694--3713

\bibitem{chiarello2019stability}
Chiarello F~A, Goatin P and Rossi E 2019 {\em Nonlinear Anal. Real World
  Appl.\/} {\bf 45} 668--687

\bibitem{FGR2021}
Friedrich J, G{\"o}ttlich S and Rossi E 2021 {\em Commun. Math. Sci.\/} {\bf
  19} 2291--2317

\bibitem{BFK2022}
Bayen A, Friedrich J, Keimer A, Pflug L and Veeravalli T 2022 {\em SIAM J.
  Appl. Dyn. Syst.\/} {\bf 21} 1495--1538

\bibitem{CK24}
Chiarello F~A and Keimer A 2024 {\em J. Math. Anal. Appl.\/} {\bf 537} 128358

\bibitem{CNAP2022}
Coclite G~M, De~Nitti N, Keimer A and Pflug L 2022 {\em Z. Angew. Math.
  Phys.\/} {\bf 73} 241

\bibitem{CGL2024}
Colombo M, Crippa G and Spinolo L~V 2024 {\em arXiv:2408.02423\/}

\bibitem{KUZ1976}
Kuznetsov N~N 1976 {\em USSR Comput. Math. Math. Phys.\/} {\bf 16}(6) 105--119
  ISSN 00415553

\bibitem{AHV2023_1}
Aggarwal A, Holden H and Vaidya G 2024 {\em IMA J. Numer. Anal.\/} {\bf 44}
  3354--3392 ISSN 0272-4979

\bibitem{CCV2022}
Chiarello F~A, Contreras H~D and Villada L~M 2022 {\em Netw. Heterog. Media\/}
  {\bf 17} 203--226

\bibitem{Coc23a}
Coclite G~M, Coron J~M, Nitti N~D, Keimer A and Pflug L 2023 {\em Ann. Inst. H.
  Poincaré C Anal. Non Linéaire\/} {\bf 40} 1205--1223

\bibitem{Col23a}
Colombo M, Crippa G, Marconi E and Spinolo L~V 2023 {\em Journ. Équ. dériv.
  partielles\/} Talk:10. Réseau thématique AEDP du CNRS, pp. 1--14

\bibitem{CCCNKMKL2023}
Coclite G~M, Colombo M, Crippa G, De~Nitti N, Keimer A, Marconi E, Pflug L and
  Spinolo L~V 2024 {\em J. Hyperbolic Differ. Equ.\/} {\bf 21} 681--705

\bibitem{CCMS2024}
Colombo M, Crippa G, Graff M and Spinolo L~V 2021 {\em ESAIM Math. Model.
  Numer. Anal.\/} {\bf 55} 2705--2723

\bibitem{CCS2019}
Colombo M, Crippa G and Spinolo L~V 2019 {\em Arch. Ration. Mech. Anal.\/} {\bf
  233} 1131--1167

\bibitem{CCMS2023}
Colombo M, Crippa G, Marconi E and Spinolo L~V 2023 {\em Arch. Ration. Mech.
  Anal.\/} {\bf 247} 18

\bibitem{CGES2021}
Colombo M, Crippa G, Marconi E and Spinolo L~V 2021 {\em Ann. Inst. H.
  Poincar{\'e} C Anal. Non Lin{\'e}aire\/} {\bf 38} 1653--1666

\bibitem{CMR2016}
Colombo R~M, Marcellini F and Rossi E 2016 {\em Netw. Heterog. Media\/} {\bf
  11} 49--67

\bibitem{GKLW2016}
Gugat M, Keimer A, Leugering G and Wang Z 2016 {\em Netw. Heterog. Media\/}
  {\bf 10} 749--785

\bibitem{CM2015}
Colombo R~M and Marcellini F 2015 {\em J. Differ. Equ.\/} {\bf 259} 6749--6773

\bibitem{HR2015}
Holden H and Risebro N~H 2015 {\em Front {T}racking for {H}yperbolic
  {C}onservation {L}aws\/} 2nd ed (Springer)

\bibitem{Sab1997}
Sabac F 1997 {\em SIAM J. Numer. Anal.\/} {\bf 34} 2306--2318

\bibitem{Kruzkov}
Kru{\v{z}}hkov S 1970 {\em Mat. Sb. (N.S.)\/} {\bf 81 (123)} 228--255

\bibitem{GTV2022}
Ghoshal S~S, Towers J~D and Vaidya G 2022 {\em Numer. Math.\/} {\bf 151}
  601--625

\bibitem{BP1998}
Bouchut F and Perthame B 1998 {\em Trans. Amer. Math. Soc\/} {\bf 350}
  2847--2870

\bibitem{FKG2018}
Friedrich J, Kolb O and Göttlich S 2018 {\em Netw. Heterog. Media\/} {\bf 13}
  531--547

\bibitem{keimer2017existence}
Keimer A and Pflug L 2017 {\em J. Differ. Equ.\/} {\bf 263} 4023--4069

\bibitem{HKLR2010}
Holden H, Karlsen K~H, Lie K~A and Risebro N~H 2010 {\em Splitting methods for
  partial differential equations with rough solutions: Analysis and MATLAB
  programs\/} (EMS Publishing House, Zürich)

\bibitem{colombo2019singular}
Colombo M, Crippa G and Spinolo L~V 2019 {\em Arch. Ration. Mech. Anal.\/} {\bf
  233} 1131--1167

\end{thebibliography}
\def\cprime{$'$}
\providecommand{\newblock}{}

\end{document}